%% file: 2main.tex
\documentclass[abstract=true, DIV=12]{scrartcl}

\usepackage{preamble}

\newcommand{\IconjHC}{Conjecture~B}

\newcommand{\Iseccat}{§2.1}
\newcommand{\Isecdga}{§2.2}
\newcommand{\Iobscofibcdgas}{Observation~2.2.7}
\newcommand{\Ilemmacofibrantflat}{Lemma~2.2.11}
\newcommand{\Ilemmabarcomplex}{Lemma~2.2.14}
\newcommand{\Ilemmalocalization}{Lemma~2.3.3}

\newcommand{\IsecTHH}{§3.2}
\newcommand{\Ilemmafibassembly}{Lemma~3.4.5}

\newcommand{\Iseceq}{§4.2}
\newcommand{\Iobspullbackeq}{Observation~4.2.3}
\newcommand{\IpropBMeq}{Proposition~4.2.7}
\newcommand{\Ilemmaloopssuspadjmodel}{Lemma~4.2.12}
\newcommand{\Ilemmarestrictcofibration}{Lemma~4.3.7}
\newcommand{\Ilemmaassemblymodel}{Lemma~4.5.4}
\newcommand{\Ilemmaevalmodel}{Lemma~4.5.5}
\newcommand{\IconHHtransfer}{Construction~4.6.1}
\newcommand{\IlemmaHomqiso}{Lemma~4.6.2}
\newcommand{\Ithmmain}{Theorem~4.6.3}

\title{A rational model for the fiberwise~THH~transfer~II: \texorpdfstring{$\Ainf$}{A∞}-algebras}
\author{Florian Naef \and Robin Stoll}
\date{April 27, 2026}

\begin{document}

\maketitle

\begin{abstract}
  In Part~I, we proved that a rational model for the fiberwise THH transfer of a map $f$ of fibrations over a base space is given by the Hochschild homology transfer of a cdga model of $f$.
  In this paper, we provide an explicit description of this Hochschild homology transfer in terms of $\Ainf$-algebras, generalizing work of Bouc.
  Using a result of Lind--Malkiewich, we deduce a rational model for the Becker--Gottlieb transfer.
  We furthermore use our results for the following applications to manifold topology.
  
  Firstly, we consider the rational characteristic classes constructed by Berglund for fibrations with fiber a Poincaré complex (which generalize classes found by Berglund--Madsen); they are defined via the Lie graph complex, and we prove that the classes corresponding to non-trivalent graphs with exactly one loop vanish when evaluated on fiber bundles with fiber a compact simply connected topological manifold.
  
  Secondly, we provide a rational model for the space of fiberwise THH-simple structures, which is a step towards obtaining rational models for the classifying spaces of diffeomorphisms and homeomorphisms of a compact simply connected manifold in the rational concordance stable range.
\end{abstract}

\setcounter{tocdepth}{\subsectiontocdepth}
\tableofcontents

\input{2introduction.tex}

\include{2preliminaries}
\include{2A_infty}
\include{2transfer}
\include{2model}
\include{2applications}

\section*{References}

\printbibliography[keyword=this,heading=subbibliography,title={This series}]
\printbibliography[notkeyword=this,heading=subbibliography,title={Other}]

\end{document}

%% file: 2introduction.tex
\section{Introduction}

The topological Hochschild homology (THH) of a space $X$ is, by definition, the topological Hochschild homology of the full subcategory of compact objects in the stable $\infty$-category of parametrized spectra over $X$ (i.e.\ functors from $X$ considered as an $\infty$-groupoid to spectra).
Pulling back parametrized spectra along a fibration $f \colon X \to Y$ whose fibers are equivalent to finite CW-complexes (or, more generally, are finitely dominated) preserves compact objects, and hence induces a map
\[ f^* \colon \THH(Y) \longto \THH(X) \]
called the \emph{THH transfer}.
Observing that THH of a space $X$ is equivalent to the suspension spectrum $\SS[\L X]$ of its free loop space, this is also called the \emph{free loop transfer}.
It turns out to be related to many other constructions; for example, Lind--Malkiewich \cite{LM} proved that one can recover the Becker--Gottlieb transfer $\SS[Y] \to \SS[X]$ from the THH transfer.
We will return to this point below.

Given a further fibration $Y \to B$, we can form the THH transfer of $f$ fiberwise over $B$ to obtain the \emph{fiberwise THH transfer}: a map $f^* \colon \THH_B(Y) \to \THH_B(X)$ of parametrized spectra over $B$; here $\THH_B(E)$ denotes fiberwise THH of a space $E$ over $B$.
In the special case that $Y = B$, we obtain a map
\[ \chi  \colon  \SS_B  \eq  \THH_B(B)  \longto  \THH_B(X) \]
from the parametrized sphere spectrum; this is called the \emph{fiberwise THH Euler characteristic} of $X$ over $B$.
Via the Dennis trace, it is related to the fiberwise \emph{A-theory} Euler characteristic (i.e.\ the analogous construction for algebraic K-theory instead of THH), which plays a central role in the topological Riemann--Roch theorem of Dwyer--Weiss--Williams \cite{DWW}.
Following a suggestion of Manuel Krannich, this relationship in fact served as one of our main motivations, and we will heavily exploit it in our applications below.
Lastly let us mention that, by work of Naef--Safronov \cite{NS}, the THH Euler characteristic measures the failure of homotopy invariance of the Goresky--Hingston coproduct in string topology.

\subsection{Main result}

In \cite{I}, we proved that the fiberwise THH transfer of a map $f$ is rationally modeled by the classical Hochschild homology transfer of a cdga model of $f$.
By definition (see below), the Hochschild homology transfer $\HH(S) \to \HH(R)$ of a map of cdgas $R \to S$ is a zig-zag involving quasi-isomorphisms pointing the opposite way.
The main result of this paper is that it can be strictified to a direct map by passing to $\Ainf$-algebras.
(Recall that an $\Ainf$-algebra, a notion originally introduced by Stasheff \cite{Sta}, is an object with a multiplication $\mu_2$ that is homotopy associative up to specified higher coherences $\mu_m$.)
To this end, we first prove that the Hochschild homology transfer of a map $R \to S$ of cdgas is homotopic to the Hochschild homology transfer associated to any unital $\Ainf$-algebra over $R$ quasi-isomorphic to $S$; see \cref{sec:transfer_natural}.
We then prove the following result, which provides an explicit formula for this transfer when the $\Ainf$-algebra is dualizable as an $R$-module.

\begin{introtheorem}[see \cref{thm:Ainf_transfer}] \label{intro:thm:main_general}
  Let $\k \to R$ be a map of cdgas over $\QQ$, and $S$ an $\Ainf$-algebra over $R$.
  If $R$ is cofibrant as a $\k$-module and $S$ is cofibrant and dualizable as an $R$-module, then the Hochschild homology transfer $\HH_\k(S) \to \HH_\k(R)$ agrees, in the homotopy category of $\k$-modules, with the composite map
  \[ \HH_\k(S)  \xlongto{\upsilon_*}  \HH_\k \bigl( \End_R(S) \bigr)  \xlongto{\tr^c_R}  \HH_\k(R) \]
  where $\End_R(S)$ is the $R$-algebra of endomorphisms of the $R$-module $S$, the $\Ainf$-algebra map $\upsilon \colon S \to \End_R(S)$ corresponds to the action of $S$ on itself, and $\tr^c_R$ is the \emph{generalized trace}, an explicit map defined in terms of a derived coevaluation $c \colon \k \to S \dertensor_R \Hom(S, R)$ (see \cref{def:gtr}).
\end{introtheorem}

Here $\HH_\k(S)$ denotes the Hochschild complex of the $\Ainf$-algebra $S$ over $\k$, a construction originally due to Getzler--Jones \cite{GJ}.
In fact we prove a version of \cref{intro:thm:main_general} for the Hochschild homology transfer $\transfer{M} \colon \HH_\k(S) \to \HH_\k(R)$ induced by an $S$-module $M$ that is dualizable over $R$; the version we stated above is obtained by choosing $M$ to be $S$ itself.
This is a generalization of a formula for $\transfer{M}$ due to Bouc \cite{Bou} (also see Loday \cite[§1.2]{Lod}), which applies in the simpler case that $S$ and $R$ are ordinary algebras (i.e.\ they have no differential).
In particular our formula is new even just for differential graded algebras, and the authors think it significantly clarifies the work of Bouc.
The advantage of passing to $\Ainf$-algebras is that, by the Homotopy Transfer Theorem, they can much more often be chosen to be dualizable as modules than differential graded algebras.
In fact it follows from work of Berglund \cite{Ber} that this is always the case in the situation relevant to modeling the fiberwise THH transfer (see \cref{sec:eq_models_exist}).
Lastly let us mention that our formula simplifies considerably when $R = \k$, see further below.

\paragraph{The Hochschild homology transfer.}
We briefly recall its definition (see e.g.\ Keller \cite[§5]{Kel21}).
Given a map $\phi \colon R \to S$ of $\k$-algebras, an $S$-module $M$ induces a transfer map $\transfer{M}$ on Hochschild homology when it is perfect (i.e.\ derived dualizable) as an $R$-module.
In that case, there is a canonical derived coevaluation map $\coev \colon S \to M \dertensor_R \dual{M}_R$ in the homotopy category of $S$-bimodules, where $\dual{M}_R$ denotes the derived $R$-dual of $M$.
The transfer $\transfer{M}$ is then defined to be the composite
\[ \HH_\k(S)  \xlongto{\coev_*}  \HH_\k(S, M \dertensor_R \dual{M}_R)  \eq  \HH_\k(R, \dual{M}_R \dertensor_S M)  \xlongto{\ev_*}  \HH_\k(R) \]
where the middle equivalence can be visualized using the picture
\[
\begin{tikzcd}[column sep = 10]
  M \vphantom{\dual{M}_R} \rar[start anchor = north, end anchor = north, bend left = 90, dash]{\dertensor_S} & \lar[start anchor = south, end anchor = south, bend left = 90, dash]{\dertensor_R} \dual{M}_R
\end{tikzcd}
\qquad = \qquad
\begin{tikzcd}[column sep = 10]
  \dual{M}_R \rar[start anchor = north, end anchor = north, bend left = 90, dash]{\dertensor_R} & \lar[start anchor = south, end anchor = south, bend left = 90, dash]{\dertensor_S} M \vphantom{\dual{M}_R}
\end{tikzcd}
\]
and we recall that the Hochschild complex of an $S$-bimodule is obtained by derived tensoring the left and right $S$-module structures together.

The derived coevaluation $\coev \colon S \to M \dertensor_R \dual{M}_R$ can usually not be represented by an actual direct map of $S$-bimodules.
However, it is always possible to choose a direct map representing its restriction along $\k \to S$, which is the input required by \cref{intro:thm:main_general}.

\subsubsection{A rational model for the fiberwise THH transfer}

Combining \cref{intro:thm:main_general} with the main result of \cite{I}, we obtain an explicit rational model for the fiberwise THH transfer in terms of $\Ainf$-algebras.
Moreover, as mentioned above, the formula of \cref{intro:thm:main_general} simplifies significantly when $R$ is the base cdga $\k$.
This has the following consequence for the fiberwise THH Euler characteristic.

\begin{introcorollary}[see \cref{cor:main}] \label{intro:cor:main}
  Given a fibration $p \colon E \to B$ between nilpotent spaces of finite rational type such that the fiber of $p$ is simply connected and finitely dominated, let $\k$ be a cdga that models $B$ and $R$ a unital $\Ainf$-algebra in $\k$-modules that models $E$.
  If $R$ is cofibrant and dualizable as a $\k$-module, then the fiberwise THH Euler characteristic $\chi \colon \SS_B \to \THH_B(E)$ is modeled by the map of $\k$-modules $\tr \colon \HH_\k(R) \to \k$ given by
  \[ x_0 \otimes x_1 \otimes \dots \otimes x_n  \longmapsto  \sum_{i = 0}^n \pm \tr_\k \bigl( \mu_{n+2}(x_i, x_{i+1}, \dots, x_n, x_0, \dots, x_{i - 1}, \blank) \bigr) \]
  where $\mu_m \colon R^{\tensor_\k m} \to R$ denotes the structure maps of the $\Ainf$-algebra $R$, and $\tr_\k(\alpha)$ the $\k$-linear trace of an endomorphism $\alpha \colon R \to R$.
\end{introcorollary}

Recall that by work of Sullivan \cite{Sul}, there is a contravariant equivalence between the rational homotopy category of nilpotent spaces of finite rational type, and the homotopy category of connected cdgas over $\QQ$ of finite type.
A cdga $\k$ is a \emph{(Sullivan) model} for a space $B$ if they correspond to each other under this equivalence.
In this situation, Braunack--Mayer \cite{Bra} furthermore showed that the rational homotopy category of nilpotent, bounded below parametrized spectra over $B$ of finite rational type is equivalent to the homotopy category of $\k$-modules of finite type (previous results in this direction were obtained by Félix--Murillo--Tanré \cite{FMT}).
This is the sense in which the map of $\k$-modules $\HH_\k(R) \to \k$ models the fiberwise THH transfer.
That $R$ models $E$ means that it is quasi-isomorphic, as a unital $\Ainf$-algebra in $\k$-modules, to a Sullivan model $\k \to R$ of $p$.

Combining \cref{intro:cor:main} with an explicit construction, due to Berglund \cite{Ber}, of $\Ainf$-models as required, we also obtain a model for the fiberwise THH Euler characteristic in terms of a dg Lie model of the fibration, see \cref{sec:Lie_models}.
The resulting description is related to maps defined by Alekseev--Torossian \cite{AT} and Enomoto--Satoh \cite{ES}, see \cref{rem:ES}.

\paragraph{Equivariant models.}
The condition of \cref{intro:cor:main} that the spaces involved are nilpotent can be inconvenient; for example it is usually not fulfilled in the case of the universal fibration $F / \aut(F) \to \B \aut(F)$ with fiber $F$.
However, if $E \to B$ is a fibration of connected spaces with simply connected fiber, one can always find a (discrete) group $G$ and a map $B \to \B G$ such that the (homotopy) fibers of $E$ and $B$ over $\B G$ are nilpotent; for example one can choose $G$ to be $\pi_1(B)$.
In that case the map $E \to B$ can be modeled by a map of cdgas with an action of $G$; see \cite{BS,GHT} for elaborations on this idea.
Following this approach, we proved the main result of \cite{I}, that the fiberwise THH transfer is modeled by the Hochschild homology transfer, in the $G$-equivariant setting.
\Cref{intro:thm:main_general} of this paper also lifts to an equivariant statement, and we thus obtain an equivariant version of \cref{intro:cor:main} as well.

It follows from forthcoming work of Berglund \cite{Ber} (which builds on work of Berglund--Zeman \cite{BZ}) that any map $E \to B$ with simply connected finitely dominated fibers admits an equivariant $\Ainf$-model as required by \cref{intro:cor:main} (as long as $B$ is connected and its universal covering is of finite rational type; see \cref{sec:eq_models_exist}).
In the case of the universal fibration $F / \aut(F) \to \B \aut(F)$, this model is completely explicit and amenable to computations.
Combining this with work of Berglund--Stoll \cite{BS24}, one obtains similarly explicit equivariant models for the classifying space of block diffeomorphisms of a manifold of dimension $\ge 6$ (with spherical boundary).

\paragraph{Further directions.}
By definition, the formula of \cref{intro:cor:main} is invariant under cyclic permutations of the input.
In particular, it uniquely factors through a map $\tr \colon \HC_\k(R) \to \k$ from the Connes complex of $R$ over $\k$ (which computes its cyclic homology).
It seems likely that this is homotopic to the cyclic homology transfer $\HC_\k(R) \to \HC_\k(\k)$ of Keller \cite{Kel98}, composed with the projection $\HC_\k(\k) \to \k$.
Combined with \cite[\IconjHC{}]{I}, this would yield the following explicit rational model for the fiberwise transfer of rational negative cyclic homology $\HCm(\blank; \QQ)$, i.e.\ the homotopy fixed points of the $\Sphere 1$-action on $\THH(\blank) \tensor \H\QQ$.

\begin{introconjecture} \label{intro:conj:HC}
  In the situation of \cref{intro:cor:main}, the fiberwise Euler characteristic of rational negative cyclic homology $\SS_B \to \HCm_B(E; \QQ)$ is modeled by the map $\tr \colon \HC_\k(R) \to \k$.
\end{introconjecture}

By work of Goodwillie \cite{Goo}, rational negative cyclic homology is closely related to algebraic K-theory.
More precisely, his result implies that the following map, induced by the Dennis trace,
\[ \fib \bigl( \A(X) \to \A(*) \bigr)  \longto  \fib \bigl( \HCm(X; \QQ) \to \HCm(*; \QQ) \bigr) \]
is a rational equivalence when $X$ is nilpotent.
In particular \cref{intro:conj:HC} would provide a rational model for a close approximation of the fiberwise A-theory Euler characteristic.
See \cref{intro:sec:THH-simple} below for more background on this.

In a different direction, it should be straightforward to generalize our methods to obtain a rational model for the fiberwise Reidemeister trace of a fiberwise self-map $g$, i.e.\ the analog of the THH Euler characteristic for THH twisted by $g$.
This is an obstruction for $g$ to be fiberwise homotopic to a map without fixed points; in fact Klein--Williams \cite{KW} showed that it is a complete obstruction under certain assumptions on the base and fiber (also see Ponto \cite{Pon}).
Using such a generalization of our model, one could thus obtain explicit rational obstructions for removing fixed points of fiberwise maps.

\subsection{Applications}

Beyond their inherent appeal, our results are also useful for many applications.
We will now highlight and explain the ones we carry out in this paper; they all fall within the realms of algebraic topology and rely on \cref{intro:cor:main}.
However, we would like to point out that the Hochschild homology transfer has applications in other areas as well; for example it is used in work of Koenig--Liu--Zhou \cite{KLZ} concerning the Auslander--Reiten conjecture.
Thus our \cref{intro:thm:main_general} might be useful in those contexts as well.

\subsubsection{Graph characteristic classes of loop order one}

The \emph{Lie graph complex} with $n$ hairs $\mathrm{GC}^d(n)$ is an explicit cochain complex that is spanned by graphs with $n$ numbered univalent vertices (\emph{hairs}) and all other vertices labeled by the cyclic Lie co-operad; its differential is given by expanding a non-univalent vertex into two vertices connected by an edge (the parameter $d$ keeps track of certain degree shifts and signs).
In forthcoming work, Berglund \cite{Ber} constructs, for a fibration $E \to B$ with fiber an oriented simply connected Poincaré duality complex of dimension $d$, a natural map in the derived category of cochain complexes
\[ \kappa_d \colon \Schur {\mathrm{GC}^d} {\Cochains * (E; \QQ)}  \longto  \Cochains * (B; \QQ) \]
where the domain is the complex of Lie graphs with any number of hairs, which are additionally labeled by cochains of $E$.
More precisely, the expression $\Schur {\mathrm{GC}^d} {V}$ denotes the Schur functor of $\mathrm{GC}^d$ evaluated at $V$, i.e.\ the cochain complex $\bigoplus_{n \ge 0} \mathrm{GC}^d(n) \tensor_{\Sigma_n} V^{\tensor n}$ (the symmetric group $\Sigma_n$ acts on $\mathrm{GC}^d(n)$ by permuting the numbering of the hairs).
The map $\kappa_d$ in particular yields a rational characteristic class for fibrations of the required type for each homology class of $\mathrm{GC}^d(0)$; this recovers a construction of Matsuyuki \cite{Mat}, which in turn generalizes work of Kontsevich \cite{Kon} for the case $B = *$.
The map $\kappa_d$ moreover recovers classes found by Berglund--Madsen \cite{BM} for certain highly connected even-dimensional manifolds; more on this below.
(A generalization of Berglund's construction to integral cohomology is part of forthcoming work of Barkan--Steinebrunner \cite{BSt}; also see Bianchi \cite{Bia}.)

Adapting the ideas of Berglund, we construct, for a fibration $E \to B$ with fiber a simply connected finitely dominated space, a map in the derived category of cochain complexes
\[ \kappa_\wheel \colon \Schur {\mathrm{GC}^\wheel_1} {\Cochains * (E; \QQ)}  \longto  \Cochains * (B; \QQ) \]
where $\mathrm{GC}^\wheel_1$ is a complex of connected hairy Lie graphs that have exactly one loop and are \emph{directed}: every edge has a specified direction, and every vertex has exactly one outgoing edge.
(Where Berglund uses the modular operads of Getzler--Kapranov \cite{GK}, we use the wheeled operads of Markl--Merkulov--Shadrin \cite{MMS}; see \cref{sec:graph_classes}.)
Forgetting the directions yields (for any $d$) a map $\mathrm{GC}^\wheel_1 \to \mathrm{GC}^d_1$ to the one-loop part of the hairy Lie graph complex; it is surjective on homology and via this map our construction generalizes the one-loop part of Berglund's to finitely dominated spaces.
Following a suggestion of Manuel Krannich, we use our \cref{intro:cor:main} and a result of Dwyer--Weiss--Williams \cite{DWW} to prove the following.

\begin{introtheorem}[see \cref{cor:kappa_vanishing}] \label{intro:thm:graph_classes}
  Let $p \colon E \to B$ be a fiber bundle with fiber a compact simply connected topological manifold.
  Assume that $B$ is connected and that its universal covering is of finite rational type.
  Then the map induced on homology by $\kappa_\wheel$ vanishes on
  \[ \Schur {\Coho {> 0} (\mathrm{GC}^\wheel_1)} {\Coho * (E; \QQ)}  \subseteq  \Coho * \bigl( \Schur {\mathrm{GC}^\wheel_1} {\Cochains * (E; \QQ)} \bigr) \]
  where $\Coho {> 0}$ denotes homology classes of positive degree.
\end{introtheorem}

It is expected that, for a compact simply connected manifold $M$, the universal covering of $\B \Homeo_\bdry(M)$ (and in fact this classifying space itself) is always of finite type, so that \cref{intro:thm:graph_classes} would apply to the corresponding universal fiber bundle (though note that in general it does not require the boundary bundle to be trivial).
This is known when $M$ is a smooth manifold of even dimension $\ge 6$ by work of Bustamante--Krannich--Kupers \cite[Theorem~6.5]{BKK}, and when $M$ is a smooth $2$-connected manifold of dimension $\neq 4, 5, 7$ with spherical boundary by work of Kupers \cite[Corollary~5.15]{Kup}.

By the discussion before its statement, \cref{intro:thm:graph_classes} is equivalent to the analogous statement for $\mathrm{GC}^d_1$ and the map $\kappa_d$.
It applies to all ``interesting'' homology classes of $\mathrm{GC}^\wheel_1$ (and thus also of $\mathrm{GC}^d_1$).
In fact one can show that $\Coho 0 (\mathrm{GC}^\wheel_1(n)) \iso \QQ$ for all $n \ge 1$ (corresponding to trivalent graphs), and explicitly identify the higher homology groups as exterior powers of the standard representation of $\Sigma_n$, see \cref{prop:wheeled_Cinf_homology}.
(Forgetting the group action, this recovers a result of Markl--Merkulov--Shadrin \cite[Theorem~7.1.1]{MMS}; see also Conant--Kassabov--Vogtmann \cite[§6.3]{CKV} for the analogous result for $\mathrm{GC}^d_1$.)

Let us also remark that, by gluing graphs along their hairs, one can assemble many complicated classes of $\mathrm{GC}^d$ from simpler ones, see e.g.\ Conant--Hatcher--Kassabov--Vogtmann \cite{CHKV}.
(Another way to say this is that $\mathrm{GC}^d$ is a modular operad in the sense of Getzler--Kapranov \cite{GK}.)
Using this, our result implies that for fiber bundles the map $\kappa_d$ of Berglund vanishes on classes that can be assembled in a way that involves at least one class of $\mathrm{GC}^d_1$ of positive degree.
This applies for example to the Morita and the Eisenstein classes, which account for most known non-trivial classes with zero or one hairs (see \cite{CHKV}).

\paragraph{The classes of Berglund--Madsen.}
We now provide more context for \cref{intro:thm:graph_classes}.
In seminal work, Berglund--Madsen \cite{BM} proved that, for certain highly connected $2n$-dimensional manifolds denoted by $W_{g,1}$, the stable rational cohomology of the classifying space $\B \aut_\bdry(W_{g,1})$ of self-equivalences (relative to the boundary) is given by $\mathrm{GC}^{2n}(0)$ (also see Stoll \cite{Sto} for a version of this result for certain odd-dimensional manifolds).
They furthermore proved that the stable rational cohomology of the classifying space of block diffeomorphisms of $W_{g,1}$ is given by $\Schur {\mathrm{GC}^{2n}} {\Pi}$, i.e.\ the Lie graph complex with hairs decorated by $\Pi \defeq \pi_*(\B \mathrm{O}) \tensor \QQ$, corresponding to (formal) Pontrjagin classes.
All of these classes are recovered by the construction $\kappa_d$ of Berglund, and hence our result applies to them as well.

While completely combinatorial, the homology of the hairy Lie graph complex is hard to compute.
By computer calculations in low degrees its structure appears to be complicated (see e.g.\ Brun--Willwacher \cite[§4.6]{BW24}), and by work of Borinsky--Vogtmann \cite{BV} on the asymptotic Euler characteristic of the hairless part, it is known that its homology is very large.
On the other hand, by work of Galatius--Randal-Williams \cite{GR14}, the stable rational cohomology of the classifying space $\B \Diff_\bdry(W_{g,1})$ of diffeomorphisms of $W_{g,1}$ is a simple-to-understand polynomial algebra and in particular much smaller.
Hence there must be many relations between the classes of Berglund--Madsen when evaluated on (smooth) fiber bundles.
Identifying these relations explicitly is a problem of significant interest.
Our \cref{intro:thm:graph_classes} provides the first result in this direction by resolving it for the classes corresponding to graphs with exactly one loop.

\paragraph{The proof.}
Adapting the work of Berglund \cite{Ber}, our map $\kappa_\wheel$ is obtained by choosing a cdga model $\k$ for $B$ and a dualizable $\Cinf$-model $R$ over $\k$ for $E$ (recall that a $\Cinf$-algebra is a strictly commutative $\Ainf$-algebra); using the $\Cinf$-structure maps of $R$ and the $\k$-linear trace of endomorphisms of $R$, we construct a map $\Schur {\mathrm{GC}^\wheel_1} {R} \to \k$.
We then observe that $\Schur {\mathrm{GC}^\wheel_1} {R}$ is quasi-isomorphic to $\HC(\Cinf(R))$, the Connes complex of the free $\Cinf$-algebra on $R$, in such a way that the following diagram commutes
\[
\begin{tikzcd}
  \HH \bigl( \Cinf(R) \bigr) \dar \ar{rr} & & \HH(R) \dar{\tr} \\
  \HC \bigl( \Cinf(R) \bigr) \rar{\eq} & \Schur {\mathrm{GC}^\wheel_1} {R} \rar & \k
\end{tikzcd}
\]
where $\tr$ is the map of \cref{intro:cor:main}.
The left-hand vertical map is (almost) surjective by comparison with the case of free commutative algebras.
Using this, \cref{intro:thm:graph_classes} can be deduced from the following result.

\begin{introtheorem}[see \cref{thm:trace_vanishes}] \label{intro:thm:trace_vanishes}
  Given a fiber bundle $p \colon E \to B$ between nilpotent spaces of finite rational type such that the fiber of $p$ is a compact simply connected topological manifold, let $\k$ be a cdga model for $B$ and $R$ a unital $\Cinf$-algebra in $\k$-modules that models $E$.
  If $R$ is cofibrant and dualizable as a $\k$-module, then the following composite is trivial in the derived category of $\k$-modules
  \[ \rHH_\k(R)  \longto  \HH_\k(R)  \xlongto{\tr}  \k \]
  where $\rHH_\k(R)$ denotes the kernel of the projection $\HH_\k(R) \to R$.
\end{introtheorem}

Just as for \cref{intro:cor:main}, we in fact prove a version of this theorem for equivariant rational models for fiberwise nilpotent spaces over $\B G$.
It is deduced by combining \cref{intro:cor:main} with a result of Dwyer--Weiss--Williams \cite{DWW}, who proved that, for topological fiber bundles, the fiberwise A-theory Euler characteristic (and hence the one for THH) lifts along the fiberwise assembly map.
The condition that $R$ is a $\Cinf$-algebra (instead of just an $\Ainf$-algebra) is required for the map $\HH_\k(R) \to R$ to be well-defined.

\subsubsection{The Becker--Gottlieb transfer}

The \emph{Becker--Gottlieb transfer} of a fibration $f \colon X \to Y$ with finitely dominated fibers is a wrong-way map $f^! \colon \SS[Y] \to \SS[X]$ of the suspension spectra.
It was originally defined by Becker--Gottlieb \cite{BG76} for fibrations whose base is finite-dimensional, and extended by Clapp \cite{Cla} to the general case.
As mentioned in the beginning of the introduction, Lind--Malkiewich \cite{LM} proved that it can be recovered from the THH transfer.
Combining this with (a slight generalization of) \cref{intro:cor:main} yields the following rational model for the Becker--Gottlieb transfer.

\begin{introcorollary}[see \cref{cor:BG}] \label{intro:cor:BG}
  Given a fibration $f \colon X \to Y$ between simply connected spaces of finite rational type such that the fiber of $f$ is simply connected and finitely dominated, let $\k$ be a cdga model for $Y$ and $R$ a unital $\Ainf$-algebra in $\k$-modules that models $X$.
  If $R$ is cofibrant and dualizable as a $\k$-module, then the Becker--Gottlieb transfer $f^! \colon \SS[Y] \to \SS[X]$ is modeled by the map of cochain complexes
  \[ R \longto \k, \qquad  r  \longmapsto  \tr_\k \bigl( \mu^R_2(r, \blank) \bigr) \]
  where $\mu^R_2 \colon R \tensor_\k R \to R$ denotes the multiplication of the $\Ainf$-algebra $R$, and $\tr_\k(\phi)$ the $\k$-linear trace of an endomorphism $\phi \colon R \to R$.
\end{introcorollary}

A parametrized version of the result of Lind--Malkiewich would allow us to prove a parametrized version of \cref{intro:cor:BG} (including for fiberwise nilpotent spaces over $\B G$).
It seems likely that the proof of Lind--Malkiewich (or the proof of Carmeli--Cnossen--Ramzi--Yanovski \cite[§6.2]{CCRY}) generalizes to the parametrized situation, but this has not appeared in the literature so far.

\subsubsection{The space of THH-simple structures} \label{intro:sec:THH-simple}

The space $\Simp^{\A}(X)$ of \emph{simple structures} on a finitely dominated space $X$ is the space of lifts of its A-theory Euler characteristic $\chi \colon \SS \to \A(X)$ along the assembly map $\SS[X] \tensor \A(*) \to \A(X)$ (see e.g.\ \cite[Lecture~13]{LurAKM}).
Given a fibration $E \to B$ with fiber $X$, we can perform this construction fiberwise to obtain a space $\Simp[B]^{\A}(E)$ over $B$, whose fiber over a point $b$ is the space of simple structures on $E_b \eq X$.
We can perform the same construction using THH instead of A-theory, to obtain a space $\Simptr[B](E)$ of \emph{fiberwise THH-simple structures}.
Using \cref{intro:cor:main}, we construct a rational model for $\Simptr[B](E)$.
As we explain below, this is a step towards obtaining a rational model for the classifying space of homeomorphisms (or diffeomorphisms) of a compact simply connected manifold in the concordance stable range.

\begin{introtheorem}[see \cref{thm:trace-simple}] \label{intro:thm:trace-simple}
  Given a fibration $f \colon E \to B$ between nilpotent spaces of finite rational type such that the fiber of $f$ is $2$-connected and finitely dominated, let $\k$ be a cdga model for $B$ and $R$ a $1$-connected unital $\Cinf$-algebra in $\k$-modules that models $E$.
  If $\k \to R$ is a cofibration and $R$ is dualizable as a $\k$-module, then the space $\Simptr[B](E)$ is modeled by the cdga
  \[ \left( \SA_\k \rnHH_\k(R) [-1], d \defeq d_1 + d_2 \right) \]
  where $d_1$ is induced by the differentials of $\k$ and $\rnHH_\k(R)$, and $d_2$ is the unique derivation that restricts to the map $\tr \colon \rnHH_\k(R) \to \k$ of \cref{intro:cor:main}.
  Here $\nHH_\k(R) \subseteq \HH_\k(R)$ denotes the normalized Hochschild complex, and $\rnHH_\k(R)$ the kernel of the projection $\nHH_\k(R) \to R$.
\end{introtheorem}

As for \cref{intro:cor:main}, we in fact prove a version of this theorem for equivariant rational models for fiberwise nilpotent spaces over $\B G$.
The condition that the fiber of $f$ is $2$-connected is required to guarantee that $\Simptr[B](E)$ is connected.
On the algebraic side, we need to pass to the normalized Hochschild complex and assume that $R$ is $1$-connected to obtain a cdga that is concentrated in degrees $\ge 0$.
The inclusion $\nHH_\k(R) \subseteq \HH_\k(R)$ is a quasi-isomorphism, so this does not change anything homotopically.

\paragraph{Relation to classifying spaces of manifolds.}
A further variant $\Simp^{\horbits {\A}{\Cyclic 2}}$ of the space of simple structures is obtained by instead considering the homotopy orbits $\horbits {\A(X)} {\Cyclic 2}$ (for a certain involution) and the corresponding assembly map.
It is a reformulation of work of Burghelea--Lashof \cite{BL} and Waldhausen--Jahren--Rognes \cite{WJR} that, for a compact topological manifold $M$, there is a map
\[ \B \Homeo_\bdry(M)  \longto  \Simp[\widetilde B]^{\horbits {\A}{\Cyclic 2}}(\widetilde E) \]
to the space of fiberwise $\horbits {\A}{\Cyclic 2}$-simple structures on the universal fibration $\widetilde E \to \widetilde B$ over the classifying space of block homeomorphisms of $M$, and that this map is rationally an equivalence in the rational concordance stable range (see also Weiss--Williams \cite{WWIII} for an integral version).
A similar result holds for smooth manifolds and the classifying space $\B \Diff_\bdry(M)$ when one furthermore lifts along the unit map $X \tensor \horbits \SS {\Cyclic 2} \to X \tensor \horbits {\A(*)} {\Cyclic 2}$.

A rational model for the universal covering of the classifying space of block diffeomorphisms (which is rationally equivalent to the classifying space of block homeomorphisms) was constructed by Berglund--Madsen \cite{BM}, and this has been upgraded to an equivariant model for the whole classifying space by Berglund--Stoll \cite{BS24}.
Thus \cref{intro:thm:trace-simple} can be considered as a step towards obtaining a rational model for $\B \Homeo_\bdry(M)$ in the concordance stable range when $M$ is compact and simply connected; since the unit map $\SS \to \THH(*)$ is an equivalence, it is similarly a step towards a rational model for $\B \Diff_\bdry(M)$.

The next step could be to prove \cref{intro:conj:HC}, which would allow to obtain an analog of \cref{intro:thm:trace-simple} for lifts of the Euler characteristic valued in rational negative cyclic homology.
By work of Goodwillie \cite{Goo}, this is closely related to A-theory rationally, so that it would only remain to incorporate the involution of the latter.
This should be tractable rationally, see Bustamente--Farrell--Jiang \cite{BFJ}.

\paragraph{Relation to manifold calculus.}
Using our version of \cref{intro:cor:main} for dg Lie models (see \cref{sec:Lie_models}), one can also obtain a dg Lie model for $\Simptr[B](E)$, see \cref{rem:Simptr_Lie}.
It is closely related to a dg Lie algebra constructed by Willwacher \cite{W23} as a model for the monoid of self-equivalences of the rationalized right $\mathrm{E}_d$-module $\mathcal{F}_M$ of configurations in a $d$-dimensional parallelized manifold $M$, which is the manifold calculus approximation of $\Diff_\bdry(M)$.
His dg Lie algebra is defined as a certain graph complex; its part of loop order $\leq 1$ is essentially the same as the model we obtain for $\Simptr[B](E)$, with Hochschild homology replaced by dihedral homology.

\paragraph{Relation to string topology.}
The notion of a THH-simple space also naturally appears (seemingly independently of the above interpretation as an approximation to a simple space) in string topology and explains the non-homotopy-invariance of the Goresky--Hingston loop coproduct.
Namely, Naef--Safronov \cite{NS} showed that, given a finitely dominated space $X$ equipped with a THH-simple structure, one can construct a loop coproduct that agrees with the Goresky--Hingston coproduct when $X$ is a closed oriented manifold.
In particular, the action of $\aut(X)$ on the loop coproduct factors through the action of $\aut(X)$ on $\Simptr(X)$, which is modeled by \cref{intro:thm:trace-simple}.
Using different methods, Naef--Willwacher \cite{NW} found that the action of $\aut(X)$ on the loop coproduct is given by a formula as in \cref{intro:cor:main} and is generally non-trivial.

\subsection*{Acknowledgments}

First and foremost, we would like to thank Nils Prigge, who was involved in the beginning of this project but declined to be an author.
Secondly, we would like to thank Manuel Krannich for suggesting to us that a rational model for the fiberwise A-theory transfer might allow to prove vanishing of graph characteristic classes, which started this project.
Moreover we would like to thank Alexander Berglund for providing us with a draft of \cite{Ber} and related discussions, and Fabian Hebestreit, Cary Malkiewich, Samuel Muñoz-Echániz, Thomas Nikolaus, Matteo Poletto, Oscar Randal-Williams, Maxime Ramzi, George Raptis, Wolfgang Steimle, Jan Steinebrunner, Ferdinand Wagner, and Kelly Wang for helpful conversations.

While working on this project, Robin Stoll was supported by a postdoctoral scholarship of the Knut and Alice Wallenberg foundation.

%% file: 2preliminaries.tex
\section{Preliminaries}

In this section, we fix our notation and recall basic constructions and facts we will need throughout the main body of the paper.

\subsection{Model categories, \texorpdfstring{$\infty$}{infinity}-categories, animas, and spectra}

This subsection is an abbreviated version of \cite[\Iseccat]{I}.
We will assume familiarity with model categories and the language of higher category theory as covered in \cite[§1 and Appendix A]{LurHTT}.
In particular the term \emph{$\infty$-category} will mean quasi-category, though we work model independently throughout.

\begin{definition}
  Let $\cat M$ be a model category.
  We write $\Underlying(\cat M) \defeq \loc {\cat M} {\We{\cat M}}$ for the \emph{underlying $\infty$-category of $\cat M$}, i.e.\ the $\infty$-categorical localization of $\cat M$ at the class $\We{\cat M}$ of the weak equivalences (see \cite[§1.3.4]{LurHA}).
\end{definition}

Recall that the homotopy category of $\Underlying(\cat M)$ is naturally equivalent to the homotopy category of $\cat M$ (see e.g.\ \kerodon{01MW}).

\subsubsection*{Animas}

\begin{notation}
  We write $\An$ for the $\infty$-category of $\infty$-groupoids and call its objects \emph{animas}.
\end{notation}

Recall that, for an anima $X$, the unstraightening construction yields an equivalence of $\infty$-categories $\Fun(X, \An) \eq \overcat{\An}{X}$ (see \cite[§3.3.2]{LurHTT}).

\begin{definition}
  A map of animas $f \colon X \to Y$ is a \emph{rational equivalence} if it induces an isomorphism on rational homology.
  An anima $L$ is \emph{rational} if the map of animas $f^* \colon \Map(Y, L) \to \Map_{\cat C}(X, L)$ is an equivalence for every rational equivalence $f \colon X \to Y$.
  We write $\An^\QQ \subset \An$ for the full subcategory spanned by the rational animas.
  We denote by $(\blank)_\QQ \colon \An \to \An^\QQ$ the \emph{rationalization functor}, i.e.\ the left adjoint of the inclusion.
\end{definition}

Note that the inclusion $\An^\QQ \subset \An$ indeed admits a left adjoint (see e.g.\ \cite[\Ilemmalocalization]{I}).
We will now recall that the rationalization functor is particularly well-behaved for nilpotent animas.
We begin by recalling the notion of a nilpotent action on a group; the definition we use is taken from Bousfield--Kan \cite[Ch.~II, 4.1]{BK}.

\begin{definition}
  Let $K$ be a group that acts on a group $G$, i.e.\ it comes equipped with a group homomorphism $K \to \Aut_\Grp(G)$.
  We say that this action is \emph{nilpotent} if $G$ has a finite normal series of subgroups (i.e.\ each $A_i$ is normal in $A_{i+1}$)
  \[ 1 = A_0 \subseteq A_1 \subseteq \dots \subseteq A_n = G \]
  such that each $A_i$ is fixed setwise by $K$ and each quotient $\quot {A_{i+1}} {A_i}$ is abelian with trivial $K$-action.
\end{definition}

\begin{definition}
  A connected anima $X$ is \emph{nilpotent} if for some $x \in X$ the action of $\pi_1(X, x)$ on $\pi_n(X, x)$ is nilpotent for all $n \ge 1$.
\end{definition}

On nilpotent animas, a model for the rationalization functor is given by the Bousfield--Kan $\QQ$-completion, see \cite[Ch.~V, Proposition~4.2]{BK}.
The $\QQ$-completion of a nilpotent anima is again nilpotent by \cite[Ch.~II, Lemma~4.8]{BK}; in particular the rationalization of a nilpotent anima is again nilpotent.
Moreover recall that for a nilpotent anima $X$, there are canonical isomorphisms $\pi_n(\rat X) \iso \pi_n(X) \tensor \QQ$ and $\Ho n (\rat X; \ZZ) \iso \Ho n (X; \QQ)$ for $n \ge 1$ (see e.g.\ \cite[Ch.~V, Proposition~3.1]{BK}); for $n = 1$, the expression $\pi_1(X) \tensor \QQ$ denotes the Malcev completion of the nilpotent group $\pi_1(X)$.

\subsubsection*{Spectra}

We will assume familiarity with stable $\infty$-categories and the $\infty$-category of spectra, see e.g.\ \cite[§1]{LurHA}.
We fix the following notation.

\begin{notation}
  We write $\Sp$ for the $\infty$-category of spectra, $\Loopsinf \colon \Sp \to \An$ for the infinite loop functor, $\SS[\blank] = \Suspinf \colon \An \to \Sp$ for its left adjoint, and $\SS \defeq \SS[*]$.
\end{notation}

\begin{definition}
  Let $X$ be an anima.
  A \emph{parametrized spectrum over $X$}, or \emph{$X$-spectrum} for short, is a functor $X \to \Sp$.
  We write $\Sp[X] \defeq \Fun(X, \Sp)$ for the $\infty$-category of $X$-spectra, and $\SS_X \defeq \const_\SS \in \Sp[X]$.
  We furthermore write
  \[ \Loopsinf[X] \colon \Sp[X] \xlongto{(\Loopsinf)_*} \Fun(X, \An) \eq \Anover{X}  \qquad \text{and} \qquad  \Suspinf[X] \colon \Anover{X} \eq \Fun(X, \An) \xlongto{(\Suspinf)_*} \Sp[X] \]
  and sometimes use the alternative notation $\SS_X[\blank] \defeq \Suspinf[X](\blank)$.
\end{definition}

\begin{notation}
  Let $X$ be an anima, $x \colon * \to X$ a map, $E$ an $X$-spectrum, and $n \in \ZZ$.
  Then we write $\pi_n(E, x) \defeq \pi_n(x^* E)$ for the fiberwise homotopy groups of $E$ at the point $x$.
\end{notation}

\begin{definition}
  Let $X$ be an anima.
  An $X$-spectrum $E$ is \emph{rational} if $\pi_n(E, x)$ is uniquely divisible (i.e.\ a rational vector space) for all $n \in \ZZ$ and $x \in X$.
  We denote by $\Sp[X]^\QQ \subseteq \Sp[X]$ the full subcategory spanned by the rational $X$-spectra, and write $\rat {(\blank)} \colon \Sp[X] \to \Sp[X]^\QQ$ for the left adjoint of the inclusion.
  A map of $X$-spectra $f \colon E \to E'$ is a \emph{rational equivalence} if $\rat f$ is an equivalence.
\end{definition}

Note that the functor $\rat {(\blank)}$ exists and is given by $\blank \tensor_X X^*(\H \QQ)$ (see e.g.\ \cite[\Iseccat]{I}).
Furthermore recall that, for an $X$-spectrum $P$ and $x \in X$, there is a natural isomorphism $\pi_*(\rat P, x) \iso \pi_*(P, x) \tensor \QQ$.
In particular a map of $X$-spectra $f \colon E \to E'$ is a rational equivalence if and only if $\pi_*(f, x) \tensor \QQ$ is an isomorphism for all $x \in X$.

As in the case of animas, the rational homotopy theory of parametrized spectra is better behaved on a certain class of nilpotent objects.
The following definition is taken from \cite[§2.2]{Bra}.

\begin{definition} \label{def:nilpotent_spectrum}
  Let $X$ be an anima.
  Then an $X$-spectrum $E$ is \emph{nilpotent} if the action of $\pi_1(X, x)$ on $\pi_n(E, x)$ is nilpotent for each $n \in \ZZ$ and $x \in X$.
  We denote by $\Sp[X]_\nil \subseteq \Sp[X]$ the full subcategory spanned by the nilpotent $X$-spectra.
\end{definition}

\subsection{Differential graded algebra}

In this subsection, which is almost identical to \cite[\Isecdga]{I}, we fix our conventions for homological algebra, and recall elementary properties of the homotopy theory of differential graded algebras and their modules.

\begin{convention}
  The base field is $\QQ$.
  We use cohomological grading conventions, i.e.\ differentials have degree $1$ and we write degrees as superscripts.
\end{convention}

\begin{notation}
  We write $\Ch$ for the category of unbounded cochain complexes, equipped with its usual symmetric monoidal structure.
  We denote by $\shift[n]$ a cohomological shift by $-n$, i.e.\ for a cochain complex $C$, we set $(\shift[n] C)^p \defeq C^{p+ n}$.
  We write $\shift C \defeq \shift[1] C$.
\end{notation}

Since we will mainly require homological shifts by $1$ (i.e.\ cohomological shifts by $-1$), we chose this convention for $\shift[n]$ to keep our notation light.

\subsubsection*{Algebras}

\begin{definition}
  A \emph{dga} is a monoid in $\Ch$.
  For a dga $R$, we denote by $\opmod{R}$ the same underlying cochain complex with the opposite multiplication, i.e.\ we define $\opmod{a} \cdot \opmod{b} \defeq (-1)^{\deg a \deg b} \opmod{(b \cdot a)}$.
\end{definition}

\begin{definition}
  A \emph{cdga} is a commutative monoid in $\Ch$.
  We denote by $\CDGA$ the category of non-negatively graded cdgas, equipped with the projective model structure, i.e.\ the weak equivalences are the quasi-isomorphisms and the fibrations are the degreewise surjections.\footnote{This projective model structure exists by \cite[Theorem~4.3]{BG}.}
  A \emph{cofibration of cdgas} will refer to a map of non-negatively graded cdgas that is a cofibration in $\CDGA$; a \emph{cofibrant cdga} will refer to a non-negatively graded cdga that is cofibrant in $\CDGA$.
  We will sometimes implicitly consider a $\QQ$-algebra to be a dga concentrated in degree $0$.
\end{definition}

\begin{definition}
  A non-negatively graded cdga $R$ is \emph{homologically connected} if $\Coho 0 (R) \iso \QQ$; it is \emph{connected} if $R^0 \iso \QQ$.
\end{definition}

\begin{definition}
  A map of non-negatively graded cdgas $R \to S$ is \emph{quasi-free} if the underlying graded algebra of $S$ is free over $R$, i.e.\ if there is an isomorphism $S \iso R \tensor \SA V$ of graded algebras under $R$ for some graded vector space $V$.
  A quasi-free map of cdgas is \emph{semifree} if the isomorphism can be chosen such that $V$ admits a basis $(v_\alpha)_{\alpha \in I}$ for some well-ordered set $I$ such that $d(v_\alpha) \in R \tensor \SA V_{< \alpha}$, where $V_{< \alpha} \subset V$ is the subspace spanned by those $v_\beta$ with $\beta < \alpha$.
  A semifree map of cdgas is \emph{minimal} if there exists such a basis such that $\deg{v_\alpha} \le \deg{v_\beta}$ whenever $\alpha \le \beta$.
  A cdga $R$ is \emph{quasi-free}, \emph{semifree}, or \emph{minimal} when the map $\QQ \to R$ is so.
  A \emph{minimal model} of a cdga $R$ is a quasi-isomorphism of cdgas $M \to R$ such that $M$ is minimal.
\end{definition}

Note that a semifree map of cdgas is a cofibration (see e.g.\ \cite[Remark~3.4]{Bra}).
Recall that a homologically connected cdga always admits a minimal model (see e.g.\ \cite[Proposition~7.7]{BG}), which is automatically connected.

\subsubsection*{Modules}

\begin{definition}
  Let $R$ be a dga.
  An \emph{$R$-module} is a left module over $R$ in $\Ch$.
  We denote by $\Mod{R}$ the category of $R$-modules, equipped with the projective model structure, i.e.\ the weak equivalences are the quasi-isomorphisms and the fibrations are the degreewise surjections.\footnote{This projective model structure exists by \cite[Theorem~3.3]{BMR}.}
  When $R$ is commutative, we equip $\Mod{R}$ with the closed symmetric monoidal structure given by $\tensor_R$.
\end{definition}

Note that for us a $\QQ$-module is by definition a cochain complex.
By \cite[Theorem~3.3]{BMR}, the model category $\Mod{R}$ is monoidal if $R$ is commutative.
In particular it is enriched over itself, and thus the internal hom-functor $\Hom_R(\blank, N) \colon \opcat{\Mod{R}} \to \Mod{R}$ is right Quillen, and so is $\Hom_R(M, \blank) \colon \Mod{R} \to \Mod{R}$ when $M$ is cofibrant (see e.g.\ \cite[§16.4]{MP}).
Furthermore recall that, if $R \to S$ is a cofibration of cdgas, then $S$ is cofibrant as an $R$-module (e.g.\ by \cite[\Iobscofibcdgas]{I}).

\begin{notation}
  For a cdga $R$ and an $R$-module $M$, we write $\dual M_R$ for $\Hom_R(M, R)$ as an $R$-module.
  Note that, if $R \to S$ is a map of cdgas, then $\dual S_R$ canonically lifts to an $S$-module.
\end{notation}

\begin{definition}
  Let $R$ be a dga.
  An $R$-module $M$ is \emph{quasi-free} if its underlying graded module over the underlying graded algebra of $R$ is free, i.e.\ if there is an isomorphism of graded $R$-modules $M \iso R \tensor V$ for some graded vector space $V$.
  A quasi-free $R$-module is \emph{semifree} if the isomorphism can be chosen such that $V$ admits a basis $(v_\alpha)_{\alpha \in I}$ for some well-ordered set $I$ such that $d(v_\alpha) \in R \tensor V_{< \alpha}$, where $V_{< \alpha} \subset V$ is the subspace spanned by those $v_\beta$ with $\beta < \alpha$.
  A semifree $R$-module is \emph{minimal} if there exists such a basis such that $\deg{v_\alpha} \le \deg{v_\beta}$ whenever $\alpha \le \beta$.
  A \emph{minimal model} of an $R$-module $N$ is a quasi-isomorphism of $R$-modules $M \to N$ such that $M$ is minimal.
\end{definition}

For a non-negatively graded cdga $R$, recall from \cite[§3.2]{Bra} that semifree $R$-modules are cofibrant, that an $R$-module whose homology is bounded below admits a minimal model, and that a minimal model is unique up to isomorphism if it exists.

\subsubsection*{Hochschild homology}

We now recall the bar construction, and the complexes defining Hochschild homology and cyclic homology of an algebra.
For more background, see for example \cite[Chapters~1 and 2]{Lod}.

\begin{definition} \label{def:bar}
  Let $\k$ be a cdga and $\k \to R$ a map of dgas.
  The \emph{two-sided bar construction} $\B_\k(R, R, R)$ of $R$ as a $\k$-algebra is the $(R \tensor_\k \opmod{R})$-module obtained as the total complex of
  \[ \cdots  \xlongto{d_3}  R \tensor_\k R^{\tensor_\k 2} \tensor_\k R  \xlongto{d_2}  R \tensor_\k R \tensor_\k R  \xlongto{d_1}  R \tensor_\k R \]
  where $d_n \defeq \sum_{i = 0}^{n} (-1)^i {\id^{\tensor i}} \tensor \mu_R \tensor \id^{\tensor n - i}$ with $\mu_R$ the multiplication of $R$.
  The augmentation $d_0 = \mu_R \colon R \tensor_\k R \to R$ yields a canonical map $\epsilon \colon \B_\k(R, R, R) \to R$ of $(R \tensor_\k \opmod{R})$-modules.
\end{definition}

Recall that the map $\epsilon$ is a quasi-isomorphism, and that the $(R \tensor_\k \opmod{R})$-module $\B_\k(R, R, R)$ is cofibrant when $R$ is cofibrant as a $\k$-module (see e.g.\ \cite[\Ilemmabarcomplex]{I}).

\begin{definition} \label{def:HH}
  Let $\k$ be a cdga and $\k \to R$ a map of dgas.
  For an $(R \tensor_\k \opmod{R})$-module $M$, we write
  \[ \HH_\k(R, M)  \defeq  \B_\k(R, R, R) \tensor_{R \tensor_\k \opmod{R}} M \]
  for the \emph{Hochschild complex} of $R$ with coefficients in $M$.
  Note that there is a canonical isomorphism $\HH_\k(R, M) \iso \bigoplus_{n \ge 0} M \tensor_\k (\shift R)^{\tensor_\k n}$ of graded $\k$-modules.
  We write $\HH_\k(R) \defeq \HH_\k(R, R)$ and define the \emph{Connes complex} of $R$ to be the quotient
  \[ \HC_\k(R)  \defeq  \shift[-1] \bigoplus_{n \ge 0} (\shift R)^{\tensor_\k n+1}_{\Cyclic {n+1}}  \longtwoheadleftarrow  \shift[-1] \bigoplus_{n \ge 0} (\shift R)^{\tensor_\k n+1}  \iso  \shift[-1] \HH_\k(R, \shift R)  \iso  \HH_\k(R) \]
  where the cyclic group $\Cyclic{n+1}$ acts by cyclic permutations.
\end{definition}

For later use, we also need a version of Hochschild homology of algebras in bimodules, as follows.

\begin{notation} \label{not:cyctensor}
  Given a cdga $\k$ and a map of dgas $\k \to R$, note that the category of $(R \tensor_\k \opmod{R})$-modules has a (generally non-symmetric) monoidal structure $M \tensor_R N$, using the right $R$-module structure of $M$ and the left $R$-module structure of $N$.
  For an $(R \tensor_\k \opmod{R})$-module $M$, we write $M^{\tensor_R n}$ for the $n$-fold tensor product of $M$ with itself in this monoidal structure, and $M^{\cyctensor_R n}$ for the $\k$-module $M^{\tensor_R n} \tensor_{R \tensor_\k \opmod{R}} R$.
  Note that there is a canonical isomorphism $\cyclic{n} \colon M^{\cyctensor_R n} \to M^{\cyctensor_R n}$ sending $m_1 \tensor \dots \tensor m_n$ to $\pm m_n \tensor m_1 \tensor \dots \tensor m_{n-1}$ with the sign determined by the Koszul rule.
\end{notation}

\begin{definition} \label{def:HH_bimodule}
  Let $\k$ be a cdga, let $\k \to R$ be a map of dgas, and let $A$ be a non-unital monoid in the category of $(R \tensor_\k \opmod{R})$-modules with respect to $\tensor_R$.
  Its \emph{Hochschild homology} is the $\k$-module $\HH_R(A)$ defined to be the total complex of
  \begin{gather*}
    \cdots  \xlongto{d_4}  A^{\cyctensor_R 3}  \xlongto{d_3}  A^{\cyctensor_R 2}  \xlongto{d_2}  A^{\cyctensor_R 1}
    \shortintertext{where}
    d_n \defeq (\mu_A \tensor \id^{\tensor n-2}) \after \cyclic{n} + \sum_{i = 1}^{n-1} (-1)^i {\id^{\tensor i-1}} \tensor \mu_A \tensor \id^{\tensor n-i-1}
  \end{gather*}
  with $\mu_A$ the multiplication of $A$.
\end{definition}

Note that this recovers the previous definition of Hochschild homology when $R = \k$.
The following lemma shows that $\HH_\k(R)$ can be recovered from the Hochschild homology of the bimodule $\B_\k(R, R, R)$ equipped with a specific multiplication.

\begin{lemma} \label{lemma:inc_qiso}
  Let $\k$ be a cdga, let $\k \to R$ be a map of dgas, and let $\epsilon \colon B \to R$ be a quasi-isomorphism of $(R \tensor_\k \opmod{R})$-modules such that $B$ is cofibrant.
  Equip $B$ with the associative multiplication $B \tensor_R B \to B$ given by $b \cdot b' \defeq b \epsilon(b')$.
  Then the canonical inclusion
  \[ \inc  \colon  B^{\cyctensor_R 1}  \longto  \HH_R ( B ) \]
  is a quasi-isomorphism.
\end{lemma}

\begin{proof}
  First note that $\HH_R(B) \iso U \tensor_{R \tensor_\k \opmod{R}} R$, where $U$ is the total complex of
  \[ \cdots  \xlongto{d_4}  B^{\tensor_R 3}  \xlongto{d_3}  B^{\tensor_R 2}  \xlongto{d_2}  B \]
  where $d_n  \defeq  \sum_{i = 0}^{n-1} (-1)^i \id^{\tensor i} \tensor \epsilon \tensor \id^{\tensor n-1-i}$.
  The maps $\epsilon^{\tensor n}$ assemble into a map $\pi$ from $U$ to the total complex $V$ of
  \[ \dots  \xlongto{d_4}  R  \xlongto{d_3}  R  \xlongto{d_2}  R \]
  where $d_n \defeq \id_R$ if $n$ is odd and $d_n \defeq 0$ if $n$ is even.
  Since $B$ is a cofibrant $(R \tensor_\k \opmod{R})$-module, the maps $\epsilon^{\tensor n}$ are quasi-isomorphisms, and hence $\pi$ is one as well.
  The projection $\pr \colon V \to R$ onto degree $1$ is clearly a quasi-isomorphism, and the composite of the inclusion $B \to U$ with $\pr \after \pi \colon U \to R$ is equal to $\epsilon$ and thus also a quasi-isomorphism.
  Hence the inclusion $B \to U$ is a quasi-isomorphism.
  Both $B$ and $U$ are cofibrant $(R \tensor \opmod{R})$-modules, and hence the induced map
  \[ B^{\cyctensor_R 1}  =  B \tensor_{R \tensor_\k \opmod{R}} R  \longto  U \tensor_{R \tensor_\k \opmod{R}} R \iso  \HH_R \bigl( B \bigr) \]
  is a quasi-isomorphism by \cite[\Ilemmacofibrantflat]{I}.
\end{proof}

\subsection{Equivariant rational models} \label{sec:eq}

By work of Sullivan \cite{Sul} (also see Bousfield--Gugenheim \cite{BG}), the homotopy category of nilpotent animas of finite rational type is equivalent to the homotopy category of homologically connected cdgas of finite rational type.
We now recall how to rationally model animas that are not necessarily nilpotent by considering them as \emph{fiberwise} nilpotent animas over $\B G$ for some discrete group $G$.
This can be achieved for any anima $X$, for example using the canonical map $X \to \B \pi_1(X)$ (however note that there exist maps of animas that cannot be realized as a map of fiberwise nilpotent animas over $\B G$ for any fixed $G$).
Since fiberwise nilpotent animas over $\B G$ are equivalently nilpotent animas with an action of $G$, they can be modeled by cdgas with an action of $G$.
We begin by recalling the relevant definitions.
Also see \cite{BS,GHT,BZ} for further elaborations on these ideas.

\begin{definition}
  Let $G$ be a group.
  A \emph{$G$-equivariant cdga} is a cdga equipped with a left $G$-action.
  We denote by $\CDGAeq{G} \defeq \Fun(\B G, \CDGA)$ the category of non-negatively graded $G$-equivariant cdgas, equipped with the injective model structure, i.e.\ the weak equivalences are the quasi-isomorphisms and the cofibrations are the cofibrations of the underlying cdgas.\footnote{This injective model structure exists by \cite[Proposition~A.2.8.2]{LurHTT}.}
\end{definition}

\begin{definition}
  A cdga $A$ is of \emph{finite homotopical type} if it admits a minimal model that is finite-dimensional in each degree.
  For a group $G$, we write $\CDGAeq[\ge 1, \fht]{G} \subset \CDGAeq{G}$ for the full subcategory spanned by the $G$-equivariant cdgas whose underlying cdga is homologically connected and of finite homotopical type.
\end{definition}

\begin{definition}
  An anima $X$ is \emph{of finite rational type} if $\Ho n (X; \QQ)$ is finite-dimensional for all $n \ge 0$.
  For an anima $B$, we write $\Anover[\fnil,\ft]{B} \subseteq \Anover{B}$ for the full subcategory spanned by the animas over $B$ such that each fiber is of finite rational type and nilpotent; we furthermore write $\Anover[\QQ,\fnil,\ft]{B} \subseteq \Anover[\fnil,\ft]{B}$ for the full subcategory spanned by those animas over $B$ such that each fiber is additionally rational.
\end{definition}

With notation as above, recall that there is an equivalence of $\infty$-categories
\[ \Real^G  \colon  \opcat{ \Underlying \bigl( \CDGAeq{G} \bigr)_{\ge 1, \fht} }  \xlongto{\eq}  \Anover[\QQ,\fnil,\ft]{\B G} \]
(see e.g.\ \cite[\Iseceq]{I}).

\begin{definition}
  Let $G$ be a group and $\cat I$ a category.
  We say that a diagram $A \colon \cat I \to \opcat{(\CDGAeq[\ge 1, \fht]{G})}$ \emph{models} a diagram $X \colon \cat I \to \Anover[\fnil,\ft]{\B G}$ if it comes equipped with an equivalence $\Real^G \after u \after A \eq \rat X$ of functors $\cat I \to \Anover[\QQ,\fnil,\ft]{\B G}$, where $u \colon \opcat{(\CDGAeq[\ge 1, \fht]{G})} \to \opcat{\Underlying(\CDGAeq{G})_{\ge 1, \fht}}$ is the canonical functor.
\end{definition}

\subsubsection*{Equivariant models for parametrized spectra}

Given a $G$-equivariant cdga $R$ that models an anima $X$ over $\B G$, we now recall from \cite[\Iseceq]{I} how to model $X$-spectra using $G$-equivariant $R$-modules.
This is a generalization of work of Braunack--Mayer \cite{Bra} for the case $G = 1$.

\begin{definition}
  Let $G$ be a group and $R$ a $G$-equivariant cdga.
  A \emph{$G$-equivariant $R$-module} is a module $M$ over the underlying cdga of $R$ together with a left $G$-action on the underlying cochain complex of $M$ such that the module structure map $R \tensor M \to M$ is $G$-equivariant (where $G$ acts diagonally on the tensor product).
  A \emph{map of $G$-equivariant $R$-modules} is a map of the underlying $R$-modules that is $G$-equivariant.
  We denote by $\Modeq{G}{R}$ the category of $G$-equivariant $R$-modules, equipped with the injective model structure, i.e.\ the weak equivalences are the quasi-isomorphisms and the cofibrations are the cofibrations of $R$-modules.%
  \footnote{This injective model structure exists by \cite[Theorem~2.30]{Bar}, noting that $\Modeq{G}{R}$ is equivalent to the category of left sections of the left Quillen presheaf $\B G \to \twoCat$ given by sending the unique object to $\Mod{R}$ and a morphism $g \in G$ to scalar extension along acting by $g$.}
\end{definition}

Equivalently, a $G$-equivariant cdga is a commutative monoid in the category of $G$-equivariant cochain complexes equipped with the symmetric monoidal structure given by the tensor product with the diagonal $G$-action, and a $G$-equivariant module is a module over such an algebra.

\begin{definition}
  For a cdga $R$, we say that an $R$-module $M$ is \emph{of finite homotopical type} if it admits a minimal model $R \tensor V$ such that $V$ is finite-dimensional in each degree; we say that $M$ is \emph{homotopically finite} if it admits a minimal model such that $V$ is finite-dimensional.
  For a group $G$ and a $G$-equivariant cdga $R$, we write $(\Modeq{G}{R})_\fht \subseteq \Modeq{G}{R}$ for the full subcategory spanned by those $G$-equivariant $R$-modules whose underlying $R$-module is of finite homotopical type.
\end{definition}

\begin{definition}
  Let $X$ be an anima.
  We say that an $X$-spectrum $E$ is of \emph{finite rational type} if $\pi_n(E, x) \tensor \QQ$ is finite-dimensional for all $n \in \ZZ$ and $x \in X$.
  We say that an $X$-spectrum $E$ is \emph{bounded below} if there exists an integer $N$ such that $\pi_*(E, x)$ is concentrated in degrees $\ge N$ for all $x \in X$.
  For a group $G$ and an anima $X$ over $\B G$, we denote by $\Sp[X]^{G}_{\nil,\ft,\bbl} \subseteq \Sp[X]$ the full subcategory spanned by those $X$-spectra whose restriction to the fiber of the structure map $X \to \B G$ is nilpotent, of finite rational type, and bounded below; we furthermore write $\Sp[X]^{G,\QQ}_{\nil,\ft,\bbl} \subseteq \Sp[X]^{G}_{\nil,\ft,\bbl}$ for the full subcategory spanned by those $X$-spectra where this restriction is furthermore rational.
\end{definition}

With notation as above, recall from \cite[\IpropBMeq]{I} that, given a group $G$ and a $G$-equivariant homologically connected cofibrant cdga $R$ of finite homotopical type that models a fiberwise nilpotent anima $X$ over $\B G$, there is an equivalence
\[ \Psi_R  \colon  \opcat{\Underlying(\Modeq{G}{R})_\fht}  \xlongto{\eq}  \Sp[X]^{G,\QQ}_{\nil,\ft,\bbl} \]
of $\infty$-categories.
This is a generalization of a result of Braunack--Mayer \cite[Theorem~4.20]{Bra} for the case $G = 1$.

\begin{definition}
  Let $G$ be a group, $R$ a $G$-equivariant homologically connected cofibrant cdga of finite homotopical type that models a fiberwise nilpotent anima $X$ over $\B G$, and $\cat I$ a category.
  We say that a diagram $M \colon \cat I \to \opcat{(\Modeq{G}{R})_\fht}$ \emph{models} a diagram $E \colon \cat I \to \Sp[X]^G_{\nil,\ft,\bbl}$ if it comes equipped with an equivalence $\Psi_R \after u \after M \eq \rat E$ of functors $\cat I \to \Sp[X]^{G,\QQ}_{\nil,\ft,\bbl}$, where $u \colon \opcat{(\Modeq{G}{R})_\fht}  \longto  \opcat{\Underlying(\Modeq{G}{R})_\fht}$ is the canonical functor.
\end{definition}

\subsection{Dualizable objects and the trace}

We recall here the notion of a dualizable object of a monoidal category, as well as the definition of the trace of an endomorphism of a dualizable object.

\begin{definition} \label{def:dualizable}
  An object $C$ of a monoidal category $\cat C$ is \emph{dualizable} if there exists an object $\dual C \in \cat C$ and maps $\coev_C \colon \unit \to C \tensor \dual C$ and $\ev_C \colon \dual C \tensor C \to \unit$ such that the \emph{triangle identities} hold, i.e.\ both of the composites
  \[ C  \xlongto{{\coev_C} \tensor \id}  C \tensor \dual C \tensor C  \xlongto{{\id} \tensor \ev_C}  C  \qquad \text{and} \qquad  \dual C  \xlongto{{\id} \tensor \coev_C}  \dual C \tensor C \tensor \dual C  \xlongto{{\ev_C} \tensor \id}  \dual C \]
  are the identity.
\end{definition}

Note that, when $\cat C$ is closed symmetric monoidal, then the dual $\dual C$ of a dualizable object $C$ is isomorphic to the internal hom $\Hom(C, \unit)$.
To see this, note that $\ev_C$ induces a map $\dual C \to \Hom(C, \unit)$, that $\coev_C$ induces a map $\Hom(C, \unit) \to \dual C$, and that these are inverse to each other by the triangle identities.
More generally, for any object $X \in \cat C$, there is a canonical isomorphism $X \tensor \dual C \iso \Hom(C, X)$.

\begin{definition} \label{def:trace}
  The \emph{trace} of an endomorphism $f \colon C \to C$ of a dualizable object in a symmetric monoidal category $\cat C$ is defined to be the composite endomorphism
  \[ \tr(f)  \colon  \unit  \xlongto{\coev}  C \tensor \dual C  \xlongto{f \tensor \id}  C \tensor \dual C  \iso  \dual C \tensor C  \xlongto{\ev}  \unit \]
  of the unit of $\cat C$.
  When $\cat C = \Mod{R}$ for some cdga $R$, we write $\tr_R(f) \defeq \tr(f)(1)$.
\end{definition}

Recall that the trace is cyclically invariant, i.e.\ given two dualizable objects $C$ and $D$ and two morphisms $f \colon C \to D$ and $g \colon D \to C$, we have $\tr(f \after g) = \tr(g \after f)$ (see e.g.\ \cite[Proposition~2.4]{PS}).

Note that, given a $G$-equivariant cdga $R$ and a $G$-equivariant $R$-module $M$, the evaluation $\Hom_R(M, R) \tensor_R M \to R$ is $G$-equivariant.
In particular the canonical map $M \tensor_R \Hom_R(M, R) \to \Hom_R(M, M)$ is $G$-equivariant.
Hence, if the underlying $R$-module of $M$ is dualizable, then the isomorphism $M \tensor_R \Hom_R(M, R) \iso \Hom_R(M, M)$ is $G$-equivariant, and so is $\coev_M$, since it corresponds under this isomorphism to the $G$-equivariant map $R \to \Hom_R(M, M)$ determined by $1 \mapsto \id$.
Thus also $\tr_R \colon \Hom_R(M, M) \to R$ is $G$-equivariant.

Furthermore note that a finite-dimensional quasi-free module over a cdga $R$ is dualizable.
The following lemma shows that the converse is often true up to quasi-isomorphism.

\begin{lemma} \label{lemma:dualizable_finite}
  Let $R$ be an augmented cdga that is bounded below, and $M$ a cofibrant $R$-module that is dualizable.
  Then $M$ is homotopically finite.
\end{lemma}

\begin{proof}
  Since $M$ is dualizable, it is in particular finitely generated (writing $\coev_M(1)$ as a finite sum $\sum_i m_i \tensor \phi_i$, any element $m \in M$ can be written as $m = \sum_i m_i \phi_i(m)$, so that the $m_i$ generate $M$).
  Hence $M$ is bounded below since $R$ is.
  Thus $M$ admits a minimal model $V \tensor R$; we want to show that $V$ is finite-dimensional.
  The functor $\blank \tensor_R \QQ$, defined using the augmentation, is a left Quillen functor and hence preserves quasi-isomorphisms between cofibrant $R$-modules; hence $V \eq M \tensor_R \QQ$ where the differential of $V$ is trivial since $V \tensor R$ is minimal.
  The functor $\blank \tensor_R \QQ$ is moreover strong monoidal and thus preserves dualizable objects.
  Hence $M \tensor_R \QQ$ is finite-dimensional and so is $V$.
\end{proof}

Lastly we include the following standard lemma, which we will need later.

\begin{lemma} \label{lemma:Hom_cofibrant}
  Let $R$ be a cdga, and let $M$ and $N$ be cofibrant $R$-modules.
  If $M$ is dualizable, then $\Hom_R(M, N)$ is cofibrant as an $R$-module.
\end{lemma}

\begin{proof}
  First note that $\Hom_R(M, N) \iso N \tensor_R \Hom_R(M, R)$, so that it is enough to prove that $\Hom_R(M, R)$ is cofibrant.
  We will now prove that $M$ is a retract of a finite cell complex.
  Since $\Hom_R(\blank, R)$ preserves retracts and finite cell complexes, and retracts of cell complexes are cofibrant, this will complete the proof.
  The model category $\Mod{R}$ is compactly generated by \cite[Theorem~3.3]{BMR}, so that the cofibrant object $M$ is a retract of a sequential cell complex $C$, which is the filtered colimit of its finite subcomplexes $C_i$.
  We now observe that
  \[ \Hom_R(M, \colim_i C_i)  \iso  (\colim_i C_i) \tensor_R \dual{M}_R  \iso  \colim_i ( C_i \tensor_R \dual{M}_R )  \iso  \colim_i \Hom_R(M, C_i) \]
  so that the inclusion $M \to C$ factors through some $C_i$.
  In particular $M$ is a retract of the finite cell complex $C_i$.
\end{proof}

\subsection{Topological Hochschild homology} \label{sec:THH}

In this subsection, we recall the construction of fiberwise topological Hochschild homology (THH) of animas, and the fiberwise THH transfer.
This uses the definition of THH as a higher categorical trace due to Hoyois--Scherotzke--Sibilla \cite[§4.5]{HSS} (see also Carmeli--Cnossen--Ramzi--Yanovski \cite[§4]{CCRY}).
This subsection is an abbreviated and simplified version of \cite[\IsecTHH]{I}; see there for more details.

\begin{notation}
  We write $\PrLLst$ for the $\infty$-category of presentable stable $\infty$-categories and strongly left adjoint functors, i.e.\ the left adjoint functors whose right adjoint is also left adjoint.
\end{notation}

Recall that there is a trace (or dimension) functor $\Dim \colon \PrLLst \to \Sp$ (see \cite[§4.5]{HSS} or \cite[Definition~4.1]{CCRY}).
Furthermore note that the construction $\Sp[\blank] = \Fun(\blank, \Sp)$ is a functor $\An \to \PrLLst$ where a map of animas $f \colon X \to Y$ is sent to the left adjoint $f_! \colon \Sp[X] \to \Sp[Y]$ of the restriction functor $f^* \colon \Sp[Y] \to \Sp[X]$.

\begin{notation}
  We write
  \[ \THH  \colon  \An  \xlongto{\Sp[\blank]}  \PrLLst  \xlongto{\Dim}  \Sp \]
  for the \emph{topological Hochschild homology} of animas.
  For an anima $B$, we write
  \[ \THH_B  \colon  \Anover{B}  \eq  \Fun(B, \An)  \xlongto{(\THH)_*}  \Fun(B, \Sp)  =  \Sp[B] \]
  for the \emph{fiberwise topological Hochschild homology} of animas over $B$.
\end{notation}

There is also a definition of $\THH$ for stable $\infty$-categories (see e.g.\ \cite[Definition~4.3]{HNS}), and $\THH$ of an anima $X$ is often defined as $\THH$ of the $\infty$-category of compact objects in $\Sp[X]$.
This agrees with the definition we use by \cite[Proposition~4.24]{HSS} (see also \cite[Remark~4.38]{HNS}).

Furthermore recall that, given a map $f \colon X \to Y$ of animas, the right adjoint $f_*$ of $f^*$ is itself left adjoint when the fibers of $f$ are compact; in particular $f^*$ is a map of $\PrLLst$ (see e.g.\ \cite[Example~4.27 and Lemma~4.21]{CCRY}).
Similarly, given a map $f \colon X \to Y$ of animas over $B$, we can consider it as a map in $\Fun(B, \An)$ and apply $\Sp[\blank]$ pointwise to obtain a map $f_!$ in $\Fun(B, \PrLLst)$.
We can take pointwise right adjoints of this map; when the fibers of $f$ are compact, this again yields a map $f^*$ of $\Fun(B, \PrLLst)$.

\begin{notation}
  For a map $f \colon X \to Y$ of animas, we write
  \[ f^* \defeq \Dim(f^*)  \colon  \THH(Y)  \longto  \THH(X) \]
  and call it the \emph{THH transfer} of $f$.
  For a map $f \colon X \to Y$ of animas over an anima $B$, we write
  \[ f^* \defeq \Dim_*(f^*)  \colon  \THH_B(Y)  \longto  \THH_B(X) \]
  and call it the \emph{fiberwise THH transfer} of $f$.
\end{notation}

Recall that there is a natural equivalence $\THH(X) \eq \SS[\L X]$, where $\L X = \Map(\Sphere 1, X)$ denotes the free loop space of $X$ (see e.g.\ \cite[Theorem~4.40]{CCRY}).
Hence there is a natural equivalence $\THH_B(X) \eq \SS_B[\L_B X]$, where $\L_B X$ denotes the fiberwise free loop space of $X$, given by fiberwise applying $\L (\blank)$.

\begin{definition} \label{def:THH_assembly}
  The \emph{assembly map} is the unique natural transformation $\alpha \colon \SS[X] \to \THH(X)$ that restricts to the equivalence $\SS \eq \THH(*)$ for $X = *$.
  For an anima $X$ over an anima $B$, the \emph{fiberwise assembly map} is the map $\SS_B[X] \to \THH_B(X)$ given by applying $\alpha$ to $X \colon B \to \An$.
\end{definition}

%% file: 2A_infty.tex
\section{\texorpdfstring{$\Ainf$}{A∞}-preliminaries} \label{sec:Ainf}

In this section, we recall the definitions of $\Ainf$-algebras, bimodules over them, and various related constructions, generalized slightly to take place over a base cdga.
In addition, we will prove various elementary lemmas that we require later.

The concept of an $\Ainf$-algebra was originally introduced by Stasheff \cite{Sta} and has found many applications since then.
See Lefèvre-Hasegawa \cite{Lef} for a comprehensive treatment, and Keller \cite{Kel} for a survey.
We will follow Getzler--Jones \cite{GJ} with our definitions, and the conventions of \cite{Kel} when dealing with explicit formulas.

\begin{notation}
  Throughout this section, we fix a cdga $\k$.
\end{notation}

\subsection{Coalgebras, comodules, and the cotensor product}

We will define $\Ainf$-algebras as coderivations on a tensor coalgebra; thus we begin by recalling coalgebras, comodules, and the cotensor product.

\begin{definition}
  A \emph{coalgebra over $\k$} is a comonoid in $\Mod{\k}$.
  Given a coalgebra $C$ over $\k$, a \emph{left comodule}, \emph{right comodule}, or \emph{bicomodule} over $C$ is an object of $\Mod{\k}$ equipped with the respective structure.
\end{definition}

\begin{definition}
  Let $C$ be a coalgebra over $\k$, $M$ a right $C$-comodule, and $N$ a left $C$-comodule.
  Then we define their \emph{cotensor product} to be
  \[ M \cotensor^C N  \defeq  \equalizer \bigl( M \tensor_\k N  \rightrightarrows  M \tensor_\k C \tensor_\k N \bigr) \]
  where the two maps are induced by the structure maps of $M$ and $N$, respectively.
\end{definition}

Note that the cotensor product is canonically functorial in $(M, N)$, and that the counit of $C$ and the structure maps of  $M$ and $N$ induce mutually inverse natural isomorphisms $C \cotensor^C N \iso N$ and $M \cotensor^C C \iso M$.
Unfortunately, when $M$ is a $B$-$C$-bicomodule and $N$ is a $C$-$D$-bicomodule, it is in general not clear how to promote $M \cotensor^C N$ to a $B$-$D$-bicomodule since the functors $B \tensor_\k \blank$ and $\blank \tensor_\k D$ do not necessarily preserve equalizers.
However, when $M$ is quasi-cofree as a right $C$-comodule, i.e.\ when there is an isomorphism $M \iso V \tensor_\k C$ of graded comodules over the graded coalgebra underlying $C$ for some graded $\k$-module $V$, then there is a natural isomorphism $(V \tensor_\k C) \cotensor^C N \iso V \tensor_\k N$ (as in the case $V = \k$ above).
This implies that the canonical natural map
\[ X \tensor_\k (M \cotensor^C N) \tensor_\k Y  \longto  (X \tensor_\k M) \cotensor^C (N \tensor_\k Y) \]
is an isomorphism for all $\k$-modules $X$ and $Y$.
Hence, in this situation, we can equip $M \cotensor^C N$ with a $B$-$D$-bicomodule structure; the same argument applies when $N$ is quasi-cofree as a left $C$-comodule.
All cotensor products we will encounter are going to be of one of these two forms, so that this suffices for our purposes.
Lastly we note that there is a canonical natural isomorphism (of $A$-$D$-bicomodules)
\[ L \cotensor^B (M \cotensor^C N)  \iso  (L \cotensor^B M) \cotensor^C N \]
when $M$ is quasi-cofree as a $B$-$C$-bicomodule, i.e.\ when there exists an isomorphism $M \iso B \tensor_\k V \tensor_\k C$ of graded $B$-$C$-bicomodules.

\subsection{\texorpdfstring{$\Ainf$}{A∞}-algebras}

We now define $\Ainf$-algebras, following Getzler--Jones \cite[§1]{GJ}; also see \cite[§2]{Kel}.

\begin{notation}
  For a $\k$-module $M$, we write $\coT_\k M$ for the cofree counital conilpotent coalgebra on $M$, i.e.\ the $\k$-module $\bigoplus_{n \ge 0} M^{\tensor_\k n}$ equipped with its canonical coalgebra structure.
  We similarly write $\rcoT_\k M \defeq \bigoplus_{n \ge 1} M^{\tensor_\k n}$ for the cofree non-counital conilpotent coalgebra on $M$.
  We omit the subscripts when $\k = \QQ$.
\end{notation}

\begin{definition} \label{def:Ainf}
  An \emph{$\Ainf$-algebra over $\k$} consists of a $\k$-module $R$ together with a differential $\mu^R$ on $\coT_\k (\shift R)$ of degree $1$ such that its linear part $\mu^R_1 \colon \shift R \to \shift R$ is the differential of $\shift R$ and $\B(R) \defeq (\coT_\k (\shift R), \mu^R)$ is a coaugmented coalgebra in $\k$-modules (with coaugmentation given by the canonical inclusion $\eta \colon \k \to \coT_\k (\shift R)$).
  A \emph{morphism of $\Ainf$-algebras} $f \colon R \to S$ is a map $\B(f) \colon \B(R) \to \B(S)$ of coaugmented coalgebras in $\k$-modules.
  We denote the category of $\Ainf$-algebras over $\k$ by $\AinfAlg$.
  We furthermore write $\rB(R) \defeq \quot{\B(R)}{\eta}$ for the non-counital coalgebra $\rcoT_\k(\shift R)$, equipped with the differential induced by $\mu_R$.
\end{definition}

Note that our $\Ainf$-algebras are in general non-unital.

\begin{observation}
  An $\Ainf$-algebra structure on a $\k$-module $R$ equivalently consists of a collection $(\mu^R_n)_{n \ge 1}$ of $\k$-linear maps of degree $1$
  \[ \mu^R_n \colon (\shift R)^{\otimes_\k n} \longto \shift R \]
  such that $\mu^R_1$ is the differential of $\shift R$ and the equation
  \begin{equation} \label{eq:Ainf}
    \sum_{r+s+t = n} \mu^R_{r+1+t} \circ ({\id^{\otimes r}} \otimes \mu^R_s \otimes \id^{\otimes t}) = 0
  \end{equation}
  holds for all $n \ge 1$.
  
  A morphism $f \colon R \to S$ is equivalently a collection $(f_n)_{n \ge 1}$ of $\k$-linear maps of degree $0$
  \[ f_n \colon (\shift R)^{\otimes_\k n} \longto \shift S\]
  such that
  \begin{equation} \label{eq:Ainf_map}
    \sum_{n_1 + \dots + n_i = n} \mu^{S}_i \circ (f_{n_1} \otimes \dots \otimes f_{n_i}) = \sum_{r+s+t = n} f_{r+1+t}\circ ({\id^{\otimes r}} \otimes \mu^R_s \otimes \id^{\otimes t})
  \end{equation}
  holds for all $n \ge 1$.
  The composite of two morphisms $f \colon R \to S$ and $g \colon S \to S'$ has structure maps
  \[ (g \circ f)_n  =  \sum_{n_1 + \dots + n_i = n} g_i \circ (f_{n_1} \otimes \dots \otimes f_{n_i}) \]
  and the identity $\id \colon R \to R$ has structure maps $\id_1 = \id$ and $\id_n = 0$ for $n \ge 2$.
\end{observation}

\begin{remark}
  Recall that $\shift(\blank)$ denotes a cohomological shift by $-1$, so that $\mu^R_n$ corresponds to a map $R^{\tensor_\k n} \to R$ of degree $2 - n$.
  There are different conventions in the literature for the precise form of this correspondence.
  For example Mescher \cite{Mes} uses maps $\bar \mu_n \colon R^{\otimes n} \to R$ corresponding to our $\mu_n$ via
  \[ \mu_n(\shift a_1 \otimes \dots \otimes \shift a_n) = \shift \bar \mu_n(a_1 \otimes \dots \otimes a_n) \]
  and Getzler--Jones \cite{GJ} and Lefèvre-Hasegawa \cite{Lef} use maps $m_n \colon R^{\otimes n} \to R$ corresponding to those via
  \[ m_n(a_1 \otimes \dots \otimes a_n) = (-1)^{\sum_i (n - i) (\deg {a_i} + 1)} \bar \mu_n(a_1 \otimes \dots \otimes a_n) \]
  or equivalently $\mu_n = \shift \circ m_n \circ (\shift[-1])^{\otimes n}$.
  See Tradler \cite[Proposition~1.4]{Tra} for a comparison.
\end{remark}

\begin{definition}
  We say that a morphism $f \colon R \to S$ of $\Ainf$-algebras over $\k$ is a \emph{quasi-isomorphism} when its linear part $f_1 \colon \shift R \to \shift S$ is a quasi-isomorphism.
  We say that $f$ is \emph{strict} when $f_n = 0$ for all $n \ge 2$.
\end{definition}

\begin{definition}
  An $\Ainf$-algebra $R$ over $\k$ is \emph{unital} if it is equipped with a distinguished cocycle $1 \in R^0$ such that
  \[ \mu_2(\shift 1 \otimes \shift a) = \shift a = (-1)^{\deg {a}} \mu_2(\shift a \otimes \shift 1) \]
  for all $a \in R$ and such that $\mu_n(x_1 \tensor \dots \tensor x_n) = 0$ whenever $n \ge 3$ and $x_i = \shift 1$ for some $i$.
  
  Let $R$ and $S$ be unital $\Ainf$-algebras over $\k$; a map $f \colon R \to S$ is \emph{unital} if
  \[ f_1(\shift 1) = \shift 1  \qquad \text{and} \qquad  f_n(x_1 \tensor \dots \tensor x_n) = 0 \]
  whenever $n \ge 2$ and $x_i = \shift 1$ for some $i$.
\end{definition}

\begin{observation} \label{obs:dga_as_Ainf}
  Let $\k \to R$ be a morphism of dgas.
  Then $R$ becomes a unital $\Ainf$-algebra over $\k$ by defining
  \[ \mu_1(\shift a) \defeq - \shift d(a), \qquad \mu_2(\shift a \tensor \shift b) \defeq (-1)^{\deg a} \shift (a \cdot b), \]
  and $\mu_n \defeq 0$ for $n \ge 3$.
  A map $f \colon R \to S$ of dgas over $\k$ yields a strict, unital map of unital $\Ainf$-algebras over $\k$ by defining $f_1(\shift a) \defeq \shift f(a)$ and $f_n \defeq 0$ for $n > 1$.
\end{observation}

Note that this in particular applies to the terminal morphism from $\k$ to the trivial dga $0$, so that $0$ becomes a unital $\Ainf$-algebra over $\k$ with $\B(0) \iso \k$.
Also note that, for a unital $\Ainf$-algebra $R$, the map of $\k$-modules $\eta \colon \shift \k \to \shift R$ determined by $\eta(\shift 1) = \shift 1$ is a strict unital map of $\Ainf$-algebras.

\subsection{\texorpdfstring{$\Ainf$}{A∞}-bimodules}

We now introduce (bi)modules over an $\Ainf$-algebra, following Getzler--Jones \cite[§3]{GJ}; also see \cite[§3]{Kel}.

\begin{definition}
  Let $R$ and $S$ be $\Ainf$-algebras over $\k$.
  An \emph{$R$-$S$-bimodule} $M$ consists of a $\k$-module $M$ together with a differential $\mu^M$ on $\coT_\k (\shift R) \tensor_\k M \tensor_\k \coT_\k (\shift S)$ of degree $1$, such that its linear part $\mu^M_{0,0} \colon M \to M$ is the differential of $M$ and
  \[ \B(R, M, S) \defeq (\coT_\k (\shift R) \tensor_\k M \tensor_\k \coT_\k (\shift S), \mu^M) \]
  is a $\B(R)$-$\B(S)$-bicomodule in $\k$-modules.
  A \emph{morphism} of $R$-$S$-modules $f \colon M \to N$ of degree $d$ is a map of $\B(R)$-$\B(S)$-bicomodules in $\k$-modules $\B(f) \colon \B(R, M, S) \to \B(R, N, S)$ of degree $d$.
  We denote the category of $R$-$S$-bimodules by $\BMod{R}{S}$; it is enriched in graded vector spaces.
  A \emph{left $R$-module} is an $R$-$0$-bimodule, and a \emph{left $S$-module} is a $0$-$S$-bimodule; we write $\LMod{R} \defeq \BMod{R}{0}$ and $\Mod{S} \defeq \BMod{0}{S}$ for the categories of left and right modules, respectively.
\end{definition}

Note that a $0$-$0$-bimodule is the same thing as a $\k$-module.

\begin{observation}
  Let $R$ and $S$ be $\Ainf$-algebras over $\k$.
  An $R$-$S$-bimodule $M$ equivalently consists of a collection $(\mu^M_{l,r})_{l,r \ge 0}$ of $\k$-linear maps of degree $1$
  \[ \mu^M_{l,r} \colon (\shift R)^{\otimes_\k l} \otimes_\k M \otimes_\k (\shift S)^{\otimes_\k r} \longto M \]
  such that $\mu^M_{0,0}$ is the differential of $M$ and the equation
  \begin{align} \label{eq:bimodule}
    \begin{split}
      0 &= \sum_{l_1 + l_2 + l_3 = l} \mu^M_{l_1 + 1 + l_3, r} \circ ({\id^{\otimes l_1}} \otimes \mu^R_{l_2} \otimes \id^{\otimes l_3 + 1 + r}) \\
      &+ \sum_{\substack{l_1 + l_2 = l \\ r_1 + r_2 = r}} \mu^M_{l_1, r_2} \circ ({\id^{\otimes l_1}} \otimes \mu^M_{l_2, r_1} \otimes \id^{\otimes r_2}) \\
      &+ \sum_{r_1 + r_2 + r_3 = r} \mu^M_{l, r_1 + 1 + r_3} \circ ({\id^{\otimes l + 1 + r_1}} \otimes \mu^{S}_{r_2} \otimes \id^{\otimes r_3})
    \end{split}
  \end{align}
  holds for all $l,r \ge 0$.
  An analogous statement holds for left and right modules by omitting all terms with $r \neq 0$ and $l \neq 0$, respectively.
  In particular an $R$-$S$-bimodule has an underlying left $R$-module and an underlying right $S$-module by forgetting all structure maps with $r \neq 0$ or $l \neq 0$, respectively.
  
  A morphism of $R$-$S$-bimodules $f \colon M \to M'$ of degree $d$ equivalently consists of a collection $(f_{l,r})_{l,r \ge 0}$ of $\k$-linear maps of degree $d$
  \[ f_{l,r} \colon (\shift R)^{\otimes_\k l} \otimes_\k M \otimes_\k (\shift S)^{\otimes_\k r} \longto M' \]
  such that the equation
  \begin{equation} \label{eq:bimodule_map}
    \begin{multlined}
      (-1)^d \sum_{\substack{l_1 + l_2 = l \\ r_1 + r_2 = r}} \mu^{M'}_{l_1, l_2} \circ ({\id^{\otimes l_1}} \otimes f_{l_2, r_1} \otimes \id^{\otimes r_2}) \\
      \hspace{5em}
      \begin{aligned}
        &= \sum_{l_1 + l_2 + l_3 = l} f_{l_1 + 1 + l_3, r} \circ ({\id^{\otimes l_1}} \otimes \mu^R_{l_2} \otimes \id^{\otimes l_3 + 1 + r}) \\
        &+ \sum_{\substack{l_1 + l_2 = l \\ r_1 + r_2 = r}} f_{l_1, r_2} \circ ({\id^{\otimes l_1}} \otimes \mu^M_{l_2, r_1} \otimes \id^{\otimes r_2}) \\
        &+ \sum_{r_1 + r_2 + r_3 = r} f_{l, r_1 + 1 + r_3} \circ ({\id^{\otimes l + 1 + r_1}} \otimes \mu^{S}_{r_2} \otimes \id^{\otimes r_3})
      \end{aligned}
    \end{multlined}
  \end{equation}
  holds for all $l,r \ge 0$.
  The composite of $f$ and a morphism $f' \colon M' \to M''$ has structure maps
  \[ (f' \circ f)_{l,r}  =  \sum_{\substack{l_1 + l_2 = l \\ r_1 + r_2 = r}} f'_{l_1, r_2} \circ ({\id^{\otimes l_1}} \otimes f_{l_2, r_1} \otimes \id^{\otimes r_2}) \]
  and the identity $\id \colon M \to M$ has structure maps $\id_{0,0} \defeq \id$ and $\id_{l,r} \defeq 0$ for $(l,r) \neq (0,0)$.
\end{observation}

\begin{definition}
  Let $R$ and $S$ be $\Ainf$-algebras over $\k$.
  We say that a morphism $f \colon M \to N$ of $R$-$S$-bimodules is a \emph{quasi-isomorphism} when its linear part $f_{0,0} \colon M \to N$ is a quasi-isomorphism.
  We say that $f$ is \emph{strict} if $f_{l,r} = 0$ whenever $(l, r) \neq (0, 0)$.
  The \emph{homotopy category} of $R$-$S$-bimodules is the localization of $\BMod{R}{S}$ at the class of quasi-isomorphisms.
\end{definition}

\begin{definition}
  Let $R$ and $S$ be unital $\Ainf$-algebras over $\k$.
  An $R$-$S$-bimodule $M$ is \emph{unital} if, for all $m \in M$,
  \begin{gather*}
    \mu^M_{1,0}(\shift 1 \otimes m) = m = (-1)^{\deg {m}} \mu^M_{0,1}(m \otimes \shift 1)
    \shortintertext{and}
    \mu^M_{l,r}(x_1 \otimes \dots \otimes x_l \otimes m \otimes y_1 \otimes \dots \otimes y_r) = 0
  \end{gather*}
  whenever $l + r > 1$ and $x_i = \shift 1$ or $y_i = \shift 1$ for some $i$.
\end{definition}

\begin{observation} \label{obs:algebra_is_bimodule}
  Given an $\Ainf$-algebra $R$ over $\k$, the $\k$-module $\shift R$ is canonically an $R$-$R$-bimodule with structure maps $\mu^{\shift R}_{l,r} \defeq \mu^R_{l+1+r}$, and hence also a left $R$-module and a right $R$-module.
  If $R$ is unital, then $\shift R$ is unital as a (bi)module.
\end{observation}

\begin{observation} \label{obs:dgm_as_Ainf}
  Let $\k \to R$ and $\k \to S$ be maps of dgas, and $M$ an $(R \tensor_\k \opmod{S})$-module.
  Then $M$ becomes a unital $R$-$S$-bimodule (with $R$ and $S$ considered as $\Ainf$-algebras over $\k$ via \cref{obs:dga_as_Ainf}) by defining, for $x \in R$, $m \in M$, and $y \in S$,
  \[ \mu^M_{0,0}(m) \defeq d(m), \qquad \mu^M_{1,0}(\shift x \tensor m) \defeq x \cdot m, \qquad \mu^M_{0,1}(m \tensor \shift y) \defeq (-1)^{\deg m} m \cdot y, \]
  and $\mu^M_{l,r} \defeq 0$ when $l + r > 1$.
  An $(R \tensor_\k \opmod{S})$-module map $f \colon M \to M'$ of degree $d$ yields a strict $R$-$S$-bimodule map of degree $d$ by defining $f_{0,0} \defeq f$ and $f_{l,r} \defeq 0$ for $l + r > 0$.
\end{observation}

\begin{lemma} \label{lemma:upsilon_algebra_map}
  Let $R$ be an $\Ainf$-algebra over $\k$, and $M$ a left $R$-module.
  Then there is a map $\upsilon \colon R \to \End_\k(M, M)$ of $\Ainf$-algebras over $\k$ (where the dga $\End_\k(M, M)$ is considered as an $\Ainf$-algebra via \cref{obs:dga_as_Ainf}) with structure maps $\upsilon_n  \colon  (\shift R)^{\tensor n} \to \shift \End_\k(M, M)$ given by $x_1 \tensor \dots \tensor x_n \mapsto \shift \mu^M_n(x_1, \dots, x_n, \blank)$.
\end{lemma}

\begin{proof}
  The $n$-th $\Ainf$-algebra map equation \eqref{eq:Ainf_map} for $\upsilon$, evaluated on $\vec x \defeq x_1 \tensor \dots \tensor x_n$, reads
  \begin{multline*}
    \sum_{\substack{n_1 + n_2 = n \\ n_i > 0}} (-1)^{\deg{\vec x_{n_1}} + 1} \shift \mu^M_{n_1} (\vec x_{n_1} \tensor \blank) \after \mu^M_{n_2}(\vec x_{n_2} \tensor \blank) - \shift \mu^M_0 \after \mu^M_n(\vec x \tensor \blank) - (-1)^{\deg{\vec x}} \shift \mu^M_n(\vec x \tensor \blank) \after \mu^M_0 \\
    = \sum_{n_1 + n_2 + n_3 = n} (-1)^{\deg{\vec x_{n_1}}} \shift \mu^M_{n_1+1+n_3} \bigl( \vec x_{n_1} \tensor \mu^R_{n_2}(\vec x_{n_2}) \tensor \vec x_{n_3} \tensor \blank \bigr)
  \end{multline*}
  where $\vec x_{n_i} \defeq x_{n_1 + \dots + n_{i-1} + 1} \tensor \dots \tensor x_{n_1 + \dots + n_i}$.
  This is precisely the $n$-th left module equation \eqref{eq:bimodule} for $M$.
\end{proof}

\subsubsection*{Restriction of scalars}

We conclude this subsection by considering how to restrict bimodule structures along morphisms of $\Ainf$-algebras.

\begin{definition}
  Let $f \colon R \to R'$ and $g \colon S \to S'$ be two morphisms of $\Ainf$-algebras over $\k$.
  We denote by $(f, g)^* \colon \BMod{R'}{S'} \to \BMod{R}{S}$ the functor given by sending an $R'$-$S'$-bimodule $M'$ to $M'$ equipped with the structure maps of
  \[ \B(R, M', S) \defeq \B(R) \cotensor^{\B(R')} \B(R', M', S') \cotensor^{\B(S')} \B(S) \]
  considered as a $\B(R)$-$\B(S)$-bicomodule in $\k$-modules.
\end{definition}

\begin{observation} \label{obs:restrict_bimodule}
  Let $f \colon R \to R'$ and $g \colon S \to S'$ be two morphisms of $\Ainf$-algebras over $\k$.
  For an $R'$-$S'$-module $M'$, the structure maps of $M \defeq (f, g)^*(M')$ are given by
  \begin{align*}
    \mu^{M}_{l,r} &\colon \shift R^{\otimes_\k l} \otimes_\k M' \otimes_\k \shift S^{\otimes_\k r} \longto M' \\
    \mu^{M}_{l,r} &= \sum_{\substack{l_1 + \dots + l_i = l \\ r_1 + \dots + r_j = r}} \mu^{M'}_{i,j} \circ ( f_{l_1} \otimes \dots \otimes f_{l_i} \otimes \id_{M'} \otimes g_{r_1} \otimes \dots \otimes g_{r_j} ).
  \end{align*}
  where $\mu^{M'}_{i,j}$ are the structure maps of $M'$.

  For a map $h \colon M' \to N'$ of $R'$-$S'$-bimodules, the structure maps of $h \defeq (f,g)^*(h')$ are given by
  \begin{align*}
    h_{l,r} &\colon (\shift R)^{\otimes_\k l} \otimes_\k M' \otimes_\k (\shift R)^{\otimes_\k r} \longto N' \\
    h_{l,r} &= \sum_{\substack{l_1 + \dots + l_i = l \\ r_1 + \dots + r_j = r}} h'_{i,j} \circ ( f_{l_1} \otimes \dots \otimes f_{l_i} \otimes \id_{M'} \otimes g_{r_1} \otimes \dots \otimes g_{r_j} )
  \end{align*}
  where $h'_{i,j}$ are the structure maps of $h'$.
\end{observation}

\begin{lemma} \label{lemma:algebra_map_is_bimodule_map}
  Let $f \colon R \to S$ be a map of $\Ainf$-algebras over $\k$.
  Then there is a map of $R$-$R$-bimodules $f' \colon \shift R \to (f, f)^*(\shift S)$ of degree $0$ with structure maps $f'_{l,r} \defeq f_{l+1+r}$.
  It is a quasi-isomorphism if and only if $f$ is a quasi-isomorphism.
\end{lemma}

\begin{proof}
  The $(l, r)$ bimodule map equation \eqref{eq:bimodule_map} for $f'$ is
  \begin{multline*}
    \sum_{\substack{l_1 + \dots + l_i = l \\ r_1 + \dots + r_j = r}} \mu^{S}_{(i - 1) + 1 + (j - 1)} \circ (f_{l_1} \otimes \dots \otimes f_{l_{i-1}} \otimes f_{l_i + 1 + r_1} \otimes f_{r_2} \otimes \dots \otimes f_{r_j}) \\
    \begin{aligned}
      &= \sum_{l_1 + l_2 + l_3 = l} f_{(l_1 + 1 + l_3) + 1 + r} \circ ({\id^{\otimes l_1}} \otimes \mu^R_{l_2} \otimes \id^{\otimes l_3 + 1 + r}) \\
      &+ \sum_{\substack{l_1 + l_2 = l \\ r_1 + r_2 = r}} f_{l_1 + 1 + r_2} \circ ({\id^{\otimes l_1}} \otimes \mu^R_{l_2 + 1 + r_1} \otimes \id^{\otimes r_2}) \\
      &+ \sum_{r_1 + r_2 + r_3 = r} f_{l + 1 + (r_1 + 1 + r_3)} \circ ({\id^{\otimes l + 1 + r_1}} \otimes \mu^{R}_{r_2} \otimes \id^{\otimes r_3})
    \end{aligned}
  \end{multline*}
  which is precisely the $(l+1+r)$-th $\Ainf$-algebra map equation \eqref{eq:Ainf_map} for $f$.
  Since $f'_{0,0} = f_1$, it is clear that $f'$ is a quasi-isomorphism if and only if $f$ is.
\end{proof}

\subsection{\texorpdfstring{$\Cinf$}{C∞}-algebras and symmetric bimodules}

We now define $\Cinf$-algebras, which are a (strictly) commutative analog of $\Ainf$-algebras.
The dual notion, of a $\Cinf$-coalgebra, first appeared in work of Smirnov \cite[Definition~3]{Smi}; the definition we use is due to Kadeishvili \cite{Kad}.
For a short summary, see \cite[§3]{Kad09}.

\begin{definition} \label{def:shuffle}
  Let $p$ and $q$ be non-negative integers.
  A \emph{$(p, q)$-shuffle} is a permutation $\sigma \in \Sigma_{p+q}$ such that the inequalities
  \[ \sigma(1) < \sigma(2) < \dots < \sigma(p)  \qquad \text{and} \qquad  \sigma(p + 1) < \sigma(p + 2) < \dots < \sigma(p + q) \]
  hold.
  We denote by $\Sh p q \subseteq \Sigma_{p + q}$ the set of $(p, q)$-shuffles.
\end{definition}

\begin{definition}
  A \emph{$\Cinf$-algebra over $\k$} is an $\Ainf$-algebra $R$ over $\k$ such that its structure maps vanish on sums of shuffles, i.e.\ for all $p,q \ge 1$ holds
  \[ \sum_{\sigma \in \Sh p q} \mu^R_{p+q} \after \sigma = 0 \]
  where $\sigma$ acts from the left, i.e.\ $\sigma \act (x_1 \tensor \dots \tensor x_{p+q}) \defeq \pm x_{\inv \sigma(1)} \tensor \dots \tensor x_{\inv \sigma(p+q)}$ with the sign determined by the Koszul rule.
  A map $R \to S$ of $\Cinf$-algebras over $\k$ is a map $f$ of the underlying $\Ainf$-algebras over $\k$ such that
  \[ \sum_{\sigma \in \Sh p q} f_{p+q} \after \sigma = 0 \]
  for all $p, q \ge 1$.
\end{definition}

\begin{observation} \label{obs:cdga_as_Cinf}
  Let $\k \to R$ be a morphism of cdgas.
  Then $R$, considered as an $\Ainf$-algebra over $\k$ via \cref{obs:dga_as_Ainf}, is a $\Cinf$-algebra over $\k$.
  Similarly a map $f \colon R \to S$ of cdgas over $\k$ yields a map of $\Cinf$-algebra over $\k$.
\end{observation}

We now introduce an $\Ainf$-analog of the notion of a symmetric bimodule over an algebra.%
\footnote{Note that our definition differs from the balanced bimodules of Markl \cite{Mar}, which are called $\Cinf$-bimodules by Ginot \cite[Definition~2.5]{Gin}.}
We will see below that a $\Cinf$-algebra is symmetric when considered as a bimodule over itself.

\begin{notation} \label{not:cyclic}
  Let $A_1, \dots, A_i$ be $\k$-modules.
  We write
  \[ \cyclic i  \colon  A_1 \otimes_\k \dots \otimes_\k A_i  \xlongto{\iso}  A_i \otimes_\k A_1 \otimes_\k \dots \otimes_\k A_{i - 1} \]
  for the canonical isomorphism permuting the factors.
\end{notation}

\begin{definition} \label{def:symmetric_bimod}
  Let $R$ be an $\Ainf$-algebra over $\k$.
  An $R$-$R$-bimodule $M$ is \emph{symmetric} if the structure maps vanish on sums of cyclic permutations, i.e.\ if for all $n \ge 1$
  \[ \sum_{i = 0}^n \mu^M_{i, n-i} \after \cyclic{n+1}^{\circ i}  =  0 \]
  as maps $M \tensor_\k (\shift R)^{\tensor_\k n} \to M$.
  A map $f \colon M \to N$ of $R$-$R$-bimodules is \emph{symmetric} if for all $n \ge 1$
  \[ \sum_{i = 0}^n f_{i, n-i} \after \cyclic{n+1}^{\circ i}  =  0 \]
  as maps $M \tensor_\k (\shift R)^{\tensor_\k n} \to N$.
\end{definition}

Note that the composite of two symmetric maps of $R$-$R$-bimodules is again symmetric.

\begin{observation} \label{obs:dgm_as_symmetric}
  Let $\k \to R$ be a map of dgas, and $M$ an $(R \tensor_\k \opmod{R})$-module.
  If $M$ is symmetric in the sense that $r m = (-1)^{\deg r \deg m} m r$, then it is symmetric when considered as an $R$-$R$-bimodule via \cref{obs:dgm_as_Ainf}.
  Furthermore any $(R \tensor_\k \opmod{R})$-module map $f \colon M \to M'$ yields a symmetric map of $R$-$R$-bimodules.
\end{observation}

\begin{observation}
  Let $f \colon R \to S$ be a map of $\Ainf$-algebras over $\k$.
  It follows from \cref{obs:restrict_bimodule} that the restriction functor $(f, f)^* \colon \BMod{S}{S} \to \BMod{R}{R}$ preserves symmetric bimodules and symmetric maps of bimodules.
\end{observation}

\begin{lemma} \label{lemma:Cinf_is_symm}
  For a $\Cinf$-algebra $R$ over $\k$, the $R$-$R$-bimodule $\shift R$ of \cref{obs:algebra_is_bimodule} is symmetric.
  For a map $f \colon R \to S$ of $\Cinf$-algebras, the $R$-$R$-bimodule map $\shift R \to (f, f)^*(\shift S)$ of \cref{lemma:algebra_map_is_bimodule_map} is symmetric.
\end{lemma}

\begin{proof}
  This follows immediately from \cref{lemma:cyclic_shuffle} below.
\end{proof}

\begin{lemma} \label{lemma:cyclic_shuffle}
  Let $\mathbb{K}$ be a field and $n \ge 2$.
  Then the sum $c_n$ of all cyclic permutations of $n$ letters lies in the $\Sigma_n$-subrepresentation $S_n \subseteq \mathbb{K}[\Sigma_n]$ spanned by $\sum_{\sigma \in \Sh p q} \sigma$ for all $p, q \ge 1$ with $p + q = n$.
\end{lemma}

\begin{proof}
  It is well-known (see e.g.\ \cite[Theorem~1.3.6]{LV}) that $S_n$ is the kernel of the surjective map
  \[ \mathbb{K}[\Sigma_n]  \iso  \dual{\mathbb{K}[\Sigma_n]}  =  \dual{\Ass(n)}  \xlongto{\dual{\kappa_n}}  \dual{\Lie(n)} \]
  where the isomorphism is the one associated to the canonical basis and $\kappa_n \colon \Lie(n) \to \Ass(n)$ is the canonical map from the $\mathbb{K}$-linear Lie operad to the $\mathbb{K}$-linear associative operad.
  Hence we want to prove that $c_n$ is mapped to $0$ under this composite, or equivalently that $\dual{c_n} \after \kappa_n = 0$.
  
  To this end, we consider the following inductively defined Lie words in the letters $\quot{\ZZ}{n} = \set {0, \dots, n - 1}$
  \[ w_\sigma^1  \defeq  \sigma(0)  \qquad \text{and} \qquad  w_\sigma^{i+1}  \defeq  [ w_\sigma^i, \sigma(i)] \]
  where $1 \le i < n$ and $\sigma$ is a permutation of $\quot{\ZZ}{n}$.
  Note that $\Lie(n)$ is spanned by the elements $w_\sigma^n$, so that it is enough to show that $\dual{c_n} \after \kappa_n$ vanishes on those.
  Also note that $\Ass(n)$ has a basis consisting of permutations of the set $\quot{\ZZ}{n}$, which we will write as tuples.
  We will prove by induction that for all $\sigma$ and $1 \le i \le n - 1$, the element $\kappa_i(w_\sigma^i)$ has at most one component of the form $\lambda_i \cdot (x_i, x_i + 1, \dots, x_i + i - 1)$ for some $\lambda_i \neq 0$ and $x_i \in \quot{\ZZ}{n}$.
  The claim is clear for $i = 1$.
  Now assume it is proven for some $1 \le i < n - 1$.
  If $\kappa_i(w_\sigma^i)$ has no component of the specified form, then neither does $\kappa_{i+1}(w_\sigma^{i+1})$.
  Now assume that $\kappa_i(w_\sigma^i)$ does have such a component for some $x_i$.
  If $\sigma(i)$ is not equal to $x_i - 1$ or $x_i + i$, then $\kappa_{i+1}(w_\sigma^{i+1})$ again has no such component; otherwise it has precisely one (here we use that $i < n - 1$ so that $x_i - 1 \neq x_i + i$).
  This completes the induction, so that we have proven the claim for $i = n - 1$.
  Now, if $\kappa_{n-1}(w_\sigma^{n-1})$ has no component of the specified form, then $\kappa_n(w_\sigma^n)$ has no cyclic components and we are done.
  Otherwise $\kappa_{n-1}(w_\sigma^{n-1})$ does have such a component for some $\lambda_{n-1}$ and $x_{n-1}$.
  In that case $\sigma(n-1) = x_{n-1} - 1 = x_{n-1} + n - 1$, so that $\kappa_n(w_\sigma^n)$ has exactly the two cyclic components
  \[ \lambda_{n-1} \cdot (x_{n-1}, x_{n-1} + 1, \dots, x_{n-1} + n - 1) - \lambda_{n-1} \cdot (x_{n-1} - 1, x_{n-1}, \dots, x_{n-1} + n - 2) \]
  which have opposite signs and hence cancel each other upon applying $\dual{c_n}$.
  This completes the proof.
\end{proof}

\subsection{Tensor products and the hom-functor}

We begin this subsection by introducing the tensor product of bimodules over $\Ainf$-algebras, following Getzler--Jones \cite[§3]{GJ}; also see \cite[§4.1.1]{Lef}.

\begin{definition}
  Let $R$, $S$, and $T$ be $\Ainf$-algebras over $\k$.
  We denote by
  \[ \blank \inftensor_S \blank  \colon  \BMod{R}{S} \times \BMod{S}{T}  \longto  \BMod{R}{T} \]
  the functor given by sending $(M, N)$ to the $R$-$T$-bimodule corresponding to
  \[ \B \bigl( R, M \inftensor_S N, T \bigr)  \defeq  \B(R, M, S) \cotensor^{\B(S)} \B(S, N, T) \]
  considered as a $\B(R)$-$\B(T)$-bicomodule in $\k$-modules.
  If $S = 0$, we simply write $M \otimes_\k N \defeq M \inftensor_0 N$.
\end{definition}

\begin{observation} \label{obs:comp_description_infty_tensor}
  Let $R$, $S$, and $T$ be $\Ainf$-algebras over $\k$, and let $M$ be an $R$-$S$-bimodule and $N$ an $S$-$T$-bimodule.
  Then the underlying $\k$-module of $M \inftensor_S N$ is $\bigoplus_{n \ge 0} M \otimes_\k (\shift S)^{\otimes_\k n} \otimes_\k N$.
  Its structure maps are given, on the $n$-th summand, by
  \begin{align*}
    \mu_{l, r} =&
      \begin{cases*}
        \sum_{n_1 + n_2 = n} \mu^M_{l,n_1} \otimes \id^{\otimes n_2 + 1}, & if $r = 0$ \\
        0, & else
      \end{cases*} \\
      +&
      \begin{cases*}
        \sum_{n_1 + n_2 + n_3 = n} {\id^{\otimes 1 + n_1}} \otimes \mu^S_{n_2} \otimes \id^{\otimes n_3 + 1}, & if $l = r = 0$ \\
        0, & else
      \end{cases*} \\
      +&
      \begin{cases*}
        \sum_{n_1 + n_2 = n} {\id^{\otimes 1 + n_1}} \otimes \mu^N_{n_2,r}, & if $l = 0$ \\
        0, & else
      \end{cases*}
  \end{align*}
  where $\mu^M$, $\mu^R$, and $\mu^N$ are the structure maps of $M$, $R$, and $N$, respectively.
  
  For a map $f \colon M \to M'$ of $R$-$S$-bimodules and a map $g \colon N \to N'$ of $S$-$T$-bimodules, the induced map $(f,g)_* \colon M \inftensor_S N \to M' \inftensor_S N'$ has the structure maps given on the $n$-th summand by
  \[ \bigl( (f,g)_* \bigr)_{l, r}  =  \sum_{n_1 + n_2 + n_3 = n} f_{l,n_1} \tensor {\id^{\otimes n_2}} \tensor g_{n_3,r} \]
  where $f_{l,r}$ and $g_{l,r}$ are the respective structure maps of $f$ and $g$.
  
  In the case $S = 0$, the underlying $\k$-module of $M \tensor_\k N = M \inftensor_S N$ is the $\k$-module $M \tensor_\k N$ (hence the notation).
  The structure maps of $M \tensor_\k N$ are given by
  \begin{align*}
    \mu_{l,r} &\colon (\shift R)^{\otimes_\k l} \otimes_\k M \otimes_\k N \otimes_\k (\shift T)^{\otimes_\k r}  \longto  M \otimes_\k N \\
    \mu_{l,r} &=
    \begin{lrcases*}
      \mu^M_l \tensor \id, & if $r = 0$ \\
      0, & else
    \end{lrcases*}
    +
    \begin{lrcases*}
      \id \tensor \mu^N_r, & if $l = 0$ \\
      0, & else
    \end{lrcases*}
  \end{align*}
  where $\mu^M_l$ and $\mu^N_r$ are the structure maps of $M$ and $N$, respectively.
\end{observation}

In the case that $S$ is a dga, the construction $M \inftensor_S N$ is the derived tensor product of $M$ and $N$ as follows.

\begin{observation} \label{obs:inftensor_dga}
  Let $\k \to S$ be a map of dgas, let $R$ and $T$ be $\Ainf$-algebras over $\k$, and let $M$ be an $R$-$S$-bimodule such that $\mu^M_{l,r} = 0$ except if $r = 0$ or $(l, r) = (0, 1)$, and $N$ an $S$-$T$-bimodule such that $\mu^N_{l,r} = 0$ except if $l = 0$ or $(l, r) = (1, 0)$.
  Then there is a canonical isomorphism of $\k$-modules
  \[ M \inftensor_S N \iso M \tensor_S \B_\k(S, S, S) \tensor_S N \]
  where on the right-hand side we consider $M$ and $N$ as modules over the dga $S$ via $\mu^M_{0,1}$ and $\mu^N_{1,0}$.
  The map $\epsilon \colon \B_\k(S, S, S) \to R$ induces a strict map $M \inftensor_S N \to M \tensor_S N$ of $R$-$T$-bimodules.
  It is a quasi-isomorphism when $S$ is cofibrant as a $\k$-module and either $M$ is cofibrant as an $S$-module or $N$ is cofibrant as an $\opmod{S}$-module (using \cite[\Ilemmabarcomplex{} and \Ilemmacofibrantflat]{I}).
\end{observation}

We will require the following interaction between the tensor product and restriction functors.

\begin{observation} \label{obs:inftensor_restrict}
  Let $f \colon R \to R'$, $g \colon S \to S'$, and $h \colon T \to T'$ be maps of $\Ainf$-algebras over $\k$, $M$ an $R'$-$S'$-bimodule, and $N$ an $S'$-$T'$-bimodule.
  Then there is a strict map of $R$-$T$-bimodules, natural in $M$ and $N$,
  \[ g_*  \colon  (f, g)^*(M) \inftensor_S (g, h)^*(N)  \longto  (f, h)^* (M \inftensor_{S'} N) \]
  associated to the canonical map
  \[
  \begin{tikzcd}[row sep = 15]
    \bigl( \B(R) \cotensor^{\B(R')} \B(R', M, S') \cotensor^{\B(S')} \B(S) \bigr) \cotensor^{\B(S)} \bigl( \B(S) \cotensor^{\B(S')} \B(S', N, T') \cotensor^{\B(T')} \B(T) \bigr) \dar \\
    \B(R) \cotensor^{\B(R')} \bigl( \B(R', M, S') \cotensor^{\B(S')} \B(S', N, T') \bigr) \cotensor^{\B(T')} \B(T)
  \end{tikzcd}
  \]
  of $\B(R)$-$\B(T)$-bicomodules.
  Its underlying map is, on the $n$-th summand, given by
  \[ \sum_{n_1 + \dots + n_i = n} \id_M \otimes g_{n_1} \otimes \dots \otimes g_{n_i} \otimes \id_N \]
  where the $g_m$ are the structure maps of $g$.
\end{observation}

%

\subsubsection*{The hom-functor}

We now recall the construction of the following special case of the hom-functor of modules over $\Ainf$-algebras
\[ \Hom_\k(\blank, \blank)  \colon  \opcat{\LMod{S}} \times \LMod{R}  \longto  \BMod{R}{S} \]
by exhibiting it as right adjoint to the tensor product.
See \cite[§4.1.1]{Lef} for a more general and extensive treatment.

\begin{lemma} \label{lemma:Ainf_Hom}
  Let $R$ and $S$ be $\Ainf$-algebras over $\k$ and $M$ a left $S$-module.
  Then the functor $\blank \inftensor_S M \colon \BMod{R}{S} \to \LMod{R}$ has a right adjoint (as functors enriched in graded vector spaces).
  Its underlying $\k$-module is given by $\Hom_\k(M, \blank)$.
\end{lemma}

\begin{proof}
  It is enough to prove that, for each left $R$-module $N$, the $\k$-module $\Hom_\k(M, N)$ can be equipped with an $R$-$S$-bimodule structure such that there is an isomorphism of graded vector spaces
  \begin{equation} \label{eq:Hom_tensorinf_adjunction}
    \Phi \colon \LMod{R}(T \inftensor_S M, N)  \xlongto{\iso}  \BMod{R}{S} \bigl( T, \Hom_\k(M, N) \bigr)
  \end{equation}
  that is natural in $T \in \BMod{R}{S}$ (see e.g.\ \cite[§IV.1, Corollary~2]{Mac}).
  We define the $R$-$S$-bimodule structure maps of $\Hom_\k(M, N)$ to be the maps
  \[ \mu_{l,r} \colon (\shift R)^{\otimes_\k l} \otimes_\k \Hom_\k(M, N) \otimes_\k (\shift S)^{\otimes_\k r}  \longto  \Hom_\k(M, N) \]
  adjoint to the maps
  \[ \hat \mu_{l,r} \colon (\shift R)^{\otimes_\k l} \otimes_\k \Hom_\k(M, N) \otimes_\k (\shift S)^{\otimes_\k r} \otimes_\k M  \longto  N \]
  given by
  \[ \hat \mu_{l,r} \defeq
  \begin{lrcases*}
    \mu^N_l \circ ({\id^{\otimes l}} \otimes \ev), & if $r = 0$ \\
    0, & else
  \end{lrcases*}
  -
  \begin{lrcases*}
    {\ev} \circ (\id \tensor \mu^M_r), & if $l = 0$ \\
    0, & else
  \end{lrcases*}
  \]
  where $\mu^M_l$ and $\mu^N_r$ are the structure maps of $M$ and $N$, respectively.
  We begin by proving that this indeed defines an $R$-$S$-bimodule.
  
  For $l = r = 0$ the bimodule equation \eqref{eq:bimodule} is $\mu_{0,0}^2=0$ which holds as $\mu_{0,0}$ agrees with the usual differential of $\Hom_\k(M,N)$.
  For $l, r > 0$ there are only two non-trivial terms in the bimodule equation and we have to show that
  \[ \mu_{l,0} \circ (\id^{\otimes l} \otimes \mu_{0,r}) + \mu_{0,r} \circ (\mu_{l,0} \otimes \id^{\otimes r}) = 0 \]
  which follows from unwinding the definitions.
  For $l > 0$ and $r = 0$, the adjoint of the bimodule equation for $\Hom_\k(M, N)$ is the left module equation for $N$ precomposed with $\ev$; for $l = 0$ and $r > 0$, the adjoint is the left module equation for $M$ postcomposed with $-\ev$.
  
  We now construct the isomorphism $\Phi$ of \eqref{eq:Hom_tensorinf_adjunction}.
  Given a left $R$-module map $f \colon T \inftensor_S M \to N$ of degree $d$ we define $\Phi(f) \colon T \to \Hom_\k(M, N)$ to be the $R$-$S$-bimodule map of degree $d$ whose structure map $\Phi(f)_{l,r}$ is adjoint to the restriction
  \[ f_{l;r} \colon (\shift R)^{\otimes_\k l} \otimes_\k T \otimes_\k (\shift S)^{\otimes_\k r} \otimes_\k M  \longto  N \]
  of $f_l$.
  To see that $\Phi$ yields a well-defined isomorphism as in \eqref{eq:Hom_tensorinf_adjunction}, it is enough to prove that the maps $f_l$ fulfill the left $R$-module map equation if and only if the maps $\Phi(f)_{l,r}$ fulfill the $R$-$S$-bimodule map equation \eqref{eq:bimodule_map}.
  This follows directly from unwinding the definitions.
  Similarly one verifies that $\Phi$ is natural in $T$.
\end{proof}

\begin{observation} \label{obs:Hom_bifunctor}
  Since $\blank \inftensor_S \blank$ is a bifunctor, there is a unique way of making the right adjoint of \cref{lemma:Ainf_Hom} into a bifunctor
  \[ \Hom_\k(\blank, \blank)  \colon  \opcat{\LMod{S}} \times \LMod{R}  \longto  \BMod{R}{S} \]
  such that the adjunction isomorphism \eqref{eq:Hom_tensorinf_adjunction} becomes natural in both variables (see e.g.\ \cite[§IV.7, Theorem~3]{Mac}).
  In particular, the unit and counit transformations
  \[ \eta  \colon  T  \longto  \Hom_\k(M, T \inftensor_S M)  \qquad \text{and} \qquad  \ev \colon \Hom_\k(M, N) \inftensor_S M  \longto  N \]
  are natural in $M \in \LMod{S}$, $T \in \BMod{R}{S}$, and $N \in \LMod{R}$ in the canonical way.
  
  For a morphism $f \colon M \to M'$ of left $S$-modules and a left $R$-module $N$, the $R$-$S$-bimodule map $f^* \colon \Hom_\k(M', N) \to \Hom_\k(M, N)$ has the structure maps whose adjoints are the maps
  \begin{align*}
    \hat f^*_{l,r} &\colon (\shift R)^{\otimes_\k l} \otimes_\k \Hom_\k(M',N) \otimes_\k (\shift S)^{\otimes_\k r} \otimes M \longto N \\
    \hat f^*_{l,r} &= \begin{cases} \ev \circ ({\id} \otimes f_r), & l = 0 \\ 0, & \text{otherwise} \end{cases}
  \end{align*}
  where $f_r$ are the structure maps of $f$.
  If $f$ is a quasi-isomorphism and $M$ and $M'$ are cofibrant as $\k$-modules, then $f^*$ is a quasi-isomorphism as well.
  
  For a left $S$-module $M$ and a morphism $g \colon N \to N'$ of left $R$-modules, the $R$-$S$-bimodule map $g_* \colon \Hom_\k(M, N) \to \Hom_\k(M, N')$ has the structure maps whose adjoints are the maps
  \begin{align*}
    (\hat g_*)_{l,r} &\colon (\shift R)^{\otimes_\k l} \otimes_\k \Hom_\k(M, N) \otimes_\k (\shift R)^{\otimes_\k r} \otimes_\k M \longto N' \\
    (\hat g_*)_{l,r} &= \begin{cases} g_l \circ ({\id^{\tensor l}} \tensor \ev), & r = 0 \\ 0, & \text{otherwise} \end{cases}
  \end{align*}
  where $g_l$ are the structure maps of $g$.
  If $g$ is a quasi-isomorphism and $M$ is cofibrant as a $\k$-module, then $g_*$ is a quasi-isomorphism as well.
\end{observation}

\begin{observation} \label{obs:hom-tensor_unit}
  Let $R$ and $S$ be $\Ainf$-algebras over $\k$ and $M$ a left $S$-module.
  For an $R$-$S$-bimodule $T$, the structure maps $\eta_{l,r}$ of the unit $\eta \colon T \to \Hom_\k ( M, T \inftensor_S M )$ are adjoint to the maps
  \[ \hat \eta_{l,r}  \colon  (\shift R)^{\otimes_\k l} \otimes_\k T \otimes_\k (\shift S)^{\otimes_\k r} \otimes_\k M  \longto  T \inftensor_S M \]
  given by the inclusion for $l = 0$ and the trivial map for $l > 0$.
  For a left $R$-module $N$, the structure maps $\ev_l$ of the counit $\ev \colon \Hom_\k(M, N) \inftensor_S M \to N$ are trivial for $l > 0$ and given by the following composite for $l = 0$
  \[ \Hom_\k(M, N) \inftensor_S M  \xlongto{\pr}  \Hom_\k(M, N) \tensor_\k M  \xlongto{\ev}  N  \]
  where $\pr$ is the projection.
\end{observation}

The construction $\Hom_\k(\blank, \blank)$ in particular allows us to define the dual of a left module over an $\Ainf$-algebra, partially recovering a construction of Tradler \cite[Lemma~2.9]{Tra}.

\begin{notation}
  Let $R$ be an $\Ainf$-algebra over $\k$ and $M$ a left $R$-module.
  We denote by $\dual M_\k$ the right $R$-module $\Hom_\k(M, \k)$ (where we consider $\k$ as a left $0$-module).
\end{notation}

\begin{observation} \label{obs:Hom_hf}
  Let $R$ and $S$ be $\Ainf$-algebras over $\k$, $M$ a left $R$-module, and $N$ a left $S$-module.
  Then the $\k$-module map
  \begin{align*}
    N \tensor_\k \dual{M}_\k  &\longto  \Hom_\k(M, N) \\
    n \tensor \phi  &\longmapsto  n \cdot \phi(\blank)
  \end{align*}
  is a strict map of $S$-$R$-bimodules, and natural in both $M$ and $N$.
  It is a quasi-isomorphism if $M$ and $N$ are cofibrant as $\k$-modules and $M$ is homotopically finite, by \cite[\IlemmaHomqiso]{I}.
\end{observation}

We will need certain compatibilities of restriction with the tensor--hom adjunction, as follows.

\begin{observation} \label{obs:Hom_restrict}
  Let $f \colon R \to R'$ and $g \colon S \to S'$ be maps of $\Ainf$-algebras over $\k$, $M$ a left $S'$-module, and $N$ a left $R'$-module.
  Then the identity defines a strict isomorphism
  \[ (f, g)^* \bigl( \Hom_\k(M, N) \bigr)  \iso  \Hom_\k \bigl( g^*(M), f^*(N) \bigr) \]
  of $R$-$S$-bimodules.
\end{observation}

\begin{lemma} \label{lemma:eta_restrict}
  Let $f \colon R \to R'$ and $g \colon S \to S'$ be maps of $\Ainf$-algebras over $\k$, $M$ a left $S'$-module, and $N$ an $R'$-$S'$-bimodule.
  Then the following diagram of $R$-$S$-bimodules commutes
  \[
  \begin{tikzcd}[column sep = 35]
    (f, g)^* (N) \rar{\eta_{(f, g)^*(N)}} \dar[swap]{(f, g)^*(\eta_N)} & \Hom_\k \bigl( g^*(M) , (f, g)^*(N) \inftensor_S g^*(M) \bigr) \dar{g_*} \\
    (f, g)^* \bigl( \Hom_\k(M , N \inftensor_{S'} M) \bigr) \rar{\iso} & \Hom_\k \bigl( g^*(M) , f^*( N \inftensor_{S'} M) \bigr) 
  \end{tikzcd}
  \]
  where $\eta$ denotes the unit of the adjunction of \cref{lemma:Ainf_Hom}, and the other two maps are induced by the maps of \cref{obs:inftensor_restrict,obs:Hom_restrict}.
\end{lemma}

\begin{proof}
  This follows from chasing through the definitions, using \cref{obs:hom-tensor_unit}.
\end{proof}

\subsection{Maps to hom-spaces}

For a left $R$-module $M$, we will need a natural map $\nu \colon \shift R \to \Hom_\k(M, M)$ of $R$-$R$-bimodules that picks out the identity of $M$.
This is almost achieved by the unit of the adjunction between $\blank \inftensor_R M$ and $\Hom_\k(M, \blank)$, which is a map $\eta \colon \shift R \to \Hom_\k ( M, \shift R \inftensor_R M )$.
There is also a canonical map $\pi \colon \shift R \inftensor_R M \to M$ of left $R$-modules, but it turns out to not be natural in the pair $(R, M)$.
We will however prove that it is natural in the homotopy category of left $R$-modules as long as $R$ and $M$ are unital (and $M$ is cofibrant as a $\k$-module), which is enough for our purposes.
To this end, we observe that $\pi$ has a natural quasi-inverse and is thus itself natural in the homotopy category.
We begin by constructing $\pi$.

\begin{lemma} \label{lemma:pi}
  Let $R$ be an $\Ainf$-algebra over $\k$ and $M$ a left $R$-module.
  Then there is a left $R$-module map $\pi \colon \shift R \inftensor_R M \to M$ of degree $1$ whose structure map $\pi_l$ is given, on the $n$-th summand of $\shift R \inftensor_R M$, by $\mu^M_{l + 1 + n}$.
\end{lemma}

\begin{proof}
  We first check that $\pi$ is a left $R$-module map.
  The left-hand side of the $l$-th left module map equation \eqref{eq:bimodule_map} for $\pi$, restricted to $(\shift R)^{\otimes_\k l} \otimes_\k (\shift R \otimes_\k (\shift R)^{\otimes_\k n} \otimes_\k M)$, is
  \[ - \sum_{l_1 + l_2 = l} \mu^M_{l_1} \circ ({\id^{\otimes l_1}} \otimes \mu^M_{l_2 + 1 + n}) \]
  and the right-hand side is
  \begin{align*}
    &\sum_{l_1 + l_2 + l_3 = l} \mu^M_{l_1 + 1 + l_3 + 1 + n} \circ ({\id^{\otimes l_1}} \otimes \mu^R_{l_2} \otimes \id^{\otimes l_3 + 1 + n + 1}) \\
    + &\sum_{\substack{l_1 + l_2 = l \\ n_1 + n_2 = n}} \mu^M_{l_1 + 1 + n_2} \circ ({\id^{\otimes l_1}} \otimes \mu^R_{l_2 + 1 + n_1} \otimes \id^{\otimes n_2 + 1}) \\
    + &\sum_{n_1 + n_2 + n_3 = n} \mu^M_{l + 1 + n_1 + 1 + n_3} \circ ({\id^{\otimes l + 1 + n_1}} \otimes \mu^R_{n_2} \otimes \id^{\otimes n_3 + 1}) \\
    + &\sum_{n_1 + n_2 = n} \mu^M_{l + 1 + n_1} \circ ({\id^{\otimes l + 1 + n_1}} \otimes \mu^M_{n_2})
  \end{align*}
  which together precisely yields the $(l + 1 + n)$-th left module equation \eqref{eq:bimodule} for $M$.
\end{proof}

We now construct a natural map in the other direction and show that it is quasi-inverse to $\pi$.

\begin{lemma}
  Let $R$ be a unital $\Ainf$-algebra over $\k$ and $M$ a left $R$-module.
  There is a left $R$-module map $\iota \colon M \to \shift R \inftensor_R M$ of degree $-1$ with structure maps
  \[ \iota_l(x_1 \otimes \dots \otimes x_l \otimes m)  \defeq  \shift 1 \otimes x_1 \otimes \dots \otimes x_l \otimes m \]
  and $\iota$ is natural in $M$.
\end{lemma}

\begin{proof}
  Writing $M' \defeq \shift R \inftensor_R M$, the left-hand side of the $l$-th left module map equation \eqref{eq:bimodule_map} for $\iota$ is
  \begin{align*}
    \smash{ - \sum_{l_1 + l_2 = l} \mu^{M'}_{l_1} \circ ({\id^{\otimes l_1}} \otimes \shift 1 \otimes \id^{\otimes l_2 + 1}) } &= \sum_{l_1 + l_2 = l} \shift 1 \otimes {\id^{\otimes l_1}} \otimes \mu^M_{l_2} \\
    &+ \sum_{l_1 + l_2 + l_3 = l} \shift 1 \otimes {\id^{\otimes l_1}} \otimes \mu^R_{l_2} \otimes \id^{\otimes l_3 + 1} \\
    &- \bigl( \mu^R_2 \after (\shift 1 \otimes \id) \bigr) \otimes \id^{\otimes l} \\
    &- \bigl( \mu^R_2 \after ({\id} \otimes \shift 1) \bigr) \otimes \id^{\otimes l}
  \end{align*}
  where on the right-hand side the last two terms cancel by unitality of $R$.
  The right-hand side of \eqref{eq:bimodule_map} is
  \[ \sum_{l_1 + l_2 + l_3 = l} \shift 1 \otimes {\id^{\otimes l_1}} \otimes \mu^R_{l_2} \otimes \id^{\otimes l_3 + 1} + \sum_{l_1 + l_2 = l} \shift 1 \otimes {\id^{\otimes l_1}} \otimes \mu^M_{l_2} \]
  and hence equals the left-hand side.
  
  Given a map $f \colon M \to N$ of left $R$-modules, we have
  \[ (\iota \after f)_l  =  \sum_{l_1 + l_2 = l} \shift 1 \otimes {\id^{\otimes l_1}} \otimes f_{l_2}  =  (-1)^{\deg f} (f_* \after \iota)_l \]
  and hence $\iota \after f = (-1)^{\deg f} f_* \after \iota$.
  Thus $\iota$ is natural.
\end{proof}

\begin{lemma} \label{lemma:pi_qiso}
  Let $R$ be a unital $\Ainf$-algebra over $\k$ and $M$ a unital left $R$-module.
  Then the composite
  \[ M  \xlongto{\iota}  \shift R \inftensor_R M  \xlongto{\pi}  M \]
  is the identity, and $\pi$ and $\iota$ are quasi-isomorphisms.
\end{lemma}

\begin{proof}
  That $\pi \after \iota = \id$ follows from unwinding the definitions.
  We will now show that furthermore $\iota_0 \after \pi_0$ is homotopic to the identity.
  To this end, we define the homotopy $h \colon \shift R \inftensor_R M \to \shift R \inftensor_R M$ by the formula
  \[ h(x \otimes y_1 \otimes \dots \otimes y_n \otimes m) \defeq s1 \otimes x \otimes y_1 \otimes \dots \otimes y_n \otimes m \]
  and compute, using \cref{obs:comp_description_infty_tensor}, that on $\shift R \otimes_\k (\shift R)^{\otimes_\k n} \otimes_\k M$ we have
  \begin{align*}
    d \circ h &= \left( \sum_{l_1 + l_2 + l_3 = n + 2} \id^{\otimes l_1} \otimes \mu^R_{l_2} \otimes \id^{\otimes l_3 + 1} + \sum_{l_1 + l_2 = n + 1} \id^{\otimes 1 + l_1} \otimes \mu^M_{l_2} \right) \circ(\shift 1 \otimes \id^{\otimes 1 + n + 1}) \\
    &= \id^{\otimes 1 + n + 1} - \sum_{l_1 + l_2 + l_3 = n + 1} \shift 1 \otimes \id^{\otimes l_1} \otimes \mu^R_{l_2} \otimes \id^{\otimes l_3 + 1} - \sum_{l_1 + l_2 = n + 1} \shift 1 \otimes \id^{\otimes l_1} \otimes \mu^M_{l_2}
  \end{align*}
  where we used that $R$ is unital.
  Moreover 
  \[ h \after d  =  \sum_{l_1 + l_2 + l_3 = n + 1} \shift 1 \otimes \id^{\otimes l_1} \otimes \mu^R_{l_2} \otimes \id^{\otimes l_3 + 1} + \sum_{l_1 + l_2 = n} \shift 1 \otimes \id^{\otimes 1 + l_1} \otimes \mu^M_{l_2} \]
  so that
  \[ d \after h + h \after d = \id^{\otimes 1 + n + 1} - \shift 1 \otimes \mu^M_{n + 1} = \id^{\otimes 1 + n + 1} - \iota_0 \after \pi_0 \]
  which completes the proof.
\end{proof}

We note the following for later use.

\begin{lemma} \label{lemma:iota_restrict}
  Let $f \colon R \to S$ be a unital map of unital $\Ainf$-algebras over $\k$ and $M$ a left $S$-module.
  Then the following composite map of left $R$-modules, obtained using \cref{lemma:algebra_map_is_bimodule_map,obs:inftensor_restrict},
  \[ f^*(M)  \xlongto{\iota}  \shift R \inftensor_R f^*(M)  \xlongto{f'_*}  (f, f)^*(\shift S) \inftensor_R f^*(M)  \xlongto{f_*}  f^*(\shift S \inftensor_S M) \]
  is equal to $f^*(\iota)$.
\end{lemma}

\begin{proof}
  The $l$-th structure map of the composite is
  \[ \sum_{l_1 + \dots + l_n = l} \shift 1 \otimes f_{l_1} \otimes \dots \otimes f_{l_n} \otimes \id_M \]
  which is also the $l$-th structure map of $f^*(\iota)$.
\end{proof}

We are now ready to define the map $\nu$ and prove that it is natural in the homotopy category, as promised in the beginning of this subsection.

\begin{definition} \label{def:nu}
  Let $R$ be an $\Ainf$-algebra over $\k$ and $M$ a left $R$-module.
  Then we define the $R$-$R$-bimodule map of degree $1$
  \[ \nu  \colon  \shift R  \xlongto{\eta}  \Hom_\k ( M, \shift R \inftensor_R M )  \xlongto{\pi_*}  \Hom_\k(M, M) \]
  to be the composite of the unit of the adjunction $\blank \inftensor_R M \dashv \Hom_\k(M, \blank)$ and the map induced by the map $\pi$ of \cref{lemma:pi}.
\end{definition}

\begin{observation} \label{obs:nu}
  The structure maps of $\nu$ are the maps
  \[ \nu_{l,r} \colon (\shift R)^{\otimes_\k l} \otimes_\k \shift R \otimes_\k (\shift R)^{\otimes_\k r}  \longto  \Hom_\k(M, M) \]
  adjoint to the structure maps $\mu^M_{l + 1 + r}$ of $M$.
\end{observation}

\begin{lemma} \label{lemma:nu_natural}
  Let $f \colon R \to S$ be a unital map of unital $\Ainf$-algebras over $\k$, $M$ a unital left $R$-module, $N$ a unital left $S$-module, and $g \colon M \to f^*(N)$ a map of left $R$-modules.
  Assume that $M$ and $N$ are cofibrant as $\k$-modules.
  Then the following diagram (in which we omitted restrictions for readability) commutes in the homotopy category of $R$-$R$-bimodules
  \[
  \begin{tikzcd}
    \shift R \ar{rr}{f'} \dar[swap]{\nu} & & \shift S \dar{\nu} \\
    \Hom_\k(M, M) \rar{g_*} & \Hom_\k(M, N) & \lar[swap]{g^*} \Hom_\k(N, N)
  \end{tikzcd}
  \]
  where $f'$ is the map of \cref{lemma:algebra_map_is_bimodule_map}.
\end{lemma}

\begin{proof}
  Consider the following diagram of $R$-$R$-bimodules
  \[
  \begin{tikzcd}
    {} \drar[phantom][near end]{\text{(a)}} & \Hom_\k(M, M) & {} \dlar[phantom][near end]{\text{(b)}} \\
    \shift R \rar{\eta} \ar[equal]{dd} \urar[bend left, end anchor = west]{\nu} & \Hom_\k(M, \shift R \inftensor_R M) \dar[swap]{g_*} \uar{\pi_*}[swap]{\eq} \drar[phantom]{\text{(d)}} & \lar{\iota_*}[swap]{\eq} \ular[bend right = 20, end anchor = east][swap]{\id} \Hom_\k(M, M) \dar{g_*} \\
    {} \rar[phantom]{\text{(c)}} & \Hom_\k(M, \shift R \inftensor_R N) \drar[phantom]{\text{(e)}} & \lar{\iota_*}[swap]{\eq} \Hom_\k(M, N) \\
    \shift R \rar{\eta} \dar[swap]{f'} \drar[phantom]{\text{(f)}} & \Hom_\k(N, \shift R \inftensor_R N) \uar{g^*} \dar[swap]{f'_*} & \lar{\iota_*}[swap]{\eq} \Hom_\k(N, N) \uar[swap]{g^*} \ar[equal]{dd} \\
    \shift S \rar{\eta} \dar[equal] \drar[phantom]{\text{(g)}} & \Hom_\k(N, \shift S \inftensor_R N) \dar[swap]{f_*} & {} \lar[phantom]{\text{(h)}} \\
    \shift S \rar{\eta} \drar[bend right, end anchor = west][swap]{\nu} & \Hom_\k(N, \shift S \inftensor_S N) \dar{\eq}[swap]{\pi_*} & \lar{\iota_*}[swap]{\eq} \dlar[bend left = 20, end anchor = east]{\id} \Hom_\k(N, N) \\
    {} \urar[phantom][near end]{\text{(a)}} & \Hom_\k(N, N) & {} \ular[phantom][near end]{\text{(b)}}
  \end{tikzcd}
  \]
  where the triangles labeled (a) commute by definition of $\nu$, the ones labeled (b) by \cref{lemma:pi_qiso}, (c) and (e) by \cref{obs:Hom_bifunctor}, (d) by the naturality of $\iota$, (f) by naturality of $\eta$, (g) by \cref{lemma:eta_restrict}, and (h) by \cref{lemma:iota_restrict}.
  The indicated maps are quasi-isomorphisms by \cref{lemma:pi_qiso} and our assumption that $M$ and $N$ are cofibrant as $\k$-modules.
\end{proof}

\subsection{Hochschild and cyclic homology}

In this subsection we define Hochschild homology and cyclic homology of $\Ainf$-algebras, and record various basic properties we will need.
We begin by recalling the construction of the Hochschild complex of an $\Ainf$-algebra with coefficients in a bimodule, following Getzler--Jones \cite[Definition~3.5]{GJ}; also see Mescher \cite[§2]{Mes}.

\begin{definition} \label{def:HH_Ainf}
  Let $R$ be an $\Ainf$-algebra over $\k$, and $M$ an $R$-$R$-bimodule.
  We define the \emph{Hochschild complex} of $R$ with coefficients in $M$ to be the $\k$-module
  \[ \HH_\k(R, M)  \defeq  \equalizer \bigl( \B(R, M, R) \rightrightarrows \B(R) \tensor_\k \B(R, M, R) \bigr)  \]
  where one of the maps is the left $\B(R)$-comodule structure of $\B(R, M, R)$ and the other map is the composite of the right $\B(R)$-comodule structure and the swap map.
  We furthermore define
  \[ \HH_\k(R)  \defeq  \shift[-1] \HH_\k(R, \shift R) \]
  where $\shift R$ is considered as an $R$-$R$-bimodule via \cref{obs:algebra_is_bimodule}.
\end{definition}

\begin{observation} \label{obs:HH}
  The $\k$-module $\HH_\k(R, M)$ is isomorphic to
  \[ \bigoplus_{n \ge 0} M \otimes_\k (\shift R)^{\otimes_\k n} \]
  with differential $d_{HH}$ given, on the $n$-th summand, by
  \[ \sum_{r + s + t = n} {\id^{\otimes 1 + r} \otimes \mu^R_{s} \otimes \id^{\otimes t}} + \sum_{r + s + l = n} (\mu^M_{l,r} \otimes \id^{\otimes s}) \circ \cyclic {1 + n}^{\circ l} \]
  where $\cyclic i$ is the cyclic permutation of \cref{not:cyclic}.
  We call $n$ the \emph{Hochschild degree} of an element of $\HH_\k(R, M)$.
\end{observation}

Note that the construction $\HH_\k(R, \blank)$ is closely related to $\blank \inftensor_R \blank$.
For instance, given a left $R$-module $M$ and right $R$-module $N$, there is a canonical isomorphism $\HH_\k(R, M \otimes_\k N) \iso N \inftensor_{R} M$.
The following is a more general version of this observation.

\begin{observation} \label{obs:HH_cyclic}
  Let $R$ and $S$ be $A_\infty$-algebras over $\k$, and let $M$ be an $R$-$S$-bimodule and $N$ an $S$-$R$-bimodule.
  Then there is a canonical isomorphism
  \[ \HH_\k \bigl( R, M \inftensor_S N \bigr) \iso \HH_\k \bigl( S, N \inftensor_R M \bigr) \]
  given by a cyclic permutation (using \cref{obs:HH,obs:comp_description_infty_tensor} to see that this is compatible with the differentials).
\end{observation}

We now provide a construction of cyclic homology for $\Ainf$-algebras via a generalization of the classical Connes complex for algebras (see \cref{def:HH}).
A different definition of cyclic homology for $\Ainf$-algebras has been given by Getzler--Jones \cite[Theorem~3.7]{GJ} via a generalization of the cyclic bicomplex.
It seems very likely that a generalization of the classical argument (see e.g.\ \cite[Theorem~2.1.5]{Lod}) proves that these complexes are naturally quasi-isomorphic.

\begin{definition}
  Let $R$ be an $\Ainf$-algebra over $\k$.
  We define the \emph{Connes complex} of $R$ to be the quotient
  \[ \HC_\k(R)  \defeq  \shift[-1] \bigoplus_{n \ge 0} (\shift R)^{\tensor_\k n+1}_{\Cyclic {n+1}}  \longtwoheadleftarrow  \shift[-1] \bigoplus_{n \ge 0} (\shift R)^{\tensor_\k n+1}  \iso  \HH_\k(R) \]
  where the cyclic group $\Cyclic{n+1}$ acts by cyclic permutations and the isomorphism is the one of \cref{obs:HH}
\end{definition}

Note that it is clear that the differential of $\HH_\k(R)$ indeed descends to $\HC_\k(R)$.
We now exhibit the Hochschild complex and the Connes complex as being functorial.

\begin{definition}
  We denote by $\BiMod$ the Grothendieck construction of the functor $\AinfAlg \to \opcat{\Cat}$ given on objects by $R \mapsto \BMod{R}{R}$ and on morphisms by $f \mapsto (f, f)^*$.
  An object of $\BiMod$ consists of a pair $(R, M)$ of an $\Ainf$-algebra $R$ over $\k$ and an $R$-$R$-bimodule $M$; a morphism $(R, M) \to (S, N)$ consist of a pair $(f, g)$ of a map $f \colon R \to S$ of $\Ainf$-algebras over $\k$ and a map $g \colon M \to (f,f)^*(N)$ of $R$-$R$-bimodules.
\end{definition}

Note that the assignment $R \mapsto (R, \shift R)$ extends to a functor $\AinfAlg \to \BiMod$ via the construction of \cref{lemma:algebra_map_is_bimodule_map}.

\begin{observation} \label{obs:HH_functor}
  The assignment $(R, M) \mapsto \HH_\k(R, M)$ extends to a functor from $\BiMod$ to $\k$-modules by letting a morphism $(f, g) \colon (R, M) \to (S, N)$ act by the map induced on horizontal equalizers by the commutative diagram
  \[
  \begin{tikzcd}
    \B(R, M, R) \rar[shift left] \rar[shift right] \dar[swap]{g} & \B(R) \tensor_\k \B(R, M, R) \dar{{\id} \tensor g} \\
    \B(R) \cotensor^{\B(S)} \B(S, N, S) \cotensor^{\B(S)} \B(R) \rar[shift left] \rar[shift right] \dar[swap]{f \cotensor {\id} \cotensor f} & \B(R) \tensor_\k \bigl( \B(R) \cotensor^{\B(S)} \B(S, N, S) \cotensor^{\B(S)} \B(R) \bigr) \dar{f \tensor (f \cotensor {\id} \cotensor f)} \\
    \B(S, N, S) \rar[shift left] \rar[shift right] & \B(S) \tensor_\k \B(S, N, S) 
  \end{tikzcd}
  \]
  of $\k$-modules.
  On Hochschild degree $n$, the induced map is explicitly given by
  \[ \sum_{n_1 + \dots + n_i = n} (g_{n_i,n_1} \otimes f_{n_2} \otimes \dots \otimes f_{n_{i-1}}) \circ \cyclic {1 + n}^{\circ n_i} \]
  (where each summand takes values in Hochschild degree $i - 2$).
  This furthermore makes $\HH_\k(\blank)$ and $\HC_\k(\blank)$ into functors from $\AinfAlg$ to $\k$-modules.
\end{observation}

As expected, the definition of Hochschild and cyclic homology of $\Ainf$-algebras generalizes the classical definitions of algebras, as follows.

\begin{observation} \label{obs:HH_agrees}
  Let $\k \to R$ be a map of dgas and $M$ an $(R \tensor_\k \opmod{R})$-module.
  Then there is a canonical natural isomorphism between the Hochschild complex $\HH_\k(R, M)$ as constructed in \cref{def:HH} and the construction of \cref{def:HH_Ainf}, where we consider $R$ as an $\Ainf$-algebra over $\k$ and $M$ as an $R$-$R$-bimodule via \cref{obs:dga_as_Ainf,obs:dgm_as_Ainf}.
  The analogous statement holds for $\HH_\k(R)$ and $\HC_\k(R)$.
\end{observation}

We now observe that both Hochschild homology and cyclic homology are invariant under quasi-isomorphisms.

\begin{lemma} \label{lemma:HH_qiso}
  Let $(f, g) \colon (R, M) \to (S, N)$ be a morphism of $\BiMod$ such that $R$ and $S$ are cofibrant as $\k$-modules.
  If $f$ and $g$ are quasi-isomorphisms, then the induced maps
  \[ \HH_\k(R, M) \to \HH_\k(S, N),  \qquad  \HH_\k(R) \to \HH_\k(S),  \qquad  \HC_\k(R) \to \HC_\k(S) \]
  are all quasi-isomorphism as well.
\end{lemma}

\begin{proof}
  Filtering $\HH_\k(R, M)$ by Hochschild degree yields an exhaustive bounded-below increasing filtration whose graded pieces are $M \otimes_\k (\shift R)^{\otimes_\k n}$ equipped with the differential of the tensor product.
  The map $(f, g)_* \colon \HH_\k(R, M) \to \HH_\k(S, N)$ respects this filtration, and induces $g_{0,0} \otimes f_0^{\otimes n}$ on the $n$-th graded piece.
  Since $R$ and $S$ are cofibrant as $\k$-modules, this is a quasi-isomorphism by \cite[\Ilemmacofibrantflat]{I}, and hence $(f, g)_*$ is one as well.
  This implies the claim for $\HH_\k(\blank)$, and the same proof works for $\HC_\k(\blank)$ (using that taking coinvariants of the action of a finite group preserves quasi-isomorphisms).
\end{proof}

We now prove that cyclic homology of an $\Ainf$-algebra can be recovered as the cyclic homology of its bar construction.

\begin{lemma} \label{lemma:HC_of_bar}
  Let $R$ be an $\Ainf$-algebra over $\k$.
  Then the following map is a quasi-isomorphism of $\k$-modules
  \begin{align*}
    \HC_\k(R)  &\xlongto{\eq}  \HC_\k(\rB R) \\
    x_0 \tensor \dots \tensor x_n  &\longmapsto  \sum_{i = 0}^n (-1)^\epsilon x_i \tensor \dots \tensor x_n \tensor x_0 \tensor \dots \tensor x_{i - 1}
  \end{align*}
  where the tensor on the right-hand side is considered as an element of $\rB R$ and $\epsilon \defeq (\deg {x_0} + \dots + \deg {x_{i - 1}}) (\deg {x_i} + \dots + \deg {x_n})$.
  Here $\rB R$ denotes the non-counital bar construction of $R$ (see \cref{def:Ainf}) and $\HC_\k(\rB R)$ the Connes complex of this $\k$-coalgebra (defined exactly dually to the Connes complex of a $\k$-algebra, see \cref{def:HH}).
\end{lemma}

\begin{proof}
  Unwinding the definitions, one checks that the map is compatible with the differentials.
  Now we filter both sides by the total number of elements of $R$.
  Both filtrations are bounded below and exhaustive, and the map is filtration preserving.
  Hence it is enough to prove that the map induces a quasi-isomorphism on associated graded.
  
  The associated graded of the filtration on $\HC_\k(R)$ is the $\k$-module $\bigoplus_{n \ge 1} (R^{\tensor_k n})_{\Cyclic n}$, where the cyclic group $\Cyclic n$ acts by cyclic permutations.
  The associated graded of the filtration on $\HC_\k(\rB R)$ is the cyclic complex $\HC_\k(\rcoT_\k (\shift R))$ for the non-counital tensor coalgebra on the $\k$-module $\shift R$.
  By the duals of \cite[Theorem~3.1.6, Proposition~2.2.16, and §2.2.13]{Lod}, the composite of the inclusions
  \[ \bigoplus_{n \ge 1} (R^{\tensor_k n})^{\Cyclic n}  \longto  \rcoT_\k (\shift R)  \longto  \HC_\k \bigl( \rcoT_\k (\shift R) \bigr) \]
  is a quasi-isomorphism.
  Since we are working rationally, the norm maps $(R^{\tensor_k n})_{\Cyclic n} \to (R^{\tensor_k n})^{\Cyclic n}$ are isomorphisms, and the claim follows.
\end{proof}

\begin{remark}
  Note that the analog of \cref{lemma:HC_of_bar} for Hochschild homology is false.
  As a counterexample, let $R$ be an $\Ainf$-algebra over $\QQ$ with $\mu_n^R = 0$ for all $n$.
  Then the bar construction $\rB R$ is just the non-counital tensor coalgebra $\rcoT(\shift R)$ whose Hochschild homology is not necessarily concentrated in Hochschild degree $0$ (see e.g.\ \cite[Theorem~3.1.4]{Lod}).
\end{remark}

The following lemma will be needed later.

\begin{lemma} \label{lemma:HH_proj}
  Let $R$ be an $\Ainf$-algebra over $\k$ and $M$ a symmetric $R$-$R$-bimodule (cf.\ \cref{def:symmetric_bimod}).
  Then the projection $\epsilon \colon \HH_\k(R, M) \to M$ onto Hochschild degree $0$ is a map of $\k$-modules (in particular it is a map of cochain complexes).
  It is natural in the subcategory of $\BiMod$ with objects the pairs where the bimodule is symmetric and morphisms those where the map of bimodules is symmetric.
  When $R$ is a $\Cinf$-algebra over $\k$, the projections
  \[ \HH_\k(R) \longto \shift[-1] \shift R \iso R  \qquad \text{and} \qquad  \HC_\k(R) \longto \shift[-1] \shift R \iso R \]
  are natural maps of $\k$-modules.
\end{lemma}

\begin{proof}
  That $\epsilon$ is a natural map of $\k$-modules is immediate from \cref{obs:HH,obs:HH_functor} and the definitions.
  The claim for $\Cinf$-algebras follows by additionally using \cref{lemma:Cinf_is_symm}.
\end{proof}

\subsubsection*{The normalized Hochschild complex}

Note that the Hochschild complex $\HH(R)$ of an $\Ainf$-algebra $R$ is generally not bounded below, even if $R$ is concentrated in non-negative degrees, since $(\shift R)^{\tensor n}$ is only concentrated in degrees $\ge -n$ (recall that $\shift(\blank)$ denotes a cohomological shift by $-1$).
To resolve this issue, we introduce a normalized version of the Hochschild complex of a unital $\Ainf$-algebra, generalizing the case of a unital algebra (see e.g.\ \cite[1.1.14]{Lod}).
Also see \cite[§5]{GJ} for a similar construction in the non-unital case.

\begin{definition} \label{def:nHH}
  Let $R$ be a unital $\Ainf$-algebra over $\k$, and $M$ a unital $R$-$R$-bimodule.
  We write $\dHH_\k(R, M) \subseteq \HH_\k(R, M)$ for the graded $\k$-submodule spanned by the elements $m \tensor x_1 \tensor \dots \tensor x_n$ such that $x_i = \shift 1$ for some $i$.
  We define the \emph{normalized Hochschild complex} $\nHH_\k(R, M)$ of $R$ with coefficients in $M$ to be the quotient of $\HH_\k(R, M)$ by $\dHH_\k(R, M)$.
  We furthermore write $\nHH_\k(R) \defeq \shift[-1] \nHH_\k(R, \shift R)$.
\end{definition}

Note that it follows from the definitions that $\dHH_\k(R, M) \subseteq \HH_\k(R, M)$ is closed under the differential, so that both $\dHH_\k(R, M)$ and $\nHH_\k(R, M)$ are cochain complexes and thus $\k$-modules.

\begin{lemma} \label{lemma:nHH_qiso}
  Let $R$ be a unital $\Ainf$-algebra over $\k$, and $M$ a unital $R$-$R$-bimodule.
  Then the quotient map $\HH_\k(R, M) \to \nHH_\k(R, M)$ is a quasi-isomorphism.
\end{lemma}

\begin{proof}
  It is enough to prove that $\dHH_\k(R, M)$ is acyclic.
  To this end, we consider the ascending filtration whose $p$-th stage $F_p \subseteq \dHH_\k(R, M)$ is the graded $\k$-submodule spanned by the elements $m \tensor x_1 \tensor \dots \tensor x_n$ such that $x_i = \shift 1$ for some $i \le p$.
  Note that it follows from the definitions that $F_p$ is closed under the differential.
  Since the filtration is bounded below and exhaustive, it is enough to prove that the graded pieces $G_p \defeq \quot{F_p}{F_{p-1}}$ are acyclic.
  To this end, consider the map of graded $\k$-modules $r_p \colon F_p \to F_p$ given by
  \[ m \tensor x_1 \tensor \dots \tensor x_n  \longmapsto  (-1)^\epsilon m \tensor x_1 \tensor \dots \tensor x_{p-1} \tensor \shift 1 \tensor x_p \tensor \dots \tensor x_n \]
  where $\epsilon \defeq \deg m + \sum_{i = 1}^{p-1} \deg{x_i}$ (by convention $r_p$ vanishes on elements with $n < p - 1$).
  We claim that $r_p$ induces a contraction $s_p \colon G_p \to G_p$, i.e.\ that $d s_p + s_p d = - \id$.
  First note that it is indeed clear that $r_p(F_{p-1}) \subseteq F_{p-1}$.
  The $\k$-module $G_p$ is spanned by elements of the form
  \[ a  \defeq  m \tensor x_1 \tensor \dots \tensor x_{p-1} \tensor \shift 1 \tensor x_{p + 1} \tensor \dots \tensor x_n \]
  with $n \ge p$.
  Its image under the differential is represented by
  \begin{equation} \label{eq:nHH_d}
    \begin{aligned}
      &\mathbin{\phantom{+}} \sum_{1 \le i < p} \pm m \tensor x_1 \tensor \dots \tensor \mu^R_1(x_i) \tensor \dots \tensor x_{p-1} \tensor \shift 1 \tensor x_{p + 1} \tensor \dots \tensor x_n \\
      &+ \sum_{p \le l \le n} \pm \mu^M_{n-l,0}(x_{l+1} \tensor \dots \tensor x_n \tensor m) \tensor x_1 \tensor \dots \tensor x_{p-1} \tensor \shift 1 \tensor x_{p + 1} \tensor \dots \tensor x_l \\
      &+ \sum_{p < l \le r \le n} \pm m \tensor x_1 \tensor \dots \tensor x_{p-1} \tensor \shift 1 \tensor x_{p + 1} \tensor \dots \tensor \mu^R_{r-l+1}(x_l \tensor \dots \tensor x_r) \tensor \dots \tensor x_n
    \end{aligned}
  \end{equation}
  (the remaining terms either lie in $F_{p-1}$ or cancel out).
  On the other hand, its image $s_p(a)$ is represented by
  \[ \pm m \tensor x_1 \tensor \dots \tensor x_{p-1} \tensor \shift 1 \tensor \shift 1 \tensor x_{p + 1} \tensor \dots \tensor x_n \]
  whose differential is represented by
  \begin{equation} \label{eq:nHH_ds}
    \begin{aligned}
      &\mathbin{\phantom{+}} \sum_{1 \le i < p} \pm m \tensor x_1 \tensor \dots \tensor \mu^R_1(x_i) \tensor \dots \tensor x_{p-1} \tensor \shift 1 \tensor \shift 1 \tensor x_{p + 1} \tensor \dots \tensor x_n \\
      &+ \sum_{p \le l \le n} \pm \mu^M_{n-l,0}(x_{l+1} \tensor \dots \tensor x_n \tensor m) \tensor x_1 \tensor \dots \tensor x_{p-1} \tensor \shift 1 \tensor \shift 1 \tensor x_{p + 1} \tensor \dots \tensor x_l \\
      &+ \sum_{p < l \le r \le n} \pm m \tensor x_1 \tensor \dots \tensor x_{p-1} \tensor \shift 1 \tensor \shift 1 \tensor x_{p + 1} \tensor \dots \tensor \mu^R_{r-l+1}(x_l \tensor \dots \tensor x_r) \tensor \dots \tensor x_n \\
      &- m \tensor x_1 \tensor \dots \tensor x_{p-1} \tensor \shift 1 \tensor x_{p + 1} \tensor \dots \tensor x_n
    \end{aligned}
  \end{equation}
  (the remaining terms again either lie in $F_{p-1}$ or cancel out).
  Applying $s_p$ to \eqref{eq:nHH_d}, we observe that the signs work out precisely so that this cancels all terms of \eqref{eq:nHH_ds} except for the last one.
  Hence we have $d s_p + s_p d = - \id$ as desired, which completes the proof.
\end{proof}

\begin{lemma} \label{lemma:nHH_cofibrant}
  Let $R$ be a unital $\Ainf$-algebra over $\k$, and $M$ a unital $R$-$R$-bimodule.
  If the unit map $\eta \colon \k \to R$ is a cofibration of $\k$-modules, and $M$ is a cofibrant $\k$-module, then $\nHH_\k(R, M)$ is a cofibrant $\k$-module.
\end{lemma}

\begin{proof}
  The underlying graded $\k$-module of $\nHH_\k(R, M)$ is $\bigoplus_{n \ge 0} M \tensor_\k (\quot R 1)^{\tensor_\k n}$.
  Since the differential does not increase the Hochschild degree $n$, this decomposition yields a canonical exhaustive increasing filtration $F_n$ on the $\k$-module $\nHH_\k(R, M)$.
  Since each graded piece is cofibrant by assumption, the inclusion of $F_n \to F_{n+1}$ is a cofibration.
  Hence $\nHH_\k(R, M)$ is cofibrant.
\end{proof}

%% file: 2transfer.tex
\section{The Hochschild homology transfer of \texorpdfstring{$\Ainf$}{A∞}-algebras} \label{sec:transfer}

In this section we first define the Hochschild homology transfer for an $\Ainf$-algebra over a cdga, and prove that it is natural.
We then state and prove our main result, which provides an explicit description for this transfer.
This generalizes a formula due to Bouc \cite{Bou}.

\subsection{Definition and naturality} \label{sec:transfer_natural}

As mentioned in the introduction, we will work in a $G$-equivariant setting to facilitate applications to models of not necessarily nilpotent spaces (see \cref{sec:eq}).
In particular we require a notion of equivariant $\Ainf$- and $\Cinf$-algebras and bimodules over them, as follows.

\begin{definition}
  Let $G$ be a group and $\k$ a $G$-equivariant cdga.
  A \emph{$G$-equivariant $\Ainf$-algebra over $\k$} is a $G$-equivariant $\k$-module $R$ equipped with an $\Ainf$-algebra structure $\mu^R$ over $\k$ such that $\mu^R \colon \coT_\k (\shift R) \to \coT_\k (\shift R)$ is $G$-equivariant.
  A \emph{map of $G$-equivariant $\Ainf$-algebras over $\k$} is a map $f \colon R \to S$ of the underlying $\Ainf$-algebras over $\k$ such that $\B(f) \colon \B(R) \to \B(S)$ is $G$-equivariant.
  A \emph{$G$-equivariant $\Cinf$-algebra over $\k$} is a $G$-equivariant $\Ainf$-algebra over $\k$ whose underlying $\Ainf$-algebra is a $\Cinf$-algebra.
  A \emph{map of $G$-equivariant $\Cinf$-algebras over $\k$} is a map of $G$-equivariant $\Ainf$-algebras over $\k$ whose underlying map of $\Ainf$-algebras is a map of $\Cinf$-algebras.
\end{definition}

Note that a $G$-equivariant $\Ainf$-algebra over $\k$ equivalently consists of a $G$-equivariant $\k$-module $R$ and $G$-equivariant structure maps $\mu^R_n \colon (\shift R)^{\tensor_\k n} \to \shift R$.
Similarly, a map $f \colon R \to S$ of $G$-equivariant $\Ainf$-algebra over $\k$ equivalently consists of $G$-equivariant structure maps $f_n \colon (\shift R)^{\tensor_\k n} \to \shift S$.

\begin{definition}
  Let $G$ be a group, $\k$ a $G$-equivariant cdga, and $R$ and $S$ two $G$-equivariant $\Ainf$-algebras over $\k$.
  A \emph{$G$-equivariant $R$-$S$-bimodule} is a $G$-equivariant $\k$-module $M$ equipped with an $R$-$S$-bimodule structure $\mu^M$ such that $\mu^M$ is $G$-equivariant as an endomorphism of $\coT_\k (\shift R) \tensor_\k M \tensor_\k \coT_\k (\shift S)$.
  A \emph{map of $G$-equivariant $R$-$S$-bimodules} is a map $f \colon M \to N$ of the underlying $R$-$S$-bimodules such that $\B(f) \colon \B(R, M, S) \to \B(R, N, S)$ is $G$-equivariant.
\end{definition}

Note that a $G$-equivariant $R$-$S$-bimodule equivalently consists of a $G$-equivariant $\k$-module $M$ and $G$-equivariant structure maps $\mu^M_{l,r} \colon (\shift R)^{\tensor_\k l} \tensor_\k M \tensor_\k (\shift S)^{\tensor_\k r} \to M$.
Similarly, a map $f \colon M \to N$ of $G$-equivariant $R$-$S$-bimodules equivalently consists of $G$-equivariant structure maps $f_{l,r} \colon (\shift R)^{\tensor_\k l} \tensor_\k M \tensor_\k (\shift S)^{\tensor_\k r} \to N$.

\subsubsection*{The Hochschild homology transfer}

We now define the Hochschild homology transfer induced by a module $M$ over an $\Ainf$-algebra $S$ over a cdga $R$.
It generalizes the classical Hochschild homology transfer in the case that $S$ is a dga (see e.g.\ Keller \cite[§5]{Kel21} or \cite[\IconHHtransfer]{I}).

\begin{construction} \label{con:HH_transfer}
  Let $G$ be a group, $\k \to R$ a map of $G$-equivariant cdgas, $S$ a $G$-equivariant $\Ainf$-algebra over $R$, and $M$ a $G$-equivariant left $S$-module.
  Assume that $R$ and $S$ are cofibrant as $\k$-modules, and that $M$ is cofibrant and homotopically finite as an $R$-module.
  Then the \emph{Hochschild homology transfer} $\transfer{M} \colon \HH_\k(S) \to \HH_\k(R)$ is the map of $\Underlying(\Modeq{G}{\k})$ defined to be zig-zag
  \[
  \begin{tikzcd}[row sep = 15]
    \shift[-1] \HH_\k(S, \shift S) \rar{\nu_*} & \HH_\k \bigl( S, \Hom_R(M, M) \bigr) & \lar[swap]{\eq}  \HH_\k \bigl( S, M \inftensor_R \dual M_R \bigr) \dar{\iso} \\
     & & \HH_\k \bigl( R, \dual M_R \inftensor_S M \bigr) \rar{\ev_*}& \HH_\k(R, R)
  \end{tikzcd}
  \]
  where $\nu$ is the map of \cref{def:nu} and $\ev$ is the counit of \cref{obs:Hom_bifunctor}.
  The quasi-isomorphism is induced by the composite
  \[ M \inftensor_R \dual M_R  \xlongto{\eq}  M \tensor_R \dual M_R  \xlongto{\eq}  \Hom_R(M, M) \]
  of the quasi-isomorphisms of \cref{obs:inftensor_dga,obs:Hom_hf}.
  The vertical isomorphism is the one of \cref{obs:HH_cyclic}, which is given by a cyclic permutation.
\end{construction}

\subsubsection*{Naturality}

We now prove that the Hochschild homology transfer $\transfer{M} \colon \HH_\k(S) \to \HH_\k(R)$ is natural in maps of $S$ and in quasi-isomorphisms of $M$ (as long as everything is unital).
In particular this implies that the Hochschild homology transfer associated to a module $M$ over an $R$-algebra $S$ can be computed by replacing $M$ by a quasi-isomorphic unital module over a unital $\Ainf$-algebra over $R$ quasi-isomorphic to $S$.

\begin{proposition} \label{prop:transfer_natural}
  Let $G$ be a group, $\k \to R$ a map of $G$-equivariant cdgas, $\phi \colon S \to S'$ a unital map of unital $G$-equivariant $\Ainf$-algebras over $R$, $M$ a unital $G$-equivariant left $S$-module, $M'$ a unital $G$-equivariant left $S'$-module, and $M \to \phi^*(M')$ a quasi-isomorphism of $G$-equivariant left $S$-modules.
  Assume that $R$, $S$, and $S'$ are cofibrant as $\k$-modules, and that $M$ and $M'$ are cofibrant and homotopically finite as $R$-modules.
  Then the composite
  \[ \HH_\k(S)  \xlongto{\phi_*}  \HH_\k(S')  \xlongto{\transfer{M'}}  \HH_\k(R) \]
  is homotopic to the transfer $\transfer{M}$ in $\Underlying(\Modeq{G}{\k})$.
\end{proposition}

\begin{proof}
  We begin by considering the diagram of $G$-equivariant $S$-$S$-bimodules
  \[
  \begin{tikzcd}
    S \ar{dd} \rar{\nu} & \Hom_R(M, M) \dar{\eq} & \\
    & \Hom_R(M, M') & \Hom_R(M', M) \ular[bend right = 20, end anchor = east][swap]{\eq} \dlar[bend left = 20, end anchor = east]{\eq} \\
    S' \rar{\nu} & \Hom_R(M', M') \uar[swap]{\eq}
  \end{tikzcd}
  \]
  where the right-hand square commutes by \cref{obs:Hom_bifunctor} and the left-hand pentagon commutes in the homotopy category by \cref{lemma:nu_natural}.
  Thus the upper left-hand pentagon of the following diagram of $G$-equivariant $\k$-modules (where we omit shifts for readability) commutes in $\Underlying(\Modeq{G}{\k})$.
  \[
  \begin{tikzcd}[column sep = small]
    \HH_\k(S, S) \rar \ar{dd} & \HH_\k \bigl( S, \Hom_R(M, M) \bigr) & \lar[swap]{\eq} \HH_\k \bigl( S, M \inftensor_R \dual M_R \bigr) \rar{\iso} & \HH_\k \bigl( R, \dual M_R \inftensor_S M \bigr) \\
    & \HH_\k \bigl( S, \Hom_R(M', M) \bigr) \dar{\eq} \uar[swap]{\eq} & \lar[swap]{\eq} \HH_\k \bigl( S, M \inftensor_R \dual{(M')}_R \bigr) \dar{\eq} \uar[swap]{\eq} \rar{\iso} & \HH_\k \bigl( R, \dual{(M')}_R \inftensor_S M \bigr) \dar{\eq} \uar[swap]{\eq} \\
    \HH_\k(S, S') \rar \dar & \HH_\k \bigl( S, \Hom_R(M', M') \bigr) \dar & \lar[swap]{\eq} \HH_\k \bigl( S, M' \inftensor_R \dual{(M')}_R \bigr) \rar{\iso} \dar & \HH_\k \bigl( R, \dual{(M')}_R \inftensor_S M' \bigr) \dar \\
    \HH_\k(S', S') \rar & \HH_\k \bigl( S', \Hom_R(M', M') \bigr) & \lar[swap]{\eq} \HH_\k \bigl( S', M' \inftensor_R \dual{(M')}_R \bigr) \rar{\iso} & \HH_\k \bigl( R, \dual{(M')}_R \inftensor_{S'} M' \bigr)
  \end{tikzcd}
  \]
  The rest of the diagram commutes by functoriality of $\HH_\k$ (see \cref{obs:HH_functor}) and inspection (using \cref{obs:inftensor_restrict} and the naturality of \cref{obs:Hom_hf}).
  The indicated maps are quasi-isomorphisms by \cref{lemma:HH_qiso}.
  Noting that the right-hand vertical maps are compatible with the respective evaluation maps to $\HH_\k(R, R)$ completes the proof.
\end{proof}

\subsection{An explicit formula}

We now provide a formula for the Hochschild homology transfer $\transfer{M}$, assuming that $M$ is dualizable as an $R$-module.
This is the first main result of this paper.
It generalizes (and, the authors think, clarifies) a formula of Bouc \cite[§3]{Bou} for the special case that $R$ and $S$ are ordinary (non-dg) algebras (also see Loday \cite[§1.2]{Lod}).
Our formula depends on the choice of a lift of the coevaluation of $M$ as follows.
However, at the end of this subsection, we provide a simpler formula that is independent of this choice for the part of the Hochschild homology transfer that lands in Hochschild degree $0$.

\begin{notation}
  In the rest of this subsection, for a cdga $\k$ and a $\k$-algebra $R$, we abbreviate $\B_\k(R) \defeq \B_\k(R, R, R)$.
\end{notation}

\begin{definition} \label{def:derived_coev}
  Let $G$ be a group, $\k \to R$ a map of $G$-equivariant cdgas, and $M$ a dualizable $G$-equivariant $R$-module.
  A \emph{$G$-equivariant derived coevaluation} for $M$ is a map $c \colon \k \to M \tensor_R \B_\k(R) \tensor_R \dual M_R$ of $G$-equivariant $\k$-modules such that the following diagram commutes
  \[
  \begin{tikzcd}
    \k \rar{c} \dar & M \tensor_R \B_\k(R) \tensor_R \dual M_R \dar{\epsilon} \\
    R \rar{\coev_M} & M \tensor_R \dual M_R
  \end{tikzcd}
  \]
  where $\epsilon \colon \B_\k(R) \to R$ is the canonical augmentation.
\end{definition}

Note that a $G$-equivariant derived coevaluation is equivalently a cocycle of $M \tensor_R \B_\k(R) \tensor_R \dual M_R$ that is fixed by $G$ and whose image in $M \tensor_R \dual M_R$ is equal to $\coev_M(1)$.

\begin{remark} \label{rem:derived_coev}
  Note that a $G$-equivariant derived coevaluation always exists when the map
  \[ \epsilon  \colon  M \tensor_R \B_\k(R) \tensor_R \dual M_R  \longto  M \tensor_R \dual M_R \]
  is a quasi-isomorphism and its source and target are degreewise semi-simple $G$-representations.
  The former is the case when $R$ is cofibrant as a $\k$-module and $M$ is cofibrant as an $R$-module (by \cite[\Ilemmabarcomplex]{I}).
  The latter is the case when $G$ is finite, or when $G$ is a linearly reductive algebraic group and $\k$, $R$, and $M$ are degreewise finite-dimensional algebraic representations of $G$.
  Also note that, when the differentials of $R$ and $M$ are trivial, then $c$ can be chosen to land in $M \tensor_\k \dual{M}_R \subseteq M \tensor_R \B_\k(R) \tensor_R \dual M_R$.
\end{remark}

Using a derived coevaluation $c$, we now define a ``generalized trace'' from the Hochschild homology of the $\k$-algebra $\End_R(M)$ to the Hochschild homology of $R$, which is given by the usual trace $\tr_R$ in Hochschild degree $0$.
See \cite[§2.2]{Bou} and \cite[Definition~1.2.1]{Lod} for the classical definition (it depends on a choice of $c$ landing in $M \tensor_\k \dual{M}_R$, and thus only makes sense when the differentials of $R$ and $M$ are trivial).
Recall \cref{not:cyctensor}.

\begin{definition} \label{def:gtr}
  Let $\k \to R$ be a map of $G$-equivariant cdgas, $M$ a $G$-equivariant $R$-module, and $c$ a $G$-equivariant derived coevaluation for $M$.
  The \emph{generalized trace} is the map $\tr_R^c \colon \HH_\k(\End_R(M)) \to \HH_\k(R)$ given in Hochschild degree $n$ by the composite
  \[
  \begin{tikzcd}[row sep = 15]
    &[-70] \shift[-1] \bigl( \shift \End_R(M) \bigr)^{\cyctensor_\k n} \ar{rr}{({\id} \tensor c)^{\cyctensor n}} &[-60] &[25] \shift[-1] \bigl( \shift \End_R(M) \tensor_\k M \tensor_R \B_\k(R) \tensor_R \dual M_R \bigr)^{\cyctensor_\k n} \dar{\iso} \\
    \HH_\k(R) & & \ar{ll}[swap]{\pi_n} \shift[-1] \bigl( \shift \B_\k(R) \bigr)^{\cyctensor_R n} & \lar[swap]{({\id} \tensor \ev)^{\cyctensor n}} \shift[-1] \bigl( \shift \B_\k(R) \tensor_R \dual M_R \tensor_\k \End_R(M) \tensor_\k M \bigr)^{\cyctensor_R n}
  \end{tikzcd}
  \]
  where $\pi_n$ is given by
  \[ \textstyle { \shift[-1] \bigotimes_{i = 1}^n \shift \bigl( x_{i,0} \tensor \bigotimes_{j = 1}^{p_i} \shift x_{i,j} \tensor x_{i,p_i+1} \bigr)  \longmapsto  \pm \shift[-1] \bigotimes_{i = 1}^n \bigl( (x_{i-1, p_{i-1} + 1} \cdot \shift x_{i, 0}) \tensor \bigotimes_{j = 1}^{p_i} \shift x_{i,j} \bigr) } \]
  where $x_{i,j} \in R$, we define $x_{0, p_0 + 1} \defeq x_{n, p_n + 1}$, and the sign is determined by the Koszul rule.
\end{definition}

We are now ready to state and prove our formula for the Hochschild homology transfer, constituting the first main result of this paper.

\begin{theorem} \label{thm:Ainf_transfer}
  Let $G$ be a group, $\k \to R$ a map of $G$-equivariant cdgas, $S$ a $G$-equivariant $\Ainf$-algebra over $R$, and $M$ a $G$-equivariant left $S$-module.
  Assume that $R$ and $S$ are cofibrant as $\k$-modules, and that $M$ is cofibrant and dualizable as an $R$-module.
  Let $c$ be a $G$-equivariant derived coevaluation for $M$.
  Then the Hochschild homology transfer $\transfer{M} \colon \HH_\k(S) \to \HH_\k(R)$ is homotopic in $\Underlying(\Modeq{G}{\k})$ to the composite map of $G$-equivariant $\k$-modules
  \[ \HH_\k(S)  \xlongto{\upsilon_*}  \HH_\k \bigl( \End_R(M) \bigr)  \xlongto{\gtr[R]{c}}  \HH_\k(R) \]
  where $\upsilon$ is the map of \cref{lemma:upsilon_algebra_map} and $\tr^c_R$ is the generalized trace of \cref{def:gtr}.
\end{theorem}

\begin{observation} \label{obs:Ainf_transfer_explicit}
  The composite map ${\gtr[R]{c}} \after \upsilon_* \colon \HH_\k(S) \to \HH_\k(R)$ of \cref{thm:Ainf_transfer} is given, in Hochschild degree $n$, by sending $\shift[-1] (x_0 \otimes x_1 \otimes \dots \otimes x_n)$ to
  \[ \shift[-1] \sum_{\substack{n_0 + \dots + n_p = n \\ 0 < n_i \text{ for } 0 < i < p}} (-1)^\epsilon \gtr[R]{c} \bigl( \shift \mu^M_{n_p + 1 + n_0}(\vec x_p \tensor x_0 \tensor \vec x_0 \tensor \blank) \tensor \textstyle\bigotimes_{i = 1}^{p - 1} \shift \mu^M_{n_i}(\vec x_i \tensor \blank) \bigr) \]
  where $\vec x_i \defeq x_{n_0 + \dots + n_{i-1} + 1} \tensor \dots \tensor x_{n_0 + \dots + n_{i-1} + n_i}$ and $\epsilon \defeq (\deg {x_0} + \deg {\vec x_0} + \dots + \deg {\vec x_{p-1}}) \deg {\vec x_p}$ and each summand takes values in Hochschild degree $p - 1$.
\end{observation}

\begin{proof}[Proof of \cref{thm:Ainf_transfer}]
  First observe that $M$ is homotopically finite as an $R$-module by \cref{lemma:dualizable_finite}, so that the transfer $\transfer{M}$ indeed exists.
  
  We begin by proving that $\gtr[R]{c}$ is indeed a map of $G$-equivariant $\k$-modules.
  The map
  \[ {\id} \tensor c  \colon  \End_R(M)  \longto  \End_R(M) \tensor_\k M \tensor_R \B_\k(R) \tensor_R \dual M_R \]
  is a map of non-unital $\k$-algebras when the target is equipped with the multiplication
  \[ (\alpha \tensor m \tensor b \tensor \phi) \cdot (\alpha' \tensor m' \tensor b' \tensor \phi')  \defeq  \bigl( \alpha \after \Phi(m \tensor b \tensor \phi) \after \alpha' \bigr) \tensor m' \tensor b' \tensor \phi' \]
  where $\Phi$ is the composite map
  \[ M \tensor_R \B_\k(R) \tensor_R \dual M_R  \xlongto{\epsilon}  M \tensor_R \dual M_R  \longto  \End_R(M) \]
  (note that $(\Phi \after c)(1) = \id_M \in \End_R(M)$ by definition of $c$).
  Similarly
  \[ {\id} \tensor \ev  \colon  \B_\k(R) \tensor_R \dual M_R \tensor_\k \End_R(M) \tensor_\k M  \longto  \B_\k(R) \]
  is a map of non-unital monoids in $(R \tensor_\k \opmod{R})$-modules when the source is equipped with the multiplication
  \[ (b \tensor \phi \tensor \alpha \tensor m) \cdot (b' \tensor \phi' \tensor \alpha' \tensor m')  \defeq  b \tensor \phi \tensor \bigl( \alpha \after \Phi(m \tensor b' \tensor \phi') \after \alpha' \bigr) \tensor m' \]
  and the target with the multiplication $b \cdot b' \defeq b \epsilon(b')$.
  We thus obtain a composite map of $G$-equivariant $\k$-modules
  \[
  \begin{tikzcd}[row sep = 15]
    \HH_\k \bigl( \End_R(M) \bigr) \rar{({\id} \tensor c)_*} &[10] \HH_\k \bigl( \End_R(M) \tensor_\k M \tensor_R \B_\k(R) \tensor_R \dual M_R \bigr) \dar{\iso} \\
    \HH_R \bigl( \B_\k(R) \bigr) & \lar[swap]{({\id} \tensor \ev)_*} \HH_R \bigl( \B_\k(R) \tensor_R \dual M_R \tensor_\k \End_R(M) \tensor_\k M \bigr)
  \end{tikzcd}
  \]
  (recall the Hochschild homology of monoids in bimodules from \cref{def:HH_bimodule}).
  Hence it only remains to show that the $G$-equivariant maps $\pi_n \colon \shift[-1] ( \shift \B_\k(R) )^{\cyctensor_R n} \to  \HH_\k(R)$ assemble into a map $\pi \colon \HH_R(\B_\k(R)) \to \HH_\k(R)$ of $\k$-modules.
  It is clear that the map is $\k$-linear, and it follows from unwinding the definitions that it is a map of cochain complexes.
  (Note that the summands of the Hochschild differential of the source correspond to multiplying $x_{i,0}$ and $x_{i, p_i+1}$ when $p_i = 0$.)
  Thus $\gtr[R]{c}$ is indeed a map of $G$-equivariant $\k$-modules.
  
  Observe that $M$ is canonically a left $\End_R(M)$-module.
  We will now prove that $\gtr[R]{c}$ is homotopic to $\transfer{M} \colon \HH_\k(\End_R(M)) \to \HH_\k(R)$.
  Using the shorthand $E \defeq \End_R(M)$, we have the following commutative diagram
  \[
  \begin{tikzcd}[column sep = -35]
    & \HH_\k(E) & \lar[swap]{\id} \HH_\k (E) \dar{(\id \tensor c)_*} \\
    \HH_\k (E, E) \rar{\eq}[swap]{\inc} \urar[bend left = 15]{\id} & \HH_E \bigl( \B_\k(E) \bigr) \uar[swap]{\pi} & \HH_\k \bigl(E \tensor_\k M \tensor_R \B_\k(R) \tensor_R \dual{M}_R \bigr) \dar{({\ev} \tensor {\id} \tensor \id)_*} \\
    \HH_\k \bigl( E, M \tensor_R \B_\k(R) \tensor_R \dual{M}_R \bigr) \uar{\Phi_*}[swap]{\eq} \drar[end anchor = north west][swap]{\inc} \ar{dd}{\iso} & & \HH_\k \bigl(M \tensor_R \B_\k(R) \tensor_R \dual{M}_R \bigr) \dlar[end anchor = north east]{\iota} \ar{dd}[swap]{\iso} \\
     & \HH_{E} \bigl( \B_\k(E) \tensor_{E} M \tensor_R \B_\k(R) \tensor_R \dual{M}_R \bigr) \ar{uu}{\eq}[swap]{({\id} \tensor \Phi)_*} \ar{dd}[swap]{\iso} & \\[-10]
    \HH_\k \bigl( R, \dual{M}_R \tensor_{E} \B_\k(E) \tensor_E M \bigr) \drar[end anchor = north west][swap]{\inc} \ar{dd}[swap]{\Psi_*} & & \HH_R \bigl( \B_\k(R) \tensor_R \dual{M}_R \tensor_\k M \bigr) \dlar[end anchor = north east]{\iota} \ar{dd}{({\id} \tensor \ev)_*} \\
     & \HH_R \bigl( \B_\k(R) \tensor_R \dual{M}_R \tensor_{E} \B_\k(E) \tensor_{E} M \bigr) \dar{({\id} \tensor \Psi)_*} & \\
    \HH_\k (R, R) \rar{\eq}[swap]{\inc} \drar[bend right = 15][swap]{\id} & \HH_R \bigl( \B_\k(R) \bigr) \rar[equal] \dar{\pi} & \dlar[bend left = 15]{\pi} \HH_R \bigl( \B_\k(R) \bigr) \\
     & \HH_\k(R) &
  \end{tikzcd}
  \]
  of $G$-equivariant $\k$-modules.
  The maps labeled $\inc$ are given by the canonical inclusions of the form
  \[ \HH_\k(A, N)  \iso  \bigl( \B_\k(A) \tensor_A N \bigr)^{\cyctensor_A 1}  \longto  \HH_A \bigl( \B_\k(A) \tensor_A N \bigr) \]
  where $\B_\k(A) \tensor_A N$ is equipped with the multiplication
  \[ (b \tensor n) \cdot (b' \tensor n')  \defeq  b \kappa(n) \epsilon(b') \tensor n' \]
  for the respective obvious map $\kappa \colon N \to A$ of $(A \tensor_\k \opmod{A})$-modules.
  The map $\Psi$ is the composite
  \[ \dual{M}_R \tensor_{E} \B_\k(E) \tensor_E M  \xlongto{\epsilon}  \dual{M}_R \tensor_E M  \xlongto{\ev}  R \]
  and the maps labeled $\iota$ are induced by the canonical inclusion $E \tensor_\k E \to \B_\k(E)$.
  The diagram commutes by inspection.
  The two indicated maps labeled $\inc$ are quasi-isomorphisms by \cref{lemma:inc_qiso} (using that $E$ is cofibrant as an $R$-module by \cref{lemma:Hom_cofibrant}, and hence as a $\k$-module by \cite[\Ilemmarestrictcofibration]{I}).
  The map $\Phi$ is a quasi-isomorphism by \cite[\Ilemmabarcomplex, \Ilemmacofibrantflat, and \IlemmaHomqiso]{I} since $R$ is cofibrant as a $\k$-module and $M$ is a cofibrant $R$-module.
  Hence the map $\Phi_*$ is a quasi-isomorphism by \cref{lemma:HH_qiso}.
  We claim that $({\id} \tensor \Phi)_*$ is also a quasi-isomorphism.
  To see this, filter both source and target by the total number the tensor factor $E$ appears and note that the map is filtration preserving.
  On associated graded, the map is of the form
  \[ \bigoplus_{p \ge 0} \bigoplus_{n \ge 0} E^{\tensor_\k p} \tensor_\k \bigl( M \tensor_R \B_\k(R) \tensor_R \dual{M}_R \bigr)^{\tensor_\k n}  \xlongto{{\id}^{\tensor p} \tensor \Phi^{\tensor n}}  \bigoplus_{p \ge 0} \bigoplus_{n \ge 0} E^{\tensor_\k p} \tensor_\k E^{\tensor_\k n} \]
  which is a quasi-isomorphism since $E$ and $M \tensor_R \B_\k(R) \tensor_R \dual{M}_R$ are cofibrant as $\k$-modules (again using \cref{lemma:Hom_cofibrant} to see that $\dual{M}_R$ is cofibrant as an $R$-module).
  This implies the claim.
  
  By the diagram above, the lower right-hand hexagon of the diagram
  \[
  \begin{tikzcd}
    & \dlar[bend right = 17][swap]{\nu_*} \HH_\k(S) \dar{\upsilon_*} \drar[bend left = 17]{\upsilon_*} \\
    \HH_\k ( S, E ) \rar{\upsilon_*} & \HH_\k (E) \rar[equal] & \HH_\k(E) \ar{ddd}{\gtr[R]{c}} \\
    \HH_\k \bigl( S, M \inftensor_R \dual M_R \bigr) \uar{\Phi_*}[swap]{\eq} \dar{\iso} \rar{\upsilon_*} & \HH_\k \bigl( E, M \inftensor_R \dual M_R \bigr) \uar{\Phi_*}[swap]{\eq} \dar{\iso} \\
    \HH_\k \bigl( R, \dual M_R \inftensor_S M \bigr) \dar[swap]{\ev_*} \rar{\upsilon_*} & \HH_\k \bigl( R, \dual M_R \inftensor_E M \bigr) \dar[swap]{\ev_*} \\
    \HH_\k(R) \rar[equal] & \HH_\k(R) \rar[equal] & \HH_\k(R)
  \end{tikzcd}
  \]
  commutes in $\Underlying(\Modeq{G}{\k})$.
  To see that the rest of the diagram commutes, we observe that $\upsilon^*(M) = M$ using \cref{obs:restrict_bimodule} (where $M$ is considered as a left $E$-module on the left-hand side and as a left $S$-module on the right-hand side), and that the construction of \cref{lemma:algebra_map_is_bimodule_map} applied to the map $\upsilon$ of \cref{lemma:upsilon_algebra_map} yields the map $\nu$ of \cref{obs:nu}.
  Since the left-hand composite in the preceding diagram is the transfer $\transfer{M}$, this finishes the proof.
\end{proof}

\subsubsection*{The part landing in Hochschild degree $0$}

The part of the generalized trace $\gtr[R]{c}$ that lands in Hochschild degree $0$ is given by the usual $R$-linear trace $\tr_R$; in particular it is independent of the choice of the $G$-equivariant derived coevaluation $c$.
Hence the same is true for the formula of \cref{thm:Ainf_transfer}.
We will now prove that the resulting description for the part landing in Hochschild degree $0$ holds even if no $G$-equivariant derived coevaluation exists.

\begin{theorem} \label{thm:main}
  Let $G$ be a group, $\k \to R$ a map of $G$-equivariant cdgas, $S$ a $G$-equivariant $\Ainf$-algebra over $R$, and $M$ a $G$-equivariant left $S$-module.
  Assume that $R$ and $S$ are cofibrant as $\k$-modules, and that $M$ is cofibrant and dualizable as an $R$-module.
  Then the following composite of the Hochschild homology transfer and the projection
  \[ \HH_\k(S)  \xlongto{\transfer{M}}  \HH_\k(R)  \longto  R \]
  is homotopic in $\Underlying(\Modeq{G}{\k})$ to the map of $G$-equivariant $\k$-modules $\tr_M \colon \HH_\k(S) \to R$ given in Hochschild degree $n$ by
  \[ \shift[-1] (x_0 \otimes x_1 \otimes \dots \otimes x_n)  \longmapsto  \sum_{i = 0}^n (-1)^\epsilon \tr_R \bigl( \mu^M_{n+1}(x_i, x_{i+1}, \dots, x_n, x_0, \dots, x_{i - 1}, \blank) \bigr) \]
  where $\epsilon \defeq (\deg {x_0} + \dots + \deg {x_{i - 1}}) (\deg {x_i} + \dots + \deg {x_n})$.
\end{theorem}

\begin{proof}
  Note that $M$ is homotopically finite as an $R$-module by \cref{lemma:dualizable_finite}, so that the transfer $\transfer{M}$ indeed exists.
  Let $\tau \colon \HH_\k ( S, \Hom_R(M, M) ) \to R$ be the map that is trivial in Hochschild degree $> 0$ and given by the trace $\tr_R$ in Hochschild degree $0$.
  Note that it is a map of cochain complexes (and thus of $\k$-modules): by \cref{obs:HH} and the proof of \cref{lemma:Ainf_Hom}, the part of the differential of the source going from Hochschild degree $n$ to Hochschild degree $0$ is given by
  \[ \alpha \tensor x_1 \tensor \dots \tensor x_n  \longmapsto  - (-1)^{\deg \alpha} \alpha \after \mu^M_n(x_1, \dots, x_n, \blank)  + (-1)^{\deg \alpha \deg {\vec x}} \mu^M_n(x_1, \dots, x_n, \blank) \after \alpha \]
  where $\vec x \defeq x_1 \tensor \dots \tensor x_n$.
  This image is killed by $\tr_R$ due to its cyclic invariance.
  
  Chasing through the definitions, we see that the following diagram of $G$-equivariant $\k$-modules
  \[
  \begin{tikzcd}[column sep = 14]
    \shift[-1] \HH_\k(S, \shift S) \rar{\nu_*} & \HH_\k \bigl( S, \Hom_R(M, M) \bigr) \drar[bend right = 15][swap]{\tau} & \lar[swap]{\eq} \HH_\k \bigl( S, M \inftensor_R \dual M_R \bigr) \rar{\iso} & \HH_\k \bigl( R, \dual M_R \inftensor_S M \bigr) \dar{\ev_*} \\
    & & R & \lar \HH_\k(R, R)
  \end{tikzcd}
  \]
  commutes.
  By \cref{obs:nu}, the composite $\tau \after \nu_*$ is given by the map $\tr_M$ in the statement.
  This finishes the proof.
\end{proof}

%% file: 2model.tex
\section{Rational models for the fiberwise THH transfer}

In \cite{I}, we proved that the fiberwise THH transfer (cf.\ \cref{sec:THH}) is modeled by the Hochschild homology transfer of cdgas.
In this section, we combine this with the results of \cref{sec:transfer} to obtain an explicit rational model of the fiberwise THH transfer in terms of $\Ainf$-models.
Further combining this with forthcoming work of Berglund \cite{Ber}, who uses dg Lie models to construct $\Ainf$-models as required, we also deduce an explicit rational model of the fiberwise THH transfer in terms of dg Lie models.
We begin by deducing from a qualitative version of Berglund's result that $\Ainf$-models of the form we need always exist.

\subsection{Existence of equivariant \texorpdfstring{$\Cinf$}{C∞}-models} \label{sec:eq_models_exist}

In this subsection, we prove that maps of connected animas with compact simply connected fiber always admit equivariant $\Cinf$-algebra models that are dualizable, as required by our results in \cref{sec:model_Ainf} below.
This is an analogue of the classical fact that an anima admits a finite-dimensional $\Cinf$-model when its rational cohomology is finite-dimensional (which follows from the Homotopy Transfer Theorem).

We deduce the statement from a forthcoming result of Berglund \cite{Ber}: building on work of Berglund--Zeman \cite{BZ}, he constructs a finite-dimensional equivariant relative $\Cinf$-model for the universal fibration with fiber a compact simply connected anima.
More precisely, Berglund in particular proves the following.
(Also see \cref{lemma:Lie_Cinf_model} below for the explicit form of his construction in the non-equivariant case.)

\begin{proposition}[Berglund] \label{prop:eq_Cinf_universal}
  Let $F$ be a compact simply connected anima.
  Then there exists a quotient $\Gamma$ of the group $\pi_0(\aut(\rat F))$, a cofibration $\k \to R$ of $\Gamma$-equivariant cofibrant cdgas, and a quasi-isomorphism $R_\infty \to R$ of $\Gamma$-equivariant $\Cinf$-algebras over $\k$ such that $R_\infty$ is finite-dimensional semifree over $\k$ and the map $\k \to R$ models the universal fibration $u \colon \rat F / \aut(\rat F) \to \B\aut(\rat F)$ as fiberwise nilpotent animas over $\B\Gamma$ of fiberwise finite rational type.
\end{proposition}

\begin{proof}
  This is contained in the forthcoming work \cite{Ber}.
\end{proof}

\begin{corollary} \label{cor:eq_Cinf_existence}
  Let $p \colon E \to B$ be a map of connected animas with compact simply connected fiber $F$.
  Assume that the universal covering of $B$ is of finite rational type.
  Then there exists a cofibration $\k \to R$ of $\pi_1(B)$-equivariant cofibrant cdgas, and a quasi-isomorphism $R_\infty \to R$ of $\pi_1(B)$-equivariant $\Cinf$-algebras over $\k$ such that $R_\infty$ is finite-dimensional semifree over $\k$ and the map $\k \to R$ models the map $p$ as fiberwise nilpotent animas over $\B \pi_1(B)$.
\end{corollary}

\begin{proof}
  Choose $\Gamma'$, $\k' \to R'$, and $R'_\infty \to R'$ as in \cref{prop:eq_Cinf_universal}.
  Consider the commutative diagram
  \[
  \begin{tikzcd}
    E \rar \dar[swap]{p} & E' \rar \dar{u'} & \rat F / \aut(\rat F) \dar{u} \\
    B \rar \dar & B' \rar \dar & \B \aut(\rat F) \dar \\
    \B \pi_1(B) \rar[equal] & \B \pi_1(B) \rar & \B \Gamma'
  \end{tikzcd}
  \]
  where $u'$ is defined to be the pullback of $u$ along the map $\B \pi_1(B) \to \B \Gamma'$ induced by the composite $B \to \B \aut(F) \to \B \aut(\rat F)$.
  The restriction of $\k' \to R'$ along the map $\pi_1(B) \to \Gamma'$ yields a $\pi_1(B)$-equivariant model for $u'$ as a map of fiberwise nilpotent animas over $\B \pi_1(B)$.
  Since $\k'$ is cofibrant and $B \to B'$ is a map of fiberwise nilpotent animas over $\B \pi_1(B)$ of fiberwise finite rational type, it can be modeled by a cofibration $\k' \to \k$ of $\pi_1(B)$-equivariant cdgas.
  By \cite[\Iobspullbackeq]{I}, since the fiber of $u'$ is connected, the pullback $p' \colon E'' \to B$ of $u'$ along $B \to B'$ is then modeled by the map $\k \to R \defeq R' \tensor_{\k'} \k$ of $\pi_1(B)$-equivariant cdgas.
  There is an induced map $f \colon E \to E''$, and considering the commutative diagram of horizontal fiber sequences
  \[
  \begin{tikzcd}
    F \rar \dar & E \rar \dar{f} & B \dar[equal] \\
    \rat F \rar & E'' \rar & B
  \end{tikzcd}
  \]
  we see that $f$ is a fiberwise rational equivalence of fiberwise simply connected animas over $\B \pi_1(B)$.
  Hence $p$ is also modeled by $\k \to R$.
  Lastly we note that the induced map $R_\infty \defeq R_\infty' \tensor_{\k'} \k \to R$ is a quasi-isomorphism of $\pi_1(B)$-equivariant $\Cinf$-algebras over $\k$, and that $R_\infty$ is finite-dimensional semifree over $\k$.
\end{proof}

\subsection{A rational model for the fiberwise THH transfer using \texorpdfstring{$\Ainf$}{A∞}-algebras} \label{sec:model_Ainf}

Combining the main result of \cite{I} with \cref{thm:Ainf_transfer}, we now provide an explicit model for the fiberwise THH transfer.
Also see \cref{obs:Ainf_transfer_explicit} where we spell out the resulting map more explicitly.
Note that an equivariant $\Ainf$-model as required always exists by \cref{cor:eq_Cinf_existence}, and that a $G$-equivariant derived coevaluation as required often exists by \cref{rem:derived_coev}.
For example this is the case for the equivariant model for the universal fibration constructed by Berglund \cite{Ber} (see \cref{prop:eq_Cinf_universal}), since there $G$ is an arithmetic subgroup of a reductive group and all actions are through finite-dimensional algebraic representations.

\begin{corollary}
  Let $G$ be a group, and let the following two left-hand maps be maps of fiberwise nilpotent animas over $\B G$
  \[ X \xlongto{f} Y \xlongto{p_Y} B  \qquad  S \xlongfrom{\phi} R \xlongfrom{\iota_R} \k  \qquad  \]
  that are modeled by the two right-hand cofibrations of $G$-equivariant homologically connected cofibrant cdgas of finite homotopical type.
  Assume that the fiber of $p_Y$ is simply connected and that the fiber of $f$ is simply connected and compact.
  Furthermore let $S_\infty \to S$ be a unital $G$-equivariant quasi-isomorphism of $\Ainf$-algebras over $R$ such that $S_\infty$ is cofibrant and dualizable as an $R$-module, and let $c$ be a $G$-equivariant derived coevaluation for $\shift S_\infty$ (cf.\ \cref{def:derived_coev}).
  Then the fiberwise THH transfer $f^* \colon \THH_B(Y) \to \THH_B(X)$ is modeled by $-1$ times the composite map of $G$-equivariant $\k$-modules
  \[ \HH_\k(S_\infty)  \xlongto{\upsilon_*}  \HH_\k \bigl( \End_R(\shift S_\infty) \bigr)  \xlongto{\tr^c_R}  \HH_\k(R) \]
  where $\upsilon$ is the map of \cref{lemma:upsilon_algebra_map} (for the left $S_\infty$-module $\shift S_\infty$ of \cref{obs:algebra_is_bimodule}) and $\tr^c_R$ is the generalized trace of \cref{def:gtr}.
\end{corollary}

\begin{proof}
  By \cite[\Ithmmain]{I}, the fiberwise THH transfer $f^*$ is modeled by the Hochschild homology transfer $\transfer{S} \colon \HH_\k(S) \to \HH_\k(R)$.
  Note that $\transfer{S} = - \transfer{\shift S}$.
  By \cref{prop:transfer_natural}, the transfer $\transfer{\shift S}$ is homotopic in $\Underlying(\Modeq{G}{\k})$ to $\transfer{\shift S_\infty}$.
  Then \cref{thm:Ainf_transfer} implies the claim.
\end{proof}

Given a map $f \colon X \to Y$ of animas over $B$, the following provides a simpler model for the fiberwise THH transfer of $f$, precomposed with the fiberwise assembly map.
When $Y = B$, the latter is an equivalence, so that we obtain a model of the fiberwise THH transfer itself.

\begin{corollary} \label{cor:main}
  Let $G$ be a group, and let the following two left-hand maps be maps of fiberwise nilpotent animas over $\B G$
  \[ X \xlongto{f} Y \xlongto{p_Y} B  \qquad  S \xlongfrom{\phi} R \xlongfrom{\iota_R} \k  \qquad  \]
  that are modeled by the two right-hand cofibrations of $G$-equivariant homologically connected cofibrant cdgas of finite homotopical type.
  Assume that the fiber of $p_Y$ is simply connected and that the fiber of $f$ is simply connected and compact.
  Furthermore let $S_\infty \to S$ be a unital $G$-equivariant quasi-isomorphism of $\Ainf$-algebras over $R$ such that $S_\infty$ is cofibrant and dualizable as an $R$-module.
  Then the following composite of the fiberwise assembly map and the fiberwise THH transfer
  \[ \SS_B[Y]  \xlongto{\alpha}  \THH_B(Y)  \xlongto{f^*}  \THH_B(X) \]
  is modeled by the map of $G$-equivariant $\k$-modules $\tr \colon \HH_\k(S_\infty) \to R$ given in Hochschild degree $n$ by
  \[ \shift[-1] (x_0 \otimes x_1 \otimes \dots \otimes x_n)  \longmapsto  \sum_{i = 0}^n (-1)^{\epsilon + 1} \tr_R \bigl( \mu^{S_\infty}_{n+2}(x_i, x_{i+1}, \dots, x_n, x_0, \dots, x_{i - 1}, \blank) \bigr) \]
  where $\epsilon \defeq (\deg {x_0} + \dots + \deg {x_{i - 1}}) (\deg {x_i} + \dots + \deg {x_n})$.
\end{corollary}

\begin{proof}
  By \cite[\Ithmmain{} and \Ilemmaassemblymodel]{I}, the composite $f^* \after \alpha$ is modeled by the composite
  \begin{equation} \label{eq:transfer_pr}
    \HH_\k(S)  \xlongto{\transfer{S}}  \HH_\k(R)  \xlongto{\pr}  R
  \end{equation}
  where $\pr$ is the projection.
  Note that $\transfer{S} = - \transfer{\shift S}$.
  Hence \eqref{eq:transfer_pr} is homotopic in $\Underlying(\Modeq{G}{\k})$ to $-1$ times the analogous composite
  \[ \HH_\k(S_\infty)  \xlongto{\transfer{\shift S_\infty}}  \HH_\k(R)  \xlongto{\pr}  R \]
  by \cref{prop:transfer_natural}.
  This composite is given by the claimed formula by \cref{thm:main} and \cref{obs:algebra_is_bimodule}.
\end{proof}

%
%

\subsubsection*{Auxiliary maps}

To conclude this subsection, we provide the corresponding models for the fiberwise assembly map $\SS_B[X] \to \THH_B(X)$ and the map in the other direction induced by evaluation of loops.

\begin{lemma} \label{lemma:eval_model_Ainf}
  Let $G$ be a group and $p \colon X \to B$ a map of fiberwise nilpotent animas over $\B G$ that is modeled by a cofibration $\k \to R$ of $G$-equivariant homologically connected cofibrant cdgas of finite homotopical type.
  Assume that the fiber of $p$ is simply connected, and let $R_\infty \to R$ be a $G$-equivariant quasi-isomorphism of $\Ainf$-algebras over $\k$ such that $R_\infty$ is cofibrant as a $\k$-module.
  Then the following map, induced by evaluation at the basepoint of $\Sphere 1$,
  \[ \THH_B(X) \eq \SS_B[\L_B X]  \xlongto{e}  \SS_B[X] \]
  is modeled by the inclusion $R_\infty \to \HH_\k(R_\infty)$.
\end{lemma}

\begin{proof}
  By \cite[\Ilemmaevalmodel]{I} the map is modeled by the inclusion $R \to \HH_\k(R)$.
  Since this map is clearly natural in maps of $\Ainf$-algebras, this implies the claim by \cref{lemma:HH_qiso}.
\end{proof}

Note that in the following lemma we require a $\Cinf$-model, whereas everything else we have done so far worked for $\Ainf$-models.

\begin{lemma} \label{lemma:assembly_model_Cinf}
  Let $G$ be a group and $p \colon X \to B$ a map of fiberwise nilpotent animas over $\B G$ that is modeled by a cofibration $\k \to R$ of $G$-equivariant homologically connected cofibrant cdgas of finite homotopical type.
  Assume that the fiber of $p$ is simply connected, and let $R_\infty \to R$ be a $G$-equivariant quasi-isomorphism of $\Cinf$-algebras over $\k$ such that $R_\infty$ is cofibrant as a $\k$-module.
  Then the fiberwise assembly map $\SS_B[X] \to \THH_B(X)$ is modeled by the map of $G$-equivariant $\k$-modules $\epsilon \colon \HH_\k(R_\infty) \to R_\infty$ given by projecting to Hochschild degree $0$.
\end{lemma}

\begin{proof}
  By \cite[\Ilemmaassemblymodel]{I}, the fiberwise assembly map is modeled by the projection $\HH_\k(R) \to R$.
  Then \cref{lemma:HH_proj,lemma:HH_qiso} imply the claim.
\end{proof}

\subsection{A rational model for the fiberwise THH transfer using dg Lie algebras} \label{sec:Lie_models}

In forthcoming work, Berglund \cite{Ber} constructs, given a dg Lie model for a fibration $p \colon E \to B$, a relative $C_\infty$-models for $p$ as required by \cref{cor:main}.
See \cref{prop:eq_Cinf_universal}, where we stated a qualitative version of his result.
In this subsection, we recall (the non-equivariant version of) his construction in more detail, and identify the map resulting from \cref{cor:main} in that case.
It turns out to be related to the trace map of Enomoto--Satoh \cite{ES}, see \cref{rem:ES} below.
For basic background on dg Lie models, see for example \cite[Part~IV]{FHT}.
We begin by fixing the following notation throughout this subsection.

\begin{notation} \label{not:X}
  We fix a simply connected anima $X$ of finite rational type, and introduce
  \[ H_* \defeq \Ho * (X; \QQ)  \qquad  \wt{H}_* \defeq \rHo * (X; \QQ)  \qquad  H^* \defeq \Coho * (X; \QQ)  \qquad  \wt{H}^* \defeq \rCoho * (X; \QQ) \]
  as shorthand notation.
\end{notation}

\begin{definition}
  For a dg Lie algebra $L$, we define $L_+$ to be the dg Lie algebra whose underlying graded Lie algebra is the free product $L * \LL(\eta)$ with $\eta$ in degree $-1$, and whose differential $d_+$ is determined by
  \[ d_+(\eta)  \defeq  - \frac{1}{2} [\eta, \eta]  \qquad \text{and} \qquad  d_+(x) \defeq d(x) - [\eta, x] \]
  for all $x \in L$, where $d$ is the differential of $L$.
\end{definition}

\begin{observation} \label{obs:Cinf_from_Lie}
  Let $L$ be the minimal Lie model of $X$.
  Note that $L$ is isomorphic to the free graded Lie algebra $\LL(\shift[-1]{\wt{H}_*})$ equipped with some differential.
  Then $L_+$ is isomorphic to $\LL(\shift[-1] H_*)$ equipped with some differential.
  Its (non-unital) universal enveloping algebra $\rU(L_+)$ is thus isomorphic to the (non-unital) tensor algebra $\rT(\shift[-1] H_*)$ equipped with some differential.
  The induced differential on the dual restricts to the tensor coalgebra $\rcoT(\shift H^*) \subset \dual{\rT(\shift[-1] H_*)}$.
  This encodes a unital $\Cinf$-algebra structure on $H^*$ with unit $1 \in H^0$, which is a $\Cinf$-model for $X$ (see e.g.\ \cite[§2.10]{BS23}).
\end{observation}

Recall that an action of a dg Lie algebra $\lie g$ on a dg Lie algebra $L$ is a map of dg Lie algebras $\lie g \to \Der(L)$; more generally, an outer action of $\lie g$ on $L$ is a map of dg Lie algebras
\[ (\phi, \xi)  \colon  \lie g  \longto  \Der(L) \ltimes_{\ad} \shift L \]
(see e.g.\ \cite[§3.5]{Ber20}).
Furthermore recall that, when $L$ is a dg Lie model for $X$, then outer actions of a nilpotent dg Lie algebra $\lie g$ on $L$ can be used to model fibrations $p \colon E \to B$ of nilpotent animas with fiber $X$ (or more precisely the holonomy action of $\Loops B$ on $X$), see e.g.\ \cite[Lemma~2.11.18]{BS24}.
The following appears in forthcoming work of Berglund \cite{Ber}.
We write $\CoCE(\lie g; M)$ for the Chevalley--Eilenberg cochain complex of a dg Lie algebra $\lie g$ with coefficients in a $\lie g$-module $M$.

\begin{proposition}[Berglund] \label{lemma:Lie_Cinf_model}
  Using \cref{not:X}, let $p \colon E \to B$ be a map of nilpotent animas with fiber $X$, let $L$ be the minimal dg Lie model of $X$, and let $\lie g$ be a nilpotent dg Lie algebra equipped with an outer action $(\phi, \xi)$ on $L$ that models $p$.
  Then $\CoCE(\lie g)$ is a cdga model for $B$, and the $\CoCE(\lie g)$-module $\CoCE(\lie g; H^*)$ can be equipped with the structure of a unital $\Cinf$-algebra over $\CoCE(\lie g)$ such that it models $E$.
  Here $\lie g$ acts on $H^*$ by the dual of the linear part of its action on $L$.
  The $\CoCE(\lie g)$-linear $\Ainf$-structure maps of $\CoCE(\lie g; H^*)$ are such that the restriction
  \begin{equation} \label{eq:Lie_Cinf_model_restr}
    (\shift H^*)^{\tensor n}  \longto  \CoCE \bigl( \lie g; (\shift H^*)^{\tensor n} \bigr)  \iso  \bigl( \shift \CoCE(\lie g; H^*) \bigr)^{\tensor_{\CoCE(\lie g)} n}  \xlongto{\mu_n}  \shift \CoCE(\lie g; H^*)
  \end{equation}
  is the sum $m_n + \psi_n$, where $m_n \colon (\shift H^*)^{\tensor n} \to \shift H^*$ is the corresponding structure map of the $\Cinf$-algebra of \cref{obs:Cinf_from_Lie}, and $\psi_n \colon (\shift H^*)^{\tensor n} \to \Hom_\QQ(\lie g, \shift H^*)$ is the adjoint of the map
  \[ x_1 \tensor \dots \tensor x_n \tensor \theta  \longmapsto  \sum_i \langle x_1 \otimes \dots \otimes x_n, \phi(\theta)(\dual{e_i}) \rangle \cdot e_i + (-1)^{\deg \theta} \langle x_1 \otimes \dots \otimes x_n, \shift[-1] \xi(\theta) \rangle \cdot \shift 1 \]
  where the $e_i$ are a homogeneous basis of $\shift \wt{H}^*$, the $\dual{e_i}$ are the corresponding dual basis of $\dual{(\shift \wt{H}^*)} \iso \shift[-1] \wt{H}_*$, and $\langle \blank, \blank \rangle$ denotes the canonical pairing $\coT(\shift H^*) \tensor \T(\shift[-1] H_*) \to \QQ$.
  (In the formula we consider $\dual{e_i} \in \shift[-1] \wt{H}_*$ as an element of $L$; both $\phi(\theta)(\dual{e_i})$ and $\shift[-1] \xi(\theta)$ lie in $L$, which we interpret as a subspace of $\T(\shift[-1] H_*)$.)
\end{proposition}

\begin{proof}[Proof sketch]
  Berglund first shows that $p$ is modeled by the map of cdgas
  \[ \CoCE(\lie g)  \longto  \CoCE \bigl( \lie g; \CoCE(L) \bigr) \]
  for a specific action of $\lie g$ on $\CoCE(L)$.
  He then equips $\CoCE(\lie g; H^*)$ with the structure of a unital $\Cinf$-algebra over $\CoCE(\lie g)$, and constructs a unital quasi-isomorphism $\CoCE(\lie g; H^*) \to \CoCE(\lie g; \CoCE(L))$ of $\Cinf$-algebras over $\CoCE(\lie g)$.
  We will now explain the construction of the $\Cinf$-algebra structure.
  
  The outer action $(\phi, \xi)$ of $\lie g$ on $L$ induces an action of $\lie g$ on $L_+$ by letting $\lie g$ act on $\eta$ via the map $\xi \colon \lie g \to \shift L$.
  Following the steps of \cref{obs:Cinf_from_Lie}, we obtain a representation of $\lie g$ on $\rT(\shift[-1] H_*)$, which restricts to $\rcoT(\shift H^*) \subset \dual{\rT(\shift[-1] H_*)}$.
  The $\Ainf$-algebra structure on $\CoCE(\lie g; H^*)$ is now uniquely determined by requiring that the canonical isomorphism of graded coalgebras
  \[ \rcoT_{\CoCE(\lie g)} \bigl( \shift \CoCE(\lie g; H^*) \bigr)  \iso  \CoCE \bigl( \lie g; \rcoT(\shift H^*) \bigr) \]
  is an isomorphism of cochain complexes.
  Berglund proves that this in fact makes $\CoCE(\lie g; H^*)$ into a unital $\Cinf$-algebra over $\CoCE(\lie g)$ (with unit $1 \in H^0$).
  
  Note that the differential of $\CoCE ( \lie g; \rcoT(\shift H^*) )$ consists of three parts: one part coming from the differential of $\lie g$, one from the differential of $\rcoT(\shift H^*)$ corresponding to the $\Ainf$-algebra structure on $H^*$ of \cref{obs:Cinf_from_Lie}, and one from the action of $\lie g$ on $\rcoT(\shift H^*)$.
  This implies that the map \eqref{eq:Lie_Cinf_model_restr} lands in $\shift \CoCE[\leq 1](\lie g, H^*) \iso \shift H^* \oplus \Hom_\QQ(\lie g, \shift H^*)$.
  Chasing through the construction, one sees that the two components of this map are as claimed in the statement.
\end{proof}

Note that the action of $\lie g$ on $H^*$ induces a representation of $\lie g$ on the Hochschild homology $\HH_\QQ(H^*)$ of the $\Ainf$-algebra $H^*$ of \cref{obs:Cinf_from_Lie}.
There then is a canonical isomorphism
\[ \HH_{\CoCE(\lie g)} \bigl( \CoCE(\lie g; H^*) \bigr)  \iso  \CoCE \bigl( \lie g; \HH_\QQ(H^*) \bigr) \]
of $\CoCE(\lie g)$-modules.
We now identify the map of \cref{cor:main} in these terms.

\begin{proposition} \label{lemma:Lie_tr}
  In the situation of \cref{lemma:Lie_Cinf_model}, assume that $\Coho * (X; \QQ)$ is finite-dimensional.
  Then the following restriction of the map $\tr$ of \cref{cor:main}
  \[ \HH_\QQ(H^*)  \longto  \CoCE \bigl( \lie g; \HH_\QQ(H^*) \bigr)  \iso  \HH_{\CoCE(\lie g)} \bigl( \CoCE(\lie g; H^*) \bigr)  \xlongto{\tr}  \CoCE(\lie g) \]
  is given by the sum $\tr_0 + \tr_1$, where $\tr_0 \colon \HH_\QQ(H^*) \to \QQ \subseteq \CoCE(\lie g)$ is the following composite of the projection and multiplication by the Euler characteristic of $X$
  \[ \tr_0  \colon  \HH_\QQ(H^*)  \xlongto{\pr}  H^0  \iso  \QQ  \xlongto{\chi(X)}  \QQ \]
  and $\tr_1 \colon \HH_\QQ(H^*) \to \dual{\lie g} \subseteq \CoCE(\lie g)$ is the map whose adjoint is, for $x_0, \dots, x_n \in \shift H^*$ and $\theta \in \lie g$, given by
  \[ x_0 \tensor \dots \tensor x_n \tensor \theta  \longmapsto  \sum_{j = 0}^n \sum_i (-1)^{\epsilon + 1 + \deg \theta \deg {e_i}} \langle x_j \otimes \dots \otimes x_n \otimes x_0 \otimes \dots \otimes x_{j-1} \tensor e_i, \phi(\theta)(\dual{e_i}) \rangle \]
  where $\epsilon \defeq (\deg {x_0} + \dots + \deg {x_{i - 1}}) (\deg {x_i} + \dots + \deg {x_n})$, and $e_i$ and $\langle \blank, \blank \rangle$ are as in \cref{lemma:Lie_Cinf_model}.
\end{proposition}

\begin{proof}
  First note that, by the description of the structure maps of $\CoCE(\lie g; H^*)$, the map indeed lands in $\CoCE[\le 1](\lie g) \iso \QQ \oplus \dual{\lie g}$.
  We write $\tr_0$ and $\tr_1$ for the respective components.
  The map $\tr_0 \colon \HH_\QQ(H^*) \to \QQ$ is the map $\tr$ associated to the $A_\infty$-algebra $H^*$ over $\QQ$.
  For degree reasons, it factors through the projection $\HH_\QQ(H^*) \to H^* \to H^0$.
  On $1 \in H^0$ it is, by definition, given by the Euler characteristic of $H^*$.
  The description of $\tr_1$ is obtained by evaluating the formula for $\tr$ using the explicit formula for the structure maps of $\CoCE(\lie g; H^*)$.
\end{proof}

\begin{remark} \label{rem:ES}
  Note that the map $\tr_1 \colon \HH_\QQ(H^*) \to \dual{\lie g}$ factors through $\HC_\QQ(H^*)$.
  Dualizing this, we obtain a map $\lie g \to \dual{\HC_\QQ(H^*)}$.
  This map is the composite
  \[ \lie g  \longto  \Der(\LL_+)  \longto  \Hom_\QQ \bigl( \shift[-1] H_*, \rT(\shift[-1] H_*) \bigr)  \xlongto{\tau}  \T(\shift[-1] H_*)  \xlongto{\pi}  \dual{\HC_\QQ(H^*)} \]
  where, for a finite-dimensional graded vector space $M$, the map $\tau$ is given by the composite
  \[ \Hom_\QQ \bigl( M, \rT(M) \bigr)  \iso  \Hom_\QQ \bigl( M, \T(M) \tensor M \bigr)  \xlongto{\tr}  \T(M) \]
  and the map $\pi$ is given by the norm maps $M^{\tensor n} \to (M^{\tensor n})^{\Cyclic n}$ and the identification $\shift[-1] H_* \iso \dual{(\shift H^*)}$.
  
  Such a formula has appeared a number of times in the literature.
  It is (the differential graded version) of the trace of Enomoto--Satoh \cite[§7.1]{ES}\footnote{In our conventions, the trace is the sum of the maps they denote by $c_k$.} defined for a closed surface of genus $g$.
  The version for a compact surface of genus $0$ with boundary is closely related to the non-commutative divergence map of Alekseev--Torossian \cite[§3.3]{AT} (see \cite[Theorem~1.3]{AKKN17} and \cite[Proposition~9.11]{AKKN18}, where the Alekseev--Torossian divergence and the Enomoto--Satoh trace are related to the Goldman bracket on genus $0$ and genus $g$ surfaces, respectively).
  It also appears as a low order term of a differential in a certain graph complex, see \cite[Proposition~5]{SeveraWillwacher} and \cite[§7]{NW}.
\end{remark}

%% file: 2applications.tex
\section{Applications}

\subsection{The Becker--Gottlieb transfer}

Recall that the \emph{Becker--Gottlieb transfer} of a map of animas $f \colon X \to Y$ with compact fibers is a ``wrong-way'' map $f^! \colon \SS[Y] \to \SS[X]$ of the suspension spectra.
It was originally defined by Becker--Gottlieb \cite{BG76} for fibrations whose base is finite-dimensional, and extended by Clapp \cite{Cla} to the general case; also see Carmeli--Cnossen--Ramzi--Yanovski \cite[§6]{CCRY} for a treatment in higher categorical language.
By the following result of Lind--Malkiewich \cite{LM}, the Becker--Gottlieb transfer can be recovered from the THH transfer.
In this subsection, we combine this with our rational model of the THH transfer to obtain a rational model for the Becker--Gottlieb transfer.

\begin{proposition}[Lind--Malkiewich] \label{prop:LM}
  Let $f \colon X \to Y$ be a map of animas with compact fibers.
  Then the following diagram commutes
  \[
  \begin{tikzcd}
    \SS[Y] \rar{f^!} \dar[swap]{\alpha} & \SS[X] \\
    \THH(Y) \rar{f^*} & \THH(X) \uar[swap]{e}
  \end{tikzcd}
  \]
  where $\alpha$ is the assembly map, and $e$ is the map $\THH(X) \eq \SS[\L X] \to \SS[X]$ induced by evaluation at the basepoint of $\Sphere{1}$.
\end{proposition}

\begin{proof}
  This is \cite[Theorem~1.2]{LM} (see also \cite[Theorem~6.12]{CCRY} for a higher categorical proof), noting that the assembly map $\alpha \colon \SS[Y] \to \THH(Y) \eq \SS[\L Y]$ is, by \cite[Proof of \Ilemmafibassembly]{I}, equivalent to the map induced by the inclusion of the constant loops $Y \to \L Y$
\end{proof}

\begin{corollary} \label{cor:BG}
  Let $f \colon X \to Y$ be a map of simply connected animas that is modeled by a cofibration $\phi \colon R \to S$ of homologically connected cofibrant cdgas of finite homotopical type, and let $S_\infty \to S$ be a unital quasi-isomorphism of $\Ainf$-algebras over $R$.
  Assume that the fiber of $f$ is simply connected and compact, and that $S_\infty$ is cofibrant and dualizable as an $R$-module.
  Then the Becker--Gottlieb transfer $f^! \colon \SS[Y] \to \SS[X]$ is modeled by the map
  \begin{align*}
    S_\infty &\longto R \\
    s &\longmapsto - \tr_R \bigl( \mu^{S_\infty}_{2}(s, \blank) \bigr)
  \end{align*}
  of cochain complexes.
\end{corollary}

\begin{proof}
  This follows from \cref{prop:LM} combined with \cref{cor:main,lemma:eval_model_Ainf}.
\end{proof}


\begin{remark}
  Note that, different from everything we have done so far, we prove neither a parametrized nor an equivariant version of \cref{cor:BG}.
  The reason is that we are not aware of a parametrized version of \cref{prop:LM}; phrased differently, this would say that the homotopy $f^! \eq e \circ f^* \circ \alpha$ is natural in self-equivalences of $f$.
  However, it does seem very likely that either of the existing proofs generalizes to show this stronger version.
  This would immediately imply a parametrized version of \cref{cor:BG}, including for equivariant models of fiberwise nilpotent spaces over $\B G$.
\end{remark}

\subsection{The space of THH-simple structures}

By definition, a \emph{simple structure} on a compact anima $X$ is a lift of its A-theory Euler characteristic along the A-theory assembly map (see e.g.\ \cite[Lecture~13]{LurAKM}).
In this subsection, we provide a rational model for the anima of fiberwise \emph{THH-simple structures}, i.e.\ fiberwise lifts of the THH Euler characteristic along the THH assembly map.
See \cref{intro:sec:THH-simple} of the introduction for more background and motivation.

\begin{notation} \label{not:rTHH}
  Let $E \to B$ be a map of animas.
  We write $\rTHH_B(E)$ for the cofiber of the fiberwise assembly map $\SS_B[E] \to \THH_B(E)$ (cf.\ \cref{def:THH_assembly}).
  We furthermore write $\Simptr[B](E)$ for the pullback of the following cospan of animas over $B$
  \[ B  \xlongto{\ol{\chi}}  \Loopsinf[B] \rTHH_B(E)  \xlongfrom{0}  B \]
  where the two maps are adjoint to the maps $\Suspinf[B] B \eq \SS_B \to \rTHH_B(E)$ given by the fiberwise THH Euler characteristic and the trivial map, respectively.
\end{notation}

Note that the fiber $\Simptr(E_b)$ of $\Simptr[B](E)$ over a point $b \in B$ is the anima of null-homotopies of the reduced THH Euler characteristic $\ol{\chi} \colon \SS \to \rTHH(E_b)$; such a null-homotopy is equivalently a lift of the THH Euler characteristic $\chi \colon \SS \to \THH(E_b)$ along the assembly map $\SS[E_b] \to \THH(E_b)$.

\begin{notation} \label{not:rHH}
  Let $R$ be a $\Cinf$-algebra over a cdga $\k$.
  We write $\rHH_\k(R) \subseteq \HH_\k(R)$ for the subcomplex of elements of Hochschild degree $\ge 1$.
  If $R$ is unital, we furthermore write $\rnHH_\k(R) \subseteq \nHH_\k(R)$ for the subcomplex of elements of Hochschild degree $\ge 1$, i.e.\ the image of $\rHH_\k(R)$ in the normalized Hochschild complex (cf.\ \cref{def:nHH}).
\end{notation}

Note that $\rHH_\k(R)$ (and hence $\rnHH_\k(R)$) is indeed closed under the differential by \cref{lemma:HH_proj}.
By \cref{lemma:assembly_model_Cinf}, these $\k$-modules are models for $\rTHH_B(E)$.
In the following theorem, we use the normalized Hochschild complex $\rnHH_\k(R)$ to obtain a cdga that is concentrated in non-negative degrees.

\begin{theorem} \label{thm:trace-simple}
  Let $G$ be a group and $p \colon E \to B$ a map of fiberwise nilpotent animas over $\B G$ that is modeled by a cofibration $\k \to R$ of $G$-equivariant homologically connected cofibrant cdgas of finite homotopical type.
  Assume that the fiber of $p$ is $2$-connected and compact.
  Furthermore let $R_\infty \to R$ be a unital $G$-equivariant quasi-isomorphism of $\Cinf$-algebras over $\k$ such that $R_\infty$ is dualizable as a $\k$-module, the unit map $\eta \colon \k \to R_\infty$ is a cofibration of $\k$-modules, and its cokernel $\quot {R_\infty} 1$ is concentrated in degrees $> 1$.
  Then the map $\Simptr[B](E) \to B$ of animas over $\B G$ is modeled by the $G$-equivariant $\k$-algebra
  \[ \left( \SA_\k \rnHH_\k(R_\infty)[-1], d \defeq d_1 + d_2 \right) \]
  where $d_1$ is induced by the differentials of $\k$ and $\rnHH_\k(R_\infty)$, and $d_2$ is the unique derivation that on wedge degree $1$ is induced by the restriction $\rHH_\k(R_\infty) \to \k$ of the map of \cref{cor:main} (in the target, we identify $\k$ with the part of wedge degree $0$).
\end{theorem}

\begin{proof}
  By \cref{lemma:assembly_model_Cinf} the maps of $B$-spectra
  \[ \SS_B[E]  \longto  \THH_B(E)  \longto  \rTHH_B(E) \]
  are modeled by the upper horizontal maps in the commutative diagram
  \[
  \begin{tikzcd}
    \rHH_\k(R_\infty) \rar \dar[swap]{\eq} & \HH_\k(R_\infty) \rar \dar{\eq} & R_\infty \dar[equal] \\
    \rnHH_\k(R_\infty) \rar & \nHH_\k(R_\infty) \rar & R_\infty
  \end{tikzcd}
  \]
  of $G$-equivariant $\k$-modules (note that the projection $\HH_\k(R_\infty) \to R_\infty$ factors through the quotient map $\HH_\k(R_\infty) \to \nHH_\k(R_\infty)$).
  By \cref{lemma:nHH_qiso} the middle vertical map is a quasi-isomorphism, and since both rows are short exact sequences, the left-hand vertical map is a quasi-isomorphism as well.
  By \cref{cor:main}, the fiberwise THH Euler characteristic $\chi \colon \SS_B \to \THH_B(E)$ is modeled by the map $\tr \colon \HH_\k(R_\infty) \to \k$, which factors over $\nHH_\k(R_\infty)$ by inspection.
  By \cite[\Ilemmaloopssuspadjmodel]{I}, the following left-hand maps of animas over $B$
  \[ B  \xlongto{\chi}  \Loopsinf[B] \rTHH_B(E)  \xlongfrom{0}  B  \qquad  \k  \longfrom  \SA_\k \rnHH_\k(R_\infty)  \longto  \k\]
  are thus modeled by the right-hand maps of cdgas over $\k$ induced by $\tr \colon \rnHH_\k(R_\infty) \to \k$ and the trivial map, respectively.
  Here we use that $\rnHH_\k(R_\infty)$ is cofibrant by (the proof of) \cref{lemma:nHH_cofibrant}, and concentrated in positive degrees since $\quot {R_\infty} 1$ is concentrated in degrees $> 1$.
  
  Note that the fiber of $0 \colon B \to \Loopsinf[B] \rTHH_B(E)$ is equivalent to $\Loops^{\infty + 1} \rTHH(F)$, where $F$ is the fiber of $p$; since $F$ is $2$-connected by assumption, the spectrum $\rTHH(F) \eq \cofib (\SS[F] \to \SS[\L F])$ is $1$-connected, and hence $\Loops^{\infty + 1} \rTHH(F)$ is connected.
  By \cite[\Iobspullbackeq]{I}, a model for $\Simptr[B](E)$ is thus given by the homotopy pushout of $\k \from \SA_\k \rnHH_\k(R_\infty) \to \k$ in $\CDGAeq{G}$, with the two maps respectively given by $\tr$ and $0$ on the generators.
  To compute this homotopy pushout, we factor the map $\rnHH_\k(R_\infty) \to 0$ as
  \[ \rnHH_\k(R_\infty)  \xlongto{\inc}  \hcofib \bigl( \id \colon \rnHH_\k(R_\infty) \to \rnHH_\k(R_\infty) \bigr)  \xlongto{\eq}  0 \]
  where $\hcofib(f)$ denotes the mapping cone of $f$.
  Since $\rnHH_\k(R_\infty)$ is a cofibrant $\k$-module, the map $\inc$ is a cofibration.
  Hence $\SA_\k \inc$ is a cofibration of cdgas, so the corresponding homotopy pushout is given by the strict pushout (this follows for example from \cite[Lemma~14.2]{FHT}).
  This yields the desired description.
\end{proof}

\begin{remark} \label{rem:Simptr_Lie}
  The map $\tr_1 \colon \HH_\QQ(H^*) \to \dual{\lie g}$ of \cref{lemma:Lie_tr} factors through $\nHH_\QQ(H^*)$, and thus its restriction dualizes to a map $\dual{\tr_1} \colon \lie g \to \dual{\rnHH_\QQ(H^*)}$.
  Similarly, the action of $\lie g$ on $\HH_\QQ(H^*)$ induces a representation on $\dual{\rnHH_\QQ(H^*)}$.
  When the Euler characteristic of $X$ is zero, together this defines an outer action of $\lie g$ on $\dual{\rnHH_\QQ(H^*)}$, considered as an abelian dg Lie algebra.
  In this situation, we can thus form the twisted semi-direct product
  \begin{equation} \label{eq:Simptr_Lie}
    \lie g \ltimes_{\dual{\tr_1}} \dual{\rnHH_\QQ(H^*)}
  \end{equation}
  (see e.g.\ \cite[§3.5]{Ber20}).
  This is a dg Lie model for $\Simptr[B](E)$: by construction and \cref{lemma:Lie_tr}, its Chevalley--Eilenberg cochains are isomorphic to the model for $\Simptr[B](E)$ of \cref{thm:trace-simple} when applied to the $\Cinf$-model of \cref{lemma:Lie_Cinf_model}.
  
  The dg Lie algebra \eqref{eq:Simptr_Lie} is closely related to the dg Lie algebra $\mathrm{GC}^{\mathrm{e}x}_{\bar{H},n}$ defined by Willwacher \cite[§13.5]{W23} as a model for the monoid of self-equivalences of the rationalized right $\mathrm{E}_d$-module $\mathcal{F}_M$ of configurations in a $d$-dimensional parallelized manifold $M$.
  More precisely $\mathrm{GC}^{\mathrm{e}x}_{\bar{H},n}$ is defined as a certain graph complex, and its part of loop order $\leq 1$ is essentially the same as \eqref{eq:Simptr_Lie} with Hochschild homology replaced by dihedral homology.
\end{remark}

\subsection{Manifold bundles}

We recall a result of Dwyer--Weiss--Williams \cite{DWW} that implies that the fiberwise THH Euler characteristic of a fiber bundle with fiber a compact topological manifold $M$ factors through the fiberwise THH assembly map.
(The manifold is allowed to have boundary, and the boundary fiber bundle does \emph{not} need to be trivial.)
As we explain below, this has consequences for certain characteristic classes of $M$-fibrations when evaluated on fiber bundles.

\begin{proposition}[Dwyer--Weiss--Williams] \label{prop:DWW}
  Let $M$ be a compact topological manifold, $\xi \colon B \to \B {\Homeo(M)}$ a map of animas, and $E \to B$ the $M$-bundle classified by $\xi$.
  Then the following composite map of $B$-spectra is null-homotopic
  \[ \SS_B \xlongto{\chi} \THH_B(E) \longto \rTHH_B(E) \]
  where $\chi$ is the fiberwise THH Euler characteristic, and $\rTHH_B(E)$ denotes the cofiber of the fiberwise assembly map (cf.\ \cref{not:rTHH}).
\end{proposition}

\begin{proof}
  By \cite[Theorem~8.4]{DWW}, there exists a dashed lift in the following diagram
  \[
  \begin{tikzcd}[column sep = 40]
    & \A(*) \tensor \SS_B[E] \dar{\alpha} \\
    \SS_B \rar{\chi} \urar[dashed, bend left = 15] & \A^B(E)
  \end{tikzcd}
  \]
  where $\chi$ is the fiberwise $\A$-theory Euler characteristic (i.e.\ the fiberwise $\A$-theory transfer of the map $E \to B$ of animas over $B$, precomposed with the unit of $\A(*)$) and $\alpha$ is the fiberwise $\A$-theory assembly map.
  (A higher categorical proof of this statement will appear in forthcoming work of Ramzi--Volpe--Wolf \cite{RVW}.)
  Fiberwise applying the Dennis trace map (see e.g.\ \cite[Lemma~10.5]{BGT}) yields a lift in the analogous diagram for $\THH$; taking the cofiber of the assembly map implies the claim.
\end{proof}

We will now explain how this result can be used to deduce relations between certain characteristic classes of fibrations when evaluated on fiber bundles.
Let $X$ be a compact anima, and $X/{\aut(X)} \to \B \aut(X)$ the universal fibration with fiber $X$.
Its (reduced) fiberwise THH Euler characteristic, pushed forward to the point, yields a map
\[ \xi  \colon  \SS[\B \aut(X)]  \longto  \bigl( \B \aut(X) \bigr)_! \rTHH_{\B \aut(X)} \bigl( X/{\aut(X)} \bigr) \]
and thus a cohomology class $\alpha$ of the target yields a cohomology class $\xi^*(\alpha)$ of $\B \aut(X)$, i.e.\ a characteristic class for fibrations with fiber $X$.
Now assume that $X = M$ is a compact topological manifold and let $B \to \B \Homeo(M)$ be a map of animas classifying an $M$-bundle $E \to B$.
We obtain a commutative diagram
\[
\begin{tikzcd}
  \SS[B] \rar{\xi} \dar & B_! \rTHH_B (E) \dar \\
  \SS[\B \aut(M)] \rar{\xi} & \bigl( \B \aut(M) \bigr)_! \rTHH_{\B \aut(M)} \bigl( M/{\aut(M)} \bigr)
\end{tikzcd}
\]
and note that the upper horizontal map is null-homotopic by \cref{prop:DWW}.
Hence, all characteristic classes of the form $\xi^*(\alpha)$ as above vanish on topological fiber bundles with compact manifold fibers.

Our \cref{cor:main} provides an explicit rational description of the fiberwise THH Euler characteristic and hence an explicit description of the image of $\xi^*$ on rational cohomology (also see \cref{sec:graph_classes} below).
Combining it with the result of Dwyer--Weiss--Williams, we obtain the following.
Note that \cref{intro:conj:HC} would imply an analog of this theorem for $\rHC$.

\begin{theorem} \label{thm:trace_vanishes}
  Let $G$ be a group, and $\k \to R$ a cofibration of $G$-equivariant homologically connected cofibrant cdgas of finite homotopical type that models a map $p \colon E \to B$ of fiberwise nilpotent animas over $\B G$.
  Furthermore, let $R_\infty \to R$ be a unital $G$-equivariant quasi-isomorphism of $\Cinf$-algebras over $\k$ such that $R_\infty$ is cofibrant and dualizable as a $\k$-module.
  If $p$ is classified by a map $B \to \B \Homeo(M)$ for a compact simply connected topological manifold $M$, then the following composite is trivial in the derived category $\Underlying(\Modeq{G}{\k})$
  \[ \rHH_\k(R_\infty)  \longto  \HH_\k(R_\infty)  \xlongto{\tr}  \k \]
  where $\tr$ is the map of \cref{cor:main}.
\end{theorem}

\begin{proof}
  By \cref{lemma:assembly_model_Cinf}, the map $\THH_B(E) \to \rTHH_B(E)$ is modeled by the inclusion $\rHH_\k(R_\infty) \to \HH_\k(R_\infty)$.
  Then the claim follows from \cref{cor:main,prop:DWW}.
\end{proof}

\subsection{Graph characteristic classes of loop order one} \label{sec:graph_classes}

In forthcoming work, Berglund \cite{Ber} constructs, for a Frobenius $\Cinf$-algebra over a cdga $\k$ (i.e.\ a $\Cinf$-algebra over $\k$ compatibly equipped with a self-duality), a map from a certain graph complex to $\k$ (also see Kontsevich \cite{Kon} and Matsuyuki \cite{Mat} for precursors).
In this subsection, we adapt his work to construct, for a $\Cinf$-algebra $R$ over $\k$ that is only assumed to be dualizable, a map from a certain ``complex of graphs of loop order one'' to $\k$.
We furthermore use \cref{thm:trace_vanishes} to prove that this map vanishes on all ``interesting'' homology classes when $\k \to R$ is a $\Cinf$-model for a fiber bundle $E \to B$ whose fibers are compact manifolds.
To this end, we show that (a variant of) our construction agrees up to quasi-isomorphism with the map $\tr \colon \HC_\k(R) \to \k$ induced by the map of \cref{cor:main}.

\subsubsection*{Wheeled operads}

To formulate these results, we will use (a non-unital version of) the notion of a \emph{wheeled operad} of Markl--Merkulov--Shadrin \cite[§5.1]{MMS}.
Just as operads are modeled on directed trees where every vertex has exactly one outgoing edge, wheeled operads are modeled on ``wheeled'' versions of these as follows.

\begin{definition}
  A \emph{wheeled tree of biarity $(m, n)$} is a connected finite directed graph $T$ with $m$ distinguished \emph{source} vertices labeled from $1$ to $m$, and $n$ distinguished \emph{sink} vertices labeled from $1$ to $n$, such that every vertex has at most one outgoing edge, a source has exactly one outgoing edge and no incoming edges, and a sink has exactly one incoming edge and no outgoing edge.
  An \emph{internal} vertex is a vertex that is neither a sink nor a source; we assume that a wheeled tree has at least one internal vertex.
  We write $V_T$ for the set of internal vertices and $E_T$ for the set of \emph{internal} edges, i.e.\ those between internal vertices.
  For an internal vertex $v$ we write $\inedges(v)$ for the number of incoming edges and $\outedges(v)$ for the number of outgoing edges.
\end{definition}

Note that a wheeled tree of biarity $(m, 0)$ either has exactly one cycle and no vertices without outgoing edges or no cycles and exactly one vertex without outgoing edges, that a wheeled tree of biarity $(m, 1)$ has no cycle and no vertices without outgoing edges, and that there are no wheeled trees of biarity $(m, n)$ with $n > 1$.
By definition, a \emph{non-unital wheeled operad} $\operad P$ in a symmetric monoidal category $\cat C$ consists of two symmetric sequences (as in the following definition) in $\cat C$, equipped with structure maps allowing to compose them along wheeled trees.

\begin{definition}
  Let $\cat C$ be a category.
  A \emph{symmetric sequence} $S$ in $\cat C$ is a sequence $S(m)$ of $\Sigma_m$-modules in $\cat C$ for $m \in \NN$.
  A \emph{symmetric bisequence} $T$ in $\cat C$ is a family $T(m, n)$ of $\Sigma_m$-modules in $\cat C$ for $m \in \NN$ and $n \in \set{0, 1}$.
  If $\cat C$ has an initial object $\emptyset$ and $S$ is a symmetric sequence in $\cat C$, we write $S^\Y$ for the symmetric bisequence given by $S^\Y(m, 1) \defeq S(m)$ and $S^{\mathsf{Y}}(m, 0) \defeq \emptyset$.
\end{definition}

\begin{notation} \label{not:wheeled_trace}
  For a non-unital wheeled operad $\operad P$ in a symmetric monoidal category $\cat C$ and a natural number $m \ge 1$, we write
  \[ \tr \colon \operad P(m, 1) \longto \operad P(m - 1, 0) \]
  for composition along the wheeled tree with one internal vertex with $m$ incoming edges whose outgoing edge is connected to its $m$-th incoming edge.
\end{notation}

Note that a non-unital wheeled operad has an underlying non-unital operad by restricting to $\operad P(m, 1)$ and compositions along trees.
The forgetful functors to both non-unital operads and symmetric bisequences admit left adjoints, so we have notions of free non-unital wheeled operads as follows.

\begin{notation}
  Let $\cat C$ be a cocomplete symmetric monoidal category such that the tensor product preserves colimits in both variables.
  For a symmetric sequence $S$ in $\cat C$, we write $\Fop(S)$ for the non-unital operad freely generated by $S$.
  For a non-unital operad $\operad P$ in $\cat C$, we write $\wheeled{\operad P}$ for its \emph{wheeled envelope}, i.e.\ the non-unital wheeled operad freely generated by $\operad P$.
  For a symmetric bisequence $T$ in $\cat C$, we write $\wFop(T)$ for the non-unital wheeled operad freely generated by $T$.
  Note that $\wFop(T)(m, n)$ consists of wheeled trees of biarity $(m, n)$, with internal vertices labeled by $T$; we write $\cwFop(T)(m, n)$ for the similarly defined $\Sigma_m$-module where we only use \emph{cycle trees}, i.e.\ wheeled trees with a cycle where every internal vertex is part of the cycle.
\end{notation}

Note that there is a canonical isomorphism $\wheeled{\Fop(S)} \iso \wFop(S^\Y)$.
There is also a canonical map $\cwFop(T) \to \wFop(T)$ of symmetric bisequences by including cycle trees into all wheeled trees.

\begin{definition}
  Let $\cat V$ be a symmetric monoidal category and $D$ a dualizable object (cf.\ \cref{def:dualizable}) of a symmetric monoidal $\cat V$-enriched category $\cat C$.
  We define the \emph{wheeled endomorphism operad} $\wEndop(D)$ in $\cat V$ via
  \[ \wEndop(D)(m, n) \defeq \Hom(D^{\tensor m}, D^{\tensor n}) \in \cat V \]
  with the composite along a labeled wheeled tree $T$ of biarity $(m, n)$ defined to be the map $D^{\tensor m} \to D^{\tensor n}$ given by
  \[ \textstyle { \bigl( \bigotimes_{E_T} \ev \otimes \bigotimes_n \id_D \bigr) \circ \omega \circ \bigl( \bigotimes_{v \in V_T} \alpha_v \otimes \bigotimes_{E_T} \id_{\dual D} \bigr) \circ \sigma \circ \bigl( \bigotimes_{E_T} \coev \otimes \bigotimes_m \id_D \bigr) } \]
  where $\alpha_v \in \Hom(D^{\tensor {\inedges(v)}}, D^{\tensor {\outedges(v)}})$ is the label of an internal vertex $v$, the maps $\ev \colon \dual D \tensor D \to \unit$ and $\coev \colon \unit \to D \tensor \dual D$ are the evaluation and coevaluation, and $\omega$ and $\sigma$ are the obvious permutations.
  For a non-unital wheeled operad $\operad P$ in $\cat V$, a \emph{$\operad P$-algebra in $\cat C$} is a dualizable object $A$ of $\cat C$ equipped with a map $\operad P \to \wEndop(A)$ of non-unital wheeled operads in $\cat V$.
\end{definition}

If $\cat C$ is additionally tensored over $\cat V$ and cocomplete, then a $\operad P$-algebra $A$ in $\cat C$ is equivalently described by structure maps $\operad P(m, n) \tensor_{\Sigma_m} A^{\tensor m} \to A^{\tensor n}$.
Also note that the underlying operad of $\wEndop(D)$ is the usual endomorphism operad of $D$ (by one of the triangle identities relating $\ev$ and $\coev$).
We furthermore observe that, on the level of a $\operad P$-algebra $A$, the map $\tr$ of \cref{not:wheeled_trace} corresponds to partially evaluating an operation $\alpha \in \operad P(m, 1)$ on elements $a_1, \dots, a_{m-1} \in A$ and taking the trace of the resulting endomorphism $\alpha(a_1, \dots, a_{m-1}, \blank)$.

We now observe that the structure maps of a $\operad P$-algebra can be bundled into a single map of symmetric sequences, as follows.
For convenience, we restrict to the case where $\cat C = \Mod{\k}$ for some cdga $\k$.

\begin{notation}
  For a symmetric bisequence $\operad P$ in cochain complexes and a module $V$ over a cdga $\k$, we write
  \[ \Schur[\k] {\operad P} {V, n}  \defeq  \bigoplus_{m \ge 0} \operad P(m, n) \tensor_{\Sigma_m} V^{\tensor_\k m} \]
  for $n \in \set{0,1}$.
  If $\operad P$ is equipped with the structure of a non-unital wheeled operad and $V$ with the structure of an algebra in $\k$-modules over the underlying non-unital operad of $\operad P$, we write
  \[ \qSchur[\k] {\operad P} {V, n}  \defeq  \coequalizer \bigl( \begin{tikzcd}[cramped] \Schur[\k] {\operad P} {\Schur[\k] {\operad P} {V, 1}, n} \rar[yshift = 3] \rar[yshift = -3] & \Schur[\k] {\operad P} {V, n} \end{tikzcd} \bigr) \]
  where the two maps respectively use the structure maps of $V$ and the structure maps of $\operad P$.
  When $\k = \QQ$, we omit the subscript.
\end{notation}

Note that, for a non-unital wheeled operad $\operad P$ in cochain complexes, a $\operad P$-algebra $A$ in $\Mod{\k}$ is equivalently described by ``evaluation maps'' $\Schur[\k] {\operad P} {A, n} \to A^{\tensor n}$ (for $n = 0, 1$).
These furthermore induce maps $\qSchur[\k] {\operad P} {A, n} \to A^{\tensor_\k n}$ by construction.
The following gives a more explicit description of $\qSchur[\k] {\wheeled{\operad P}} {A, 0}$ for a free operad $\operad P$.

\begin{lemma} \label{lemma:qSchur}
  Let $S$ be a symmetric sequence in cochain complexes, $\k$ a cdga, and $A$ an algebra in $\k$-modules over the non-unital operad $\Fop(S)$.
  Then there is an isomorphism
  \[ \Schur[\k] {\cwFop(S^\Y)} {A, 0}  \xlongto{\iso}  \qSchur[\k] {\wFop(S^\Y)} {A, 0} \]
  induced by the canonical map $\cwFop(S^\Y) \to \wFop(S^\Y)$.
\end{lemma}

\begin{proof}
  Consider the zig-zag
  \[ \Schur[\k] {\wFop(S^\Y)} {A, 0}  \xlongfrom{\iso}  \Schur[\k] {\cwFop(S^\Y)} {A \oplus \Schur {\Fop(S)} {A, 1}, 0}  \longto  \Schur[\k] {\cwFop(S^\Y)} {A, 0} \]
  where the left-hand map is induced by the structure maps of $\wFop(S^\Y)$ and the right-hand map by the structure maps of $A$.
  The left-hand map is an isomorphism, and we obtain an induced map $\qSchur[\k] {\wFop(S^\Y)} {A, 0} \to \Schur[\k] {\cwFop(S^\Y)} {A, 0}$, which is inverse to the map from the statement by construction.
\end{proof}

\subsubsection*{The loop-order-one graph complex}

We first recall the operads representing $\Ainf$- and $\Cinf$-algebras, using our grading conventions.

\begin{definition} \label{def:Cinf_operad}
  We define $\sAinf$ to be the non-unital operad in cochain complexes whose underlying non-unital operad in graded vector spaces is freely generated by the symmetric sequence $A(n) \defeq \QQ \langle \mu_n \rangle \tensor \QQ[\Sigma_n]$ for $n \ge 2$ with $\mu_n$ of degree $1$ and whose differential is given by
  \[ d(\mu_n)  \defeq  - \sum_{\substack{r+s+t = n \\ 1 < s < n}} \mu_{r+1+t} \circ_{r+1} \mu_s \]
  where $\circ_i$ denotes composition along the $i$-th input.
  We define $\sCinf$ analogously, except that it is generated by the symmetric sequence $C(n) \defeq \QQ \langle \mu_n \rangle \tensor \quot{\QQ[\Sigma_n]}{S_n}$ for $n \ge 2$, where $S_n \subseteq \QQ[\Sigma_n]$ denotes the $\Sigma_n$-subrepresentation generated by the elements $\sum_{\sigma \in \Sh p q} \sigma$ for all $p, q \ge 1$ with $p + q = n$ (cf.\ \cref{def:shuffle}).
  We write $\Ainf$ and $\Cinf$ for the non-unital operads obtained as the operadic desuspensions (see e.g.\ \cite[§7.2.2]{LV}) of $\sAinf$ and $\sCinf$, respectively.
\end{definition}

By definition, an $\Ainf$-algebra $R$ over $\k$ is the same thing as the structure of an algebra over the non-unital operad $\Ainf$ on the $\k$-module $R$, which is equivalent to an $\sAinf$-algebra structure on $\shift R$ (and analogously for $\sCinf$).
Note however that maps of algebras over the non-unital operad $\Ainf$ correspond only to \emph{strict} maps of $\Ainf$-algebras.
Note that the degree of $\mu_n$ in $\Ainf$ and $\Cinf$ is $2 - n$, and recall that the maps $\Ainf \to \Ass$ and $\Cinf \to \Com$, given by sending $\mu_2$ to the multiplication and $\mu_n$ to $0$ for $n > 2$, are quasi-isomorphisms (see e.g.\ \cite[Corollary~9.2.5 and §13.1.8]{LV}).

In particular we can now form the wheeled envelopes $\wheeled{\sAinf}$, $\wheeled{\sCinf}$, $\wheeled{\Ainf}$, and $\wheeled{\Cinf}$.
They are spanned by wheeled trees whose vertices are labeled by the $\mu_n$, and their differentials are sums of edge expansions: replacing a vertex with two vertices connected by an edge.
Hence they can be considered to be ``complexes of hairy directed graphs of loop order at most one''.

\begin{remark}
  The (unitalizations of the) wheeled envelopes $\wheeled{\Ainf}$ and $\wheeled{\Cinf}$ are \emph{not} equivalent to the wheeled operads $\mathrm{Ass}^\circlearrowright_\infty$ and $\mathrm{Com}^\circlearrowright_\infty$ of \cite[Theorems~A and B]{MMS}; the latter are resolutions of the wheeled envelopes $\wheeled{\Ass}$ and $\wheeled{\Com}$.
  
  Also note that the map $\tr$ of \cref{cor:main} looks like a commutative analog of the ``cyclic characteristic class'' of \cite[§6.7]{MMS}.
  In particular it seems plausible that an analog of \cite[Theorem~6.7.3]{MMS} proves that, for a finite-dimensional $\Cinf$-algebra $R$ over $\QQ$, the cochain map $\tr \colon \HH_\QQ(R) \to \QQ$ is null-homotopic if and only if $R$ can be extended to an algebra over $\mathrm{Com}^\circlearrowleft_\infty$.
\end{remark}

\begin{remark} \label{rem:Berglund_map}
  Given a $\Cinf$-algebra $R$ over $\k$ such that $R$ is dualizable as a $\k$-module, we obtain an evaluation map $\Schur[\k] {\wheeled{\Cinf}} {R, 0} \to \k$.
  This is an analog for dualizable $\Cinf$-algebras of a map constructed by Berglund \cite{Ber} for Frobenius $\Cinf$-algebras (which are in particular dualizable).
  Instead of the wheeled envelope $\wheeled{\Cinf}$, he uses the modular envelope $\mathbb{M}_{\mathrm{det}^d}(\Cinf)$ of Getzler--Kapranov \cite{GK} for a certain ``hyperoperad'' $\det^d$ (here $d$ is the degree of the Frobenius $\Cinf$-algebra) and obtains a map $\Schur[\k] {\mathbb{M}_{\mathrm{det}^d}(\Cinf)} {R} \to \k$.
  One can think of $\mathbb{M}_{\mathrm{det}^d}(\Cinf)$ as a complex of hairy graphs whose vertices are labeled by the operations $\mu_n$; in fact it is isomorphic (up to shifts) to the even or the odd Lie graph complex, depending on the parity of $d$.
  There is a map $\wheeled{\Cinf}(\blank, 0) \to \mathbb{M}_{\mathrm{det}^d}(\Cinf)$ of symmetric sequences for any $d$, given by forgetting the direction of an edge (or rather, remembering it only as an orientation), and the resulting diagram
  \[
  \begin{tikzcd}
    \Schur[\k] {\wheeled{\Cinf}} {R, 0} \dar \drar[bend left] & \\
    \Schur[\k] {\mathbb{M}_{\mathrm{det}^d}(\Cinf)} {R} \rar & \k
  \end{tikzcd}
  \]
  commutes for a Frobenius $\Cinf$-algebra $R$ over $\k$.
  The map $\wheeled{\Cinf}(\blank, 0) \to \mathbb{M}_{\mathrm{det}^d}(\Cinf)$ is surjective on homology onto the part consisting of graphs of loop order one (see \cref{rem:comparison_gc} below); in particular this shows that the loop order one part of Berglund's construction can be defined more generally for dualizable $\Cinf$-algebras.
\end{remark}

We will now prove that $\qSchur[\k] {\wheeled{\Cinf}} {R, 0} \iso \qSchur[\k] {\wheeled{\sCinf}} {\shift R, 0}$ is quasi-isomorphic to the Connes complex of the $\Cinf$-algebra $R$.
This uses similar arguments to work of Conant--Kassabov--Vogtmann \cite[§6]{CKV}, who related the loop-order one part of the (undirected) hairy Lie graph complex to dihedral homology of free symmetric algebras.

\begin{lemma} \label{lemma:graph_complex_iso}
  Let $\k$ be a cdga and $R$ a $\Cinf$-algebra over $\k$.
  Then the following map is an isomorphism of $\k$-modules
  \begin{align*}
    \HC_\k(\rB R)  &\xlongto{\iso}  \qSchur[\k] {\wheeled{\sCinf}} {\shift R, 0} \\
    \vec x_0 \tensor \dots \tensor \vec x_n  &\longmapsto  \tr(\mu_{k_0 + 1} \circ_{k_0 + 1} \dots \circ_{k_{n-1} + 1} \mu_{k_n + 1}) \tensor (\vec x_0 \tensor \dots \tensor \vec x_n)
  \end{align*}
  where $\vec x_i = x_{i,1} \tensor \dots \tensor x_{i,k_i} \in \rB R$ for $0 \le i \le n$.
  Here $\rB R$ denotes the non-counital bar construction of $R$ (see \cref{def:Ainf}) and $\HC_\k(\rB R)$ the cyclic homology of this $\k$-coalgebra (defined exactly dually to cyclic homology of a $\k$-algebra, see \cref{def:HH}).
\end{lemma}

\begin{proof}
  Unwinding the definitions, we see that the map in the statement is compatible with the differentials; in particular it is a map of $\k$-modules.
  Now let $\cwsCinf \defeq \cwFop(C)$, where $C$ is the symmetric sequence of graded vector spaces of \cref{def:Cinf_operad}.
  By \cref{lemma:qSchur}, the inclusion induces an isomorphism $\Schur[\k] \cwsCinf {R, 0} \iso \qSchur[\k] {\wheeled{\sCinf}} {R, 0}$ of graded $\k$-modules.
  Hence it is enough to prove that the map $\HC_\k(\rB R) \to \qSchur[\k] \cwsCinf {R, 0}$, defined as in the statement, is an isomorphism of graded $\k$-modules.
  
  We first note that, using the shuffle relations, an element of $\quot{\QQ[\Sigma_n]}{S_n}$ represented by $\sigma \in \Sigma_n$ such that $\sigma(n) < n$ is equal to a linear combination of $\omega \in \Sigma_n$ such that $\omega(n) > \sigma(n)$.
  By iterating this procedure, we see that $\sigma$ is equivalent to a linear combination of $\omega \in \Sigma_n$ such that $\omega(n) = n$.
  It is furthermore well-known that $\quot{\QQ[\Sigma_n]}{S_n} \iso \dual{\Lie(n)}$ and that this vector space is $(n-1)!$-dimensional (see e.g.\ \cite[Theorem~1.3.6 and §13.2.3]{LV}).
  Hence the inclusion $\Sigma_{n-1} \to \Sigma_n$ induces an isomorphism $\QQ[\Sigma_{n-1}] \iso \quot{\QQ[\Sigma_n]}{S_n}$.
  
  Thus a basis of $\cwsCinf(m, 0)$ is given by cycle trees up to rotation whose internal vertices are labeled by $\mu_n \tensor \sigma$ such that $\sigma(n) = n$ and such that the incoming edge belonging to the cycle is the $n$-th incoming edge.
  Hence the map $\HC_\k(\rB R) \to \qSchur[\k] \cwsCinf {R, 0}$ is an isomorphism.
\end{proof}

\begin{proposition} \label{prop:HC_qSchur}
  Let $\k$ be a cdga and $R$ a $\Cinf$-algebra over $\k$.
  Then the following map is a quasi-isomorphism of $\k$-modules
  \begin{align*}
    \HC_\k(R)  &\xlongto{\eq}  \qSchur[\k] {\wheeled{\sCinf}} {\shift R, 0} \\
    x_0 \tensor \dots \tensor x_n  &\longmapsto  \sum_{i = 0}^n (-1)^\epsilon \tr(\mu_{n+2}) \tensor x_i \tensor \dots \tensor x_n \tensor x_0 \tensor \dots \tensor x_{i - 1}
  \end{align*}
  where $\epsilon \defeq (\deg {x_0} + \dots + \deg {x_{i - 1}}) (\deg {x_i} + \dots + \deg {x_n})$.
  It is natural in strict maps of $\Cinf$-algebras over $\k$.
  Moreover, if $R$ is dualizable as a $\k$-module, then the following composite with the evaluation map
  \[ \HC_\k(R)  \xlongto{\eq}  \qSchur[\k] {\wheeled{\sCinf}} {\shift R, 0}  \longto  \k \]
  is $-1$ times the map induced by the map $\tr$ of \cref{cor:main}.
\end{proposition}

\begin{proof}
  This follows by combining \cref{lemma:graph_complex_iso,lemma:HC_of_bar}.
\end{proof}

Note that $\wheeled{\Cinf}(n, 1) \iso \Cinf(n) \eq \Com(n)$, so that $\Coho * (\wheeled{\Cinf})(n, 1) \iso \QQ$ (spanned by any wheeled tree of biarity $(n, 1)$ such that every internal vertex has exactly two incoming edges).
The homology of $\wheeled{\Cinf}(n, 0)$ is interesting however, and we will now compute it as a symmetric sequence using \cref{prop:HC_qSchur}.
Forgetting the $\Sigma_n$-actions, this recovers a result of Markl--Merkulov--Shadrin \cite[Theorem~7.1.1]{MMS}, who proved that the dimension of $\Coho k (\wheeled{\Cinf})(n, 0)$ is $\binom{n-1}{k}$.
In particular $\Coho 0 (\wheeled{\Cinf})(n, 0)$ is one-dimensional for $n \ge 1$, and is thus spanned by any wheeled tree of biarity $(n, 0)$ such that every internal vertex has exactly two incoming edges (this can again be seen using the map $\wheeled{\Cinf} \to \wheeled{\Com}$).
Our argument is similar to work of Conant--Kassabov--Vogtmann \cite[§6.3]{CKV}, who computed the loop-order one part of the (undirected) hairy Lie graph complex.

\begin{proposition} \label{prop:wheeled_Cinf_homology}
  For $n \ge 1$, there is an isomorphism of graded $\Sigma_n$-representations
  \[ \Coho* ( \wheeled{\Cinf} ) (n, 0)  \iso  \SA \shift (\quot {\QQ^n} \Delta) \]
  where $\Delta \defeq (1, \dots, 1) \in \QQ^n$.
  That is, the $k$-th homology of $\wheeled{\Cinf}(n, 0)$ is isomorphic to the $k$-th exterior power of the standard representation of $\Sigma_n$.
\end{proposition}

\begin{proof}
  In the following, we write $\Cinf(V) \iso V \oplus \Schur {\Cinf} {V}$ for the $\Cinf$-algebra freely generated by a vector space $V$.
  Then we have
  \[ \HC_\QQ \bigl( \Cinf(V) \bigr)  \eq  \qSchur {\wheeled{\sCinf}} {\shift \Cinf(V), 0}  \iso  \qSchur {\wheeled{\Cinf}} {\Cinf(V), 0}  \iso  \Schur {\wheeled{\Cinf}} {V, 0} \]
  by \cref{prop:HC_qSchur}.
  Moreover the map $\Cinf(V) \to \Com(V) \iso \SA^+ V$ is a quasi-isomorphism, where $\SA^+ V \defeq \bigoplus_{n \ge 1} (V^{\tensor n})_{\Sigma_n}$ denotes the free commutative non-unital $\QQ$-algebra on $V$.
  Hence we have $\Schur {\wheeled{\Cinf}} {V, 0} \eq \HC_\QQ ( \SA^+ V )$ by (the proof of) \cref{lemma:HH_qiso}.
  
  By \cite[2.2.13 and 2.2.16]{Lod}, the canonical map $\HC_\QQ(\SA^+ V) \oplus \HC_\QQ(\QQ) \to \HC_\QQ(\SA V)$ is a quasi-isomorphism.
  By \cite[Theorem~3.2.5]{Lod}, the homology of $\HC_\QQ(\SA^+ V)$ is thus isomorphic to $\quot {\Omega^*_S} {d \Omega^{*-1}_S}$ where $\Omega^*_S$ denotes the de Rham complex of $S \defeq \SA^+ V$, i.e.\ the non-unital cdga $\SA^+ (V \oplus \shift V)$ with differential determined by $d(v) \defeq \shift v$ for $v \in V$.
  This is isomorphic to $\bigoplus_{n \ge 1} D^{\tensor n} \tensor_{\Sigma_n} V^{\tensor n}$ where $D \defeq \QQ \oplus \shift \QQ$ with differential determined by $d(1) \defeq \shift 1$.
  The cochain complex with $\Sigma_n$-action $D^{\tensor n}$ is furthermore isomorphic to $\SA \shift \QQ^n$ with differential $d(x) \defeq \shift \Delta \wedge x$.
  Thus $\quot {D^{\tensor n}} {d(D^{\tensor n})}$ is $\Sigma_n$-equivariantly isomorphic to $\SA \shift (\quot {\QQ^n} \Delta)$.
  Hence there is a natural isomorphism
  \[ \Schur {\Coho* (\wheeled{\Cinf})} {V, 0}  \iso  \Coho* \bigl( \Schur {\wheeled{\Cinf}} {V, 0} \bigr)  \iso  \bigoplus_{n \ge 1} \SA \shift (\quot {\QQ^n} \Delta) \tensor_{\Sigma_n} V^{\tensor n} \]
  of graded vector spaces.
  Chasing through the argument, we see that it restricts to natural isomorphisms
  \[ \Coho* (\wheeled{\Cinf})(n, 0) \tensor_{\Sigma_n} V^{\tensor n}  \iso  \SA \shift (\quot {\QQ^n} \Delta) \tensor_{\Sigma_n} V^{\tensor n} \]
  for $n \ge 1$.
  This implies the claim since the functor $P \mapsto P \tensor_{\Sigma_n} (\blank)^{\tensor n}$ from finite-dimensional $\Sigma_n$-representations to endofunctors of the category of finite-dimensional vector spaces is fully faithful (see Macdonald \cite[Ch.~I, Appendix A, 5.2]{Mac95}).
\end{proof}

\begin{remark} \label{rem:comparison_gc}
  As mentioned in \cref{rem:Berglund_map}, forgetting the direction of an edge (or rather, remembering it only as an orientation) yields a map of symmetric sequences from $\wheeled{\Cinf}(\blank, 0)$ to both the loop-order-one part $\mathrm{GC}_1^+(\Lie)$ of the even hairy Lie graph complex and the corresponding odd version $\mathrm{GC}_1^-(\Lie)$ (i.e.\ the genus one parts of the modular envelopes of the cyclic operad $\Cinf$, with respect to the trivial and the ungraded determinant hyperoperad; see Getzler--Kapranov \cite[§4]{GK}).
  In fact, flipping the direction of the cycle is an involution on $\wheeled{\Cinf}(\blank, 0)$, and its $+1$ and $-1$ eigenspaces recover $\mathrm{GC}_1^+(\Lie)$ and $\mathrm{GC}_1^-(\Lie)$, respectively.
  In particular the homology of $\wheeled{\Cinf}(\blank, 0)$ splits as the sum of the homologies of $\mathrm{GC}_1^+(\Lie)$ and $\mathrm{GC}_1^-(\Lie)$.
  On degree $k$ the involution acts by $(-1)^k$, so that $\Coho * (\mathrm{GC}_1^+(\Lie))$ consists of the even-dimensional and $\Coho * (\mathrm{GC}_1^-(\Lie))$ of the odd-dimensional classes of $\Coho * (\wheeled{\Cinf})(\blank, 0)$.
  As mentioned above, a computation of $\Coho * (\mathrm{GC}_1^+(\Lie))$ along the same lines has been carried out in \cite[§6.3]{CKV}\footnote{Note that, in their notation, the involution on cyclic homology acts by $(-1)^n$ and not $(-1)^{k + 1}$ as they claim, which changes the result slightly.}; a different proof of this can be obtained by combining \cite[Theorem~11.1]{CKV12} and \cite[Proposition~3.7]{CHKV}.
\end{remark}

\subsubsection*{Vanishing of graph classes for fiber bundles}

We will now deduce from \cref{thm:trace_vanishes} that, for a topological fiber bundle modeled by a dualizable $\Cinf$-algebra $R_\infty$ over a cdga $\k$, the evaluation map $\Schur {\wheeled{\Cinf}} {R_\infty, 0} \to \k$ vanishes on all homology classes coming from positive degrees of $\wheeled{\Cinf}$.
By \cref{prop:wheeled_Cinf_homology}, this covers all ``interesting'' homology classes.
More precisely $\Coho 0 (\wheeled{\Cinf})(n, 0)$ has dimension one for each arity $n \ge 1$ (spanned by any wheeled tree of biarity $(n, 0)$ such that every internal vertex has exactly two incoming edges), but there are a lot of classes in positive degrees.
By \cref{rem:Berglund_map}, this result implies the analogous statement for the loop-order-one part of the analogous map of Berglund \cite{Ber}.
In fact, following his approach, we do everything $G$-equivariantly, i.e.\ we consider the following.

\begin{construction} \label{con:equv_eval}
  Let $G$ be a group and $R_\infty$ a $G$-equivariant $\Cinf$-algebra over a $G$-equivariant cdga $\k$.
  If $R_\infty$ is dualizable as a $\k$-module, then we obtain a composite map
  \[ \kappa_\wheel(R_\infty)  \colon  \Schur {\wheeled{\Cinf}} { \Cochains{*}(G; R_\infty), 0}  \longto  \Cochains{*} \bigl( G; \Schur {\wheeled{\Cinf}} {R_\infty, 0} \bigr)  \longto  \Cochains{*} \bigl( G; \Schur[\k] {\wheeled{\Cinf}} {R_\infty, 0} \bigr)  \longto  \Cochains{*}(G; \k) \]
  where $\Cochains{*}(G; M)$ denotes the cochain complex defining group cohomology of $G$ with coefficients in $M$ (note that $\Cochains{*}(G; \blank)$ is lax symmetric monoidal).
\end{construction}

\begin{theorem} \label{cor:graph_vanishing}
  In the situation of \cref{thm:trace_vanishes}, the map induced on homology by $\kappa_\wheel(R_\infty)$ vanishes on
  \[ \Schur {\Coho {> 0} (\wheeled{\Cinf})} {\Coho * (G; R_\infty), 0}  \subseteq  \Schur {\Coho * (\wheeled{\Cinf})} {\Coho * (G; R_\infty), 0}  \iso  \Coho * \bigl( \Schur {\wheeled{\Cinf}} {\Cochains * (G; R_\infty), 0} \bigr) \]
  where $\Coho {> 0}$ denotes homology classes of positive degree.
\end{theorem}

\begin{proof}
  In the following, we write $\Cinf(N)_\k \iso N \oplus \Schur[\k] {\Cinf} {N}$ for the $\Cinf$-algebra over $\k$ freely generated by a $\k$-module $N$; when $\k = \QQ$, we omit the subscript.
  Note that there is a canonical strict map $\Cinf(R_\infty)_\k \to R_\infty$ of $G$-equivariant $\Cinf$-algebras over $\k$.
  We obtain a commutative diagram of $G$-equivariant cochain complexes
  \[
  \begin{tikzcd}
    \rHH_\QQ \bigl( \Cinf(R_\infty) \bigr) \rar \dar & \HC_\QQ \bigl( \Cinf(R_\infty) \bigr) \rar{\eq} \dar & \qSchur {\wheeled{\sCinf}} {\shift \Cinf(R_\infty), 0} \dar & \lar[swap]{\iso} \Schur {\wheeled{\sCinf}} {\shift R_\infty, 0} \dar \\
    \rHH_\k \bigl( \Cinf(R_\infty)_\k \bigr) \rar \dar & \HC_\k \bigl( \Cinf(R_\infty)_\k \bigr) \rar{\eq} \dar & \qSchur[\k] {\wheeled{\sCinf}} {\shift \Cinf(R_\infty)_\k, 0} \dar & \lar[swap]{\iso} \Schur[\k] {\wheeled{\sCinf}} {\shift R_\infty, 0} \dar \\
    \rHH_\k ( R_\infty ) \rar & \HC_\k ( R_\infty ) \rar{\eq} & \qSchur[\k] {\wheeled{\sCinf}} {\shift R_\infty, 0} \rar & \k
  \end{tikzcd}
  \]
  where the middle quasi-isomorphisms are the maps of \cref{prop:HC_qSchur}.
  The composite of the bottom row is trivial in $\Underlying(\Modeq{G}{\k})$ by \cref{thm:trace_vanishes}.
  Applying $\Cochains * (G; \blank)$ and identifying $\Schur {\wheeled{\sCinf}} {\shift R_\infty, 0} \iso \Schur {\wheeled{\Cinf}} {R_\infty, 0}$, we obtain the commutative diagram
  \[
  \begin{tikzcd}
    \rHH_\QQ \bigl( \Cinf( C^*(G; R_\infty) ) \bigr) \rar \dar & \Schur {\wheeled{\Cinf}} {C^*(G; R_\infty), 0} \dar \\
    \Cochains * \bigl( G; \rHH_\QQ ( \Cinf(R_\infty) ) \bigr) \rar \dar & \Cochains * \bigl( G; \Schur {\wheeled{\Cinf}} {R_\infty, 0} \bigr) \dar \\
    \Cochains * \bigl( G; \rHH_\k ( R_\infty ) \bigr) \rar & \Cochains * (G; \k)
  \end{tikzcd}
  \]
  where the bottom horizontal map is trivial in $\Underlying(\Mod{\k})$ since $\Cochains * (G; \blank)$ preserves quasi-isomorphisms.
  In particular this map is trivial on homology.
  It is thus enough to prove that, for a cochain complex $V$, the image of the map induced on homology by $\rHH_\QQ ( \Cinf(V) ) \to \Schur {\wheeled{\Cinf}} {V, 0}$ contains $\Schur {\Coho {> 0} (\wheeled{\Cinf})} {\Coho * (V), 0}$.
  
  We first prove this for the analogous map from $\rHC_\QQ ( \Cinf(V) )$.
  To this end, we consider the commutative diagram of cochain complexes
  \[
  \begin{tikzcd}
    \HC_\QQ \bigl( \Cinf(V) \bigr) \ar{rr}{\eq} \dar[two heads] & &[-40] \Schur {\wheeled{\sCinf}} {\shift V, 0} \ar{rr}{\iso} &[-40] & \Schur {\wheeled{\Cinf}} {V, 0} \dar[two heads]{\pi} \\
    \Cinf(V) \rar{\eq} & \Com(V) \ar{rr}{\iso} & & \SA^+ V & \lar[swap]{\iso} \Schur {\wheeled{\Com}} {V, 0}
  \end{tikzcd}
  \]
  where the left-hand vertical map is projection onto Hochschild degree $0$ (see \cref{lemma:HH_proj}), and $\SA^+ V \defeq \bigoplus_{n \ge 1} (V^{\tensor n})_{\Sigma_n}$ denotes the free commutative non-unital $\QQ$-algebra on $V$.
  Since both vertical maps are surjective, we obtain a quasi-isomorphism $\rHC_\QQ ( \Cinf(V) ) \to \ker(\pi)$.
  Since $\wheeled{\Com}$ is concentrated in degree $0$, this implies the claim.
  
  We will now complete the proof by showing that the map $\rHH_\QQ ( \Cinf(V) ) \to \rHC_\QQ ( \Cinf(V) )$ is surjective on homology.
  Note that both $\rHH_\QQ ( \Cinf(V) )$ and $\rHC_\QQ ( \Cinf(V) )$ preserve quasi-isomorphisms in $V$ by (the proof of) \cref{lemma:HH_qiso}.
  Choosing a quasi-isomorphism $V \to \Coho * (V)$, we can thus assume that the differential of $V$ is trivial.
  Now, the strict quasi-isomorphism of $\Cinf$-algebras $\Cinf(V) \to \Com(V)$ induces the left-hand square of the following commutative diagram of cochain complexes
  \[
  \begin{tikzcd}
    \rHH_\QQ \bigl( \Cinf(V) \bigr) \rar{\eq} \dar & \rHH_\QQ \bigl( \Com(V) \bigr) \rar \dar & \rHH_\QQ (\SA V) \dar \\
    \rHC_\QQ \bigl( \Cinf(V) \bigr) \rar{\eq} & \rHC_\QQ \bigl( \Com(V) \bigr) \rar & \rHC_\QQ (\SA V)
  \end{tikzcd}
  \]
  where the right-hand square is induced by the inclusion $\Com(V) \iso \SA^+ V \subset \SA V$.
  The upper right-hand map is a quasi-isomorphism by \cref{lemma:nHH_qiso} since the map $\rnHH_\QQ(\SA^+ V) \to \rnHH_\QQ(\SA V)$ is an isomorphism.
  The canonical map $\rHC_\QQ(\SA^+ V) \oplus \rHC_\QQ(\QQ) \to \rHC_\QQ(\SA V)$ is a quasi-isomorphism by \cite[2.2.13 and 2.2.16]{Lod}\footnote{Note that the construction Loday denotes by $\ol{HC}$ is different from our $\rHC$.}.
  Using this, the results of \cite[§3.2]{Lod} imply that the map $\rHH_\QQ (\SA V) \to \rHC_\QQ (\SA V)$ is surjective on homology onto the image of $\rHC_\QQ(\SA^+ V)$.
  This completes the proof.
\end{proof}

\subsubsection*{Model-independent formulation}

To finish this section, we provide a reformulation of \cref{cor:graph_vanishing}, analogous to how Berglund \cite{Ber} phrases his results.

\begin{construction} \label{con:kappa}
  Let $p \colon E \to B$ be a map of connected animas with compact simply connected fiber.
  Assume that the universal covering of $B$ is of finite rational type.
  Then \cref{cor:eq_Cinf_existence} provides a group $G$ and a $G$-equivariant $\Cinf$-algebra $R_\infty$ over a $G$-equivariant cdga $\k$ such that $R_\infty$ is dualizable as a $\k$-module and such that $\k \to R_\infty$ models $p$.
  Using the quasi-isomorphisms $\Cochains{*}(G; R_\infty) \eq \Cochains{*}(E; \QQ)$ and $\Cochains{*}(G; \k) \eq \Cochains{*}(B; \QQ)$ (see e.g.\ \cite[Lemmas~2.8.6, 2.8.7, and 2.8.14]{BS24}), the map $\kappa_\wheel(R_\infty)$ of \cref{con:equv_eval} thus yields a map
  \[ \kappa_\wheel(p) \colon \Schur {\wheeled{\Cinf}} { \Cochains{*}(E; \QQ), 0}  \longto  \Cochains{*}(B; \QQ) \]
  in $\Underlying(\Mod{\QQ})$.
\end{construction}

\begin{remark} \label{rem:Berglund_kappa}
  Our construction of the map $\kappa_\wheel(p)$ is analogous to a construction of Berglund \cite{Ber} for maps $p \colon E \to B$ whose fiber is a Poincaré duality complex of dimension $d$ (and thus in particular compact).
  In that case, his map $\Schur[\k] {\mathbb{M}_{\mathrm{det}^d}(\Cinf)} {R_\infty} \to \k$ (see \cref{rem:Berglund_map}) induces a map
  \[ \kappa_d(p) \colon \Schur {\mathbb{M}_{\mathrm{det}^d}(\Cinf)} { \Cochains{*}(E; \QQ) }  \longto  \Cochains{*}(B; \QQ) \]
  which yields $\kappa_\wheel(p)$ when restricted along the map $\wheeled{\Cinf}(\blank, 0) \to \mathbb{M}_{\mathrm{det}^d}(\Cinf)$.
  In particular, using that the latter map is surjective on homology onto the loop-order-one part by \cref{rem:comparison_gc}, our construction generalizes the loop-order-one part of Berglund's to maps whose fiber is merely compact.
\end{remark}

\begin{remark}
  A priori, our map $\kappa_\wheel(p)$ depends on the choice of the model $R_\infty$.
  However, Berglund \cite{Ber} shows that his map $\kappa_d(p)$ is independent of this choice, and we expect that his arguments can be generalized to our situation.
\end{remark}

Our \cref{cor:graph_vanishing} implies that the map $\kappa_\wheel(p)$ is almost trivial when $p$ is a topological fiber bundle, as in the following corollary.
By \cref{rem:Berglund_kappa}, this implies the analogous statement for the loop-order-one part of Berglund's map $\kappa_d(p)$.

\begin{corollary} \label{cor:kappa_vanishing}
  Let $M$ be a compact simply connected topological manifold, and $p \colon E \to B$ the fiber bundle classified by a map $B \to \B \Homeo(M)$.
  Assume that $B$ is connected and that its universal covering is of finite rational type.
  Then the map induced on homology by the map $\kappa_\wheel(p)$ of \cref{con:kappa} vanishes on
  \[ \Schur {\Coho {> 0} (\wheeled{\Cinf})} {\Coho * (E; \QQ), 0}  \subseteq  \Schur {\Coho * (\wheeled{\Cinf})} {\Coho * (E; \QQ), 0}  \iso  \Coho * \bigl( \Schur {\wheeled{\Cinf}} {\Cochains{*}(E; \QQ), 0} \bigr) \]
  where $\Coho {> 0}$ denotes homology classes of positive degree.
\end{corollary}

\begin{proof}
  This follows from \cref{cor:graph_vanishing}.
\end{proof}